\documentclass[twoside,11pt]{article}

\usepackage{amsmath,amsthm}
\usepackage[preprint]{jmlr2e}  
\graphicspath{{./pics/}}
\usepackage{subcaption}
\usepackage{enumitem}
\usepackage{mathtools}
\usepackage{bm}
\usepackage{overpic}
\usepackage{booktabs}
\usepackage{makecell}
\usepackage[mathlines]{lineno}  
\newcommand{\BlackBoxProof}{\rule{1.5ex}{1.5ex}}
\renewenvironment{proof}{\par\noindent{\bf Proof\ }}{\hfill\BlackBoxProof\par\medskip}

\usepackage{multirow}

\usepackage[normalem]{ulem}           
\usepackage{xcolor}
\usepackage[textsize=tiny,textwidth=2.5cm]{todonotes}

\newcommand{\rev}[1]{#1}                           

\newtheorem{assumption}[theorem]{Assumption}

\newcommand{\nm}[1]{\left\lVert#1\right\rVert}

\newcommand{\ip}[1]{\left\langle #1 \right\rangle}
\newcommand{\R}{\mathbb{R}}

\renewcommand{\P}{\mathbf{P}}
\newcommand{\D}{\mathrm{D}}
\newcommand{\ud}{\mathrm{d}}
\newcommand{\Exp}{\mathrm{Exp}}
\newcommand{\dd}{\mathrm{d}}
\newcommand{\Hess}{\operatorname{Hess} }
\newcommand{\grad}{\operatorname{grad} }

\newcommand{\x}{{\bm x}}
\newcommand{\NN}{\mathrm{NN}}
\newcommand{\Lip}{\mathrm{Lip}}
\newcommand{\prox}{\mathrm{prox}}
\newcommand{\argmin}{\mathop{\mathrm{arg min}}}

\newcommand{\calF}{\mathcal{F}}
\newcommand{\calG}{\mathcal{G}}
\newcommand{\calH}{\mathcal{H}}
\newcommand{\calK}{\mathcal{K}}
\newcommand{\calM}{\mathcal{M}}
\newcommand{\calN}{\mathcal{N}}

\newcommand{\Log}{\mathrm{Log}}
\newcommand{\dist}{\mathrm{dist}}

\jmlrheading{}{2026}{1-??}{}{}{}{Zheng, Wang and Yang}
\firstpageno{1}

\title{Global Convergence and Error Propagation in Neural Gradient Flows: A Riemannian Optimization Framework.}

\author{\name Shixin Zheng \email siuzheng@umd.edu \\
       \addr Department of Mathematics, University of Maryland, College Park, MD
       \AND
       \name Yiwei Wang \email yiweiw@ucr.edu \\
       \addr Department of Mathematics, University of California-Riverside, Riverside, CA
       \AND
       \name Haizhao Yang \email hzyang@umd.edu \\
       \addr Department of Mathematics, University of Maryland, College Park, MD}

\editor{TBD}

\begin{document}

\maketitle

\begin{abstract}
We develop a geometric convergence theory for neural-network optimization within the minimizing movement scheme (MMS) framework. Reformulating each neural MMS step as a minimization over the set of increments in a Hilbert space, we show that under a $C^2$ network with locally non-degenerate Jacobian this increment set is a boundaryless smooth embedded submanifold, on which a natural preconditioned (Gauss--Newton-type) gradient flow in parameter space induces exactly the Riemannian gradient flow. Under a strict interior-localization condition and an explicit data condition, the reached sublevel set is geodesically convex and the subproblem objective is geodesically strongly convex on it; both the continuous Riemannian gradient flow and its discrete companion via the exponential map converge linearly to the unique subproblem minimizer. Propagating finite-time inner-solver inexactness and neural-approximation error through the MMS iterations yields a uniform function-space tracking bound and an explicit trajectory budget, so the inexact neural iterates converge to an $O(\delta)$-neighborhood of the global minimum. Numerical experiments on nonlinear regression and a small-scale latent-diffusion testbed indicate that the Gauss--Newton-type inner solver achieves smaller trajectory errors with substantially fewer inner iterations than first-order baselines.
\end{abstract}

\begin{keywords}
Riemannian optimization, neural network optimization, minimizing movement scheme, geodesic convexity, Gauss--Newton method, global convergence
\end{keywords}

\section{Introduction}
Many scientific computing and machine-learning problems can be formulated as minimizing an energy functional
\begin{equation}\label{eq:intro:main_problem}
\min_{u\in \calH} \calF[u],
\end{equation}
where $\calH$ is a real Hilbert space (e.g., $L^2(\Omega)$) and $\calF:\calH\to\mathbb{R}$ is continuously Fr\'echet differentiable, coercive, and nonnegative (the last condition is without loss of generality, since adding a constant does not change the minimizer or the gradient flow). When $\calF$ is convex, the associated continuous-time gradient flow
\begin{equation}\label{eq:GD}
\partial_t u_t=-\nabla \calF[u_t]
\end{equation}
provides a canonical dynamics that converges to a global minimizer under certain conditions.
A classical time discretization of the gradient flow (\ref{eq:GD}) is the implicit Euler method,
which admits the variational characterization known as the \emph{minimizing movement scheme} (MMS) \citep{ambrosio2005gradient}.
This paper studies a neural-network based spatial discretization of MMS, which is numerically investigated in \citet{HuEnergetic2024}, and provides a geometric convergence theory by viewing each neural MMS step as a Riemannian optimization problem on a Hilbert-embedded manifold of \emph{increments}.

\subsection{Minimizing Movement Scheme (MMS)}
Given stepsize $\tau>0$ and an iterate $u^n\in \calH$, the MMS solves the optimization problem
\begin{equation}\label{eq:intro:mms}
u^{n+1} = J_\tau(u^n),
\qquad
J_\tau(x):=\prox_{\tau \calF}(x):=\argmin_{u\in \calH}\left\{\frac{1}{2\tau}\|u-x\|_\calH^2 + \calF[u]\right\}.
\end{equation}
at each time step. Here, $u^n$ is an approximation to the gradient flow solution $u$ of (\ref{eq:GD}) at $t = n \tau$.
Throughout, we assume $\calF$ is $m_\calF$-strongly convex with $m_\calF>0$, so the objective in \eqref{eq:intro:mms} is strongly convex and the resolvent operator $J_\tau$ is well-defined and single-valued for any $\tau>0$. As the time step $\tau\downarrow 0$, the continuous interpolants of the discrete MMS sequence converge uniformly to the unique continuous trajectory of the associated gradient flow, which itself converges exponentially to the unique global minimizer $u^*$ of \eqref{eq:intro:main_problem} as time tends to infinity.
Therefore, since the exact discrete iterate $u^N$ approximates the gradient flow state at time $T=N\tau$, it can be regarded as an approximation of $u^*$ for sufficiently small $\tau$ and sufficiently large $N$.
For completeness, we detail the rigorous convergence of this scheme in Appendix~\ref{sec:conv-mms}.


\subsection{Neural-network Parametrization of MMS (Neural MMS)}
\rev{%
To solve the MMS subproblem~\eqref{eq:intro:mms} in practice,
one must restrict from the infinite-dimensional space $\calH$ to a finite-dimensional model class.
Classical spatial discretization, such as finite elements, spectral methods, or finite differences, introduce
a mesh that scales exponentially with the spatial dimension $d$, making them impractical for
high-dimensional problems such as those arising in molecular dynamics, optimal control, or
mean-field games. A natural alternative is to use neural networks as a mesh-free parametrization for the spatial variable \citep{park2023deep,HuEnergetic2024}.
Neural networks offer a compelling alternative for several reasons.
First, they are \emph{mesh-free}: classical discretizations such as finite elements tie the
function representation to a fixed set of grid points or quadrature nodes, so the two cannot
be varied independently. By contrast, the parametrization $w\mapsto u_{\NN}(\cdot\,;w)$
defines a function over the entire domain that can be evaluated at \emph{any} point,
decoupling the function representation from the choice of quadrature or collocation points.  This flexibility is
particularly attractive in high dimensions, where tensor-product grid-based
discretizations suffer from an exponential growth in degrees of freedom.
Second, deep networks possess a powerful \emph{compositional approximation} structure:
they can represent functions with hierarchical features using a number of parameters
that scales polynomially (rather than exponentially) with dimension for suitable function
classes \citep{JMLR:v23:21-1404,maiti2024optimal}.
Third, the parametrization is \emph{adaptive}: unlike fixed basis functions, the network
automatically adjusts its internal representation to the solution structure during training.
Finally, and crucially for this paper, the MMS proximal structure provides a natural
\emph{regularization} of the nonconvex neural optimization landscape, as the quadratic penalty
$\|u_{\NN}(w)-u^n_{\NN}\|_{\calH}^2/(2\tau)$ forces the next iterate to stay close to the
previous one, which is precisely the locality that enables our Riemannian convergence theory.
}

Specifically, we restrict the optimization domain from $\calH$ to the model class $\NN\subset \calH$, the image of a smooth parametrization $u_\NN:\R^p\to\calH$. Unlike Galerkin methods, the class $\NN$ is not a vector subspace of $\calH$ (the sum of two same-architecture networks need not be representable in that architecture), but it is the image of a finite-dimensional parameter space under the parametrization map. 
This distinction is important for the analysis. In a classical Galerkin
discretization, the trial space is a finite-dimensional linear subspace of
$\calH$ and therefore inherits the Hilbert-space geometry. Hence, if $\calF$ is
strongly convex on $\calH$, its restriction to the Galerkin space remains
strongly convex, and the discrete MMS subproblem is still a convex optimization
problem. Moreover, the trial-space and quadrature errors can be analyzed by
standard approximation and perturbation arguments. By contrast, the neural model
class is a nonlinear parametrized set rather than a vector space, so the
Hilbert-space convexity of $\calF$ does not directly translate into convexity
with respect to the parameters. This loss of linear structure is the main reason
we reformulate the neural MMS step geometrically in function space.

In the theoretical part of this work, the neural network is used purely as a \emph{parametrization} of functions in a Hilbert space
$\mathcal H$ (e.g.\ $L^2(\Omega)$ or $H^1(\Omega)$). The spatial variable $x\in\Omega\subset\mathbb R^d$ is a
\emph{continuous} argument, and the analysis is carried out at the function-space level.
Concretely, for each parameter $w\in\mathbb R^p$, the network defines a function
$x\mapsto u_{\NN}(x;w)$ and we identify $u_{\NN}(w)$ with an element of $\mathcal H$.
When sampling or quadrature is used in computation, it serves only to approximate inner products and integrals in $\mathcal H$ and does not change the analytical (infinite-dimensional) formulation.

Let $u_{\NN}:\mathbb{R}^p\to \NN\subset \calH$ denote the network realization from parameter space where $p$ is the number of parameters, and write $u^n_{\NN}:=u_{\NN}(w^n)$. The neural MMS step solves
\begin{equation}\label{eq:intro:evnn_param}
w^{n+1} = \argmin_{w\in\mathbb{R}^p}\left\{\frac{1}{2\tau}\|u_{\NN}(w)-u_{\NN}(w^n)\|_\calH^2 + \calF[u_{\NN}(w)]\right\},
\end{equation}
While MMS enjoys a clean convex-analytic theory in $\calH$, the neural parameterization typically makes \eqref{eq:intro:evnn_param} nonconvex in $w$, thus the global convergence of gradient-based optimization methods for \eqref{eq:intro:evnn_param} is difficult to achieve. 

It is natural to reformulate (\ref{eq:intro:evnn_param}) in function space by defining the fixed-architecture neural model class $\mathcal{N}:=\{ u_{\NN}(w):w \in \mathbb{R}^p \}$ (note that $\mathcal{N}$ is not a subspace of $\mathcal{H}$, since the sum of two neural networks with the same architecture need not be representable in that architecture). In the function space, we consider
\begin{equation}\label{eq:intro:evnn_function}
u^{n+1}_{\NN} = J_{\tau,\NN}(u^n_{\NN}),
\qquad
J_{\tau,\NN}(x):=\argmin_{u_{\NN}\in \mathcal{N}}\left\{\frac{1}{2\tau}\|u_{\NN}-x\|_\calH^2 + \calF[u_{\NN}]\right\}.
\end{equation}
Well-posedness of $J_{\tau,\NN}$ -- existence and uniqueness of the minimizer over $\calN$ -- follows from Assumption~\ref{ass:intro:locality} together with Theorem~\ref{thm:RGF-conv} via the identification~\eqref{eq:J-tau-NN-well-defined}. Our goal is to recover a global convergence theory by changing viewpoints: we do not optimize over the entire parameter space at each step; instead, we optimize over a local set of \emph{increments} in $\calH$ that forms a Riemannian submanifold.

\subsection{Equivalent Formulation via Increment Optimization}
For $v\in\calH$, define the auxiliary increment functional $g(\cdot;v):\calH\to\mathbb{R}$ by
\begin{equation}\label{eq:intro:g_def}
g(h;v) := \frac{\|h\|_\calH^2}{2\tau}+\calF\bigl[\,h+v\,\bigr].
\end{equation}
Then the MMS update \eqref{eq:intro:mms} is equivalent to
\begin{equation*}
h^n = \argmin_{h\in \calH} g(h;u^n),
\qquad
u^{n+1} = u^n + h^n,
\end{equation*}
so that $h^n = u^{n+1}-u^n$ is the optimal increment. Likewise, the neural MMS update \eqref{eq:intro:evnn_function} can be (formally) written as an increment minimization:
\begin{equation}\label{eq:intro:evnn_increment}
h^n \in \argmin_{h\in \{u_{\NN}(w)-u^n_{\NN}:\,w\in\mathbb{R}^p\}} g(h;u^n_{\NN}),
\qquad
u^{n+1}_{\NN}=u^n_{\NN}+h^n.
\end{equation}
This reformulation separates the \emph{functional} geometry in $\calH$ from the \emph{parameter} nonconvexity, and it is the starting point of our geometric analysis.

\subsection{Increment Manifold and Riemannian-Optimization Viewpoint}
Because the proximal term in \eqref{eq:intro:g_def} penalizes large increments (when $\tau$ is small), the optimizer of \eqref{eq:intro:evnn_increment} is expected to be ``local'' around the previous parameters $w^n$ under certain locality assumption. We formalize this through the following assumption.

\begin{assumption}[Locality of the neural increment]\label{ass:intro:locality}
At each step $n\ge 0$, there exists a radius $r_{w^n}>0$ (possibly depending on $w^n$) such that the increment problem~\eqref{eq:intro:evnn_increment} admits at least one global minimizer that lies in the open neighborhood of $w^n$:
\[
\argmin_{h\in \{u_{\NN}(w)-u^n_{\NN}:\,w\in\mathbb{R}^p\}} g(h;u^n_{\NN})
\ \cap\
\Big\{u_{\NN}(w)-u^n_{\NN}:\,\|w-w^n\|_2 < r_{w^n}\Big\}\ \neq\ \emptyset.
\]
\end{assumption}

Assumption~\ref{ass:intro:locality} formalizes a \emph{locality of global minimizers}:
although the increment problem is posed over the full parameter space,
at least one global minimizer is required to lie near $u_{\mathrm{NN}}(w^n)$.
This is enough to ensure that the local Riemannian gradient flow on $\calM(u_{\mathrm{NN}}(w^n))$
attains the same global minimum value as the unconstrained problem (see Section~\ref{sec:conv}).

Although Assumption~\ref{ass:intro:locality} has the appearance of a circular hypothesis (``the global minimum is local''), it is a verifiable consequence of two standard ingredients: an output-isolation margin $\rho^{\rm in}_{w^n}>0$ around $u_\NN(w^n)$ outside the inner ball $\{w : \|w-w^n\|_2\le r_{w^n}/2\}$, and a sufficiently small step $\tau$ satisfying $2\tau(\calF[u_\NN(w^n)]-\calF_{\inf}) < (\rho^{\rm in}_{w^n})^2$. Under this condition, the strict interior-localization estimate (Lemma~\ref{lem:interior-loc}) places the entire decreasing sublevel set of the increment objective inside the inner ball, which implies Assumption~\ref{ass:intro:locality} and ensures the Riemannian gradient flow analyzed in Section~\ref{sec:conv} never approaches $\partial\calM(u_\NN(w^n))$. Output isolation $\rho^{\rm in}_{w^n}>0$ is a local form of injectivity of $u_\NN$, plausible away from the symmetry orbits of the network architecture; we elaborate in Appendix~\ref{app:suff-locality}, where Theorem~\ref{thm:small-tau-locality} establishes a less conservative variant via the original outer-ball output-isolation radius.

Intuitively, the proximal parameter $\tau$ enforces locality
in function space. Indeed, by the property of MMS, we have
\[
\|u^{n+1}_{\NN}-u^n_{\NN}\|_{\calH}^2
\le
2 \tau (\calF[u^n_{\NN}]-\calF[u^{n+1}_{\NN}])
\le
2 \tau (\calF[u^n_{\NN}]-\calF_{\inf}).
\]
Thus the neural increment is small in the $\calH$-norm when $\tau$ is small;
moreover, this localization becomes stronger in the late-stage regime where
the energy gap $\calF[u^n_{\NN}]-\calF_{\inf}$ is small. 
The additional non-degeneracy and output-isolation conditions serve precisely
to convert this function-space localization into a parameter-space localization
near $w^n$.

Under Assumption~\ref{ass:intro:locality}, define the \emph{increment manifold}
\begin{equation}\label{eq:intro:M_def}
\calM(u^n_{\NN}) := \Big\{u_{\NN}(w)-u^n_{\NN}:\,\|w-w^n\|_2 < r_{w^n}\Big\}\subset \calH,
\end{equation}
together with its compact \emph{localization image}
\begin{equation}\label{eq:intro:K_def}
\calK(u^n_{\NN}) := \Big\{u_{\NN}(w)-u^n_{\NN}:\,\|w-w^n\|_2 \le r_{w^n}/2\Big\}\subset \calM(u^n_{\NN}).
\end{equation}
Then, the neural MMS step becomes
\begin{equation}\label{eq:intro:evnn_increment_manifold}
h^n = \argmin_{h\in \calM(u^n_{\NN})} g(h;u^n_{\NN}),
\qquad
u^{n+1}_{\NN}=u^n_{\NN}+h^n.
\end{equation}
In Section~\ref{sec:embedded-manifold}, under a local non-degeneracy condition on the Jacobian of $u_\NN$, we show that $\calM(u_\NN(\tilde w))$ is a smooth embedded Riemannian submanifold of $\calH$ and identify its tangent spaces and orthogonal projectors; this turns~\eqref{eq:intro:evnn_increment_manifold} into a Riemannian optimization problem in function space. The image $\calK(u^n_{\NN})$ is used in §\ref{sec:conv} only as a compact localization device for the Riemannian gradient flow trajectory.

\subsection{Main Results}\label{sec:intro:main-results}
Our main results are as follows.
\begin{itemize}
  \item \textbf{Geometry of the increment set (Proposition~\ref{prop:manifold}).}
  Under a $C^2$ network with $L$-Lipschitz Jacobian and a locally non-degenerate Jacobian, the increment manifold $\calM(u_\NN(\tilde w))$ --- the image of an \emph{open} parameter ball $B(\tilde w,r_{\tilde w})$ under the parametrization --- is a $p$-dimensional boundaryless smooth embedded submanifold of $\calH$. The image $\calK(u_\NN(\tilde w))$ of the corresponding \emph{closed} inner ball $\{\|w-\tilde w\|_2\le r_{\tilde w}/2\}$ is a compact subset of $\calM(u_\NN(\tilde w))$ and serves as the working set of the analysis.

  \item \textbf{Riemannian-gradient-flow interpretation (Section~\ref{sec:RGF}).}
  A natural preconditioned (Gauss--Newton-type) gradient flow in parameter space induces exactly the Riemannian gradient flow of $g(\cdot;v)$ on $\calM(u_\NN(\tilde w))$ in $\calH$, connecting Gauss--Newton preconditioning to intrinsic function-space dynamics.

  \item \textbf{Continuous and discrete subproblem convergence (Theorems~\ref{thm:RGF-conv},~\ref{thm:discrete-RGD}).}
  Under a strict interior-localization condition (Lemma~\ref{lem:interior-loc}: small step $\tau$ relative to a local output-isolation margin), the decreasing-energy sublevel set $S$ reached by the Riemannian gradient flow is contained in $\calK(u_\NN(\tilde w))$ and lies inside a strongly geodesically convex normal neighborhood $\calG$ of the origin (Lemma~\ref{lem:convex-radius}). On this invariant sublevel set we establish geodesic strong convexity and smoothness of $g(\cdot;v)$ under an explicit \emph{data condition} on $(v,\tau,\calF,L,\lambda_{\tilde w})$ (Theorem~\ref{thm:g}), giving exponential convergence of the Riemannian gradient flow and a matching linear-rate guarantee for the discrete Riemannian gradient descent that takes its steps along the geodesics (i.e., via the exponential map).

  \item \textbf{Global error propagation for inexact neural MMS (Theorems~\ref{thm:global_error},~\ref{thm:global_error-RGD}).}
  We quantify how (i) finite-time inexactness of the inner solver --- either the continuous Riemannian gradient flow run for time $T$ or the discrete Riemannian gradient descent run for $K$ steps --- and (ii) neural approximation error $\varepsilon$ in $J_{\tau,\NN}$ propagate through MMS iterations. Under a step-size condition, the inexact iterates track the exact MMS iterates uniformly and converge to an $O(\delta)$-neighborhood of the true minimizer when $m_\calF>0$.

  \item \textbf{Parameter-space convergence (Theorem~\ref{thm:global-PGF}).}
  The function-space tracking bound transfers to a uniform tracking bound on the parameter iterates themselves, together with an explicit trajectory-length budget that confines the entire sequence $\{w^n_{\NN,T}\}$ to a bounded region of $\R^p$.
\end{itemize}

\rev{%
Figure~\ref{fig:hierarchy} summarizes the conceptual hierarchy from the continuous gradient flow in $\calH$ down to the equivalent algorithmic views (Riemannian gradient flow on $\calM$ and preconditioned gradient descent in $\mathbb{R}^p$), highlighting where each layer of discretization enters.
Figure~\ref{fig:function-space} illustrates the global error propagation: the exact MMS iterates $\{u^n\}$ in $\calH$ converge to the minimizer $u^*$, while the neural MMS iterates $\{u^n_{\mathrm{NN}}\}$, constrained to the model class $\calN$, track the exact iterates within a uniform distance $\delta$ at every step.
}

\noindent\textbf{Scope of this paper.}
We distinguish four sources of error in a practical neural MMS computation.
Let $u(t_n)$ denote the exact gradient-flow solution at time $t_n = n\tau$,
$u^n := J_\tau^n(u^0)$ the exact MMS iterate in $\calH$,
and $u^n_{\NN,T}\in\calN$ the neural MMS iterate produced by running the
inner Riemannian solver for time~$T$ at each step.
Formally,
\begin{equation}\label{eq:intro:full-error-decomp}
u(t_n)-u^n_{\NN,T}
=
\underbrace{\bigl[u(t_n)-u^n\bigr]}_{\text{(I) time discr.}}
+
\underbrace{\bigl[u^n-\tilde u^n\bigr]}_{\text{(II) trial space}}
+
\underbrace{\bigl[\tilde u^n-\hat u^n\bigr]}_{\text{(III) sampling}}
+
\underbrace{\bigl[\hat u^n-u^n_{\NN,T}\bigr]}_{\text{(IV) inner opt.}}.
\end{equation}
Here, $\tilde u^n$ and $\hat u^n$ are conceptual intermediate iterates
corresponding respectively to the population and empirical neural MMS problems, (I) is the classical time-discretization error of MMS;
(II) arises from restricting to the neural model class~$\calN$;
(III) is the quadrature or sampling error from replacing population quantities
by $N$-sample empirical approximations;
and (IV) is the inner-optimization residual from stopping the Riemannian
solver at finite time~$T$.
This paper focuses on term~\textbf{(IV)}.
Terms~(II) and~(III) are not analyzed separately; motivated by universal
approximation theory and Monte Carlo estimates, their combined effect is
absorbed into the one-step precision~$\varepsilon$ of
Assumption~\ref{assm:approx-precision}, leaving a detailed
architecture- and sample-size-dependent analysis for future work.
Under Assumption~\ref{assm:approx-precision}, Theorem~\ref{thm:global_error}
shows that the inner-optimization error~(IV) decays as $O(e^{-\nu T})$ and
propagates stably through the MMS iterations.

\begin{figure}[t]
\centering
\includegraphics[width=0.85\textwidth]{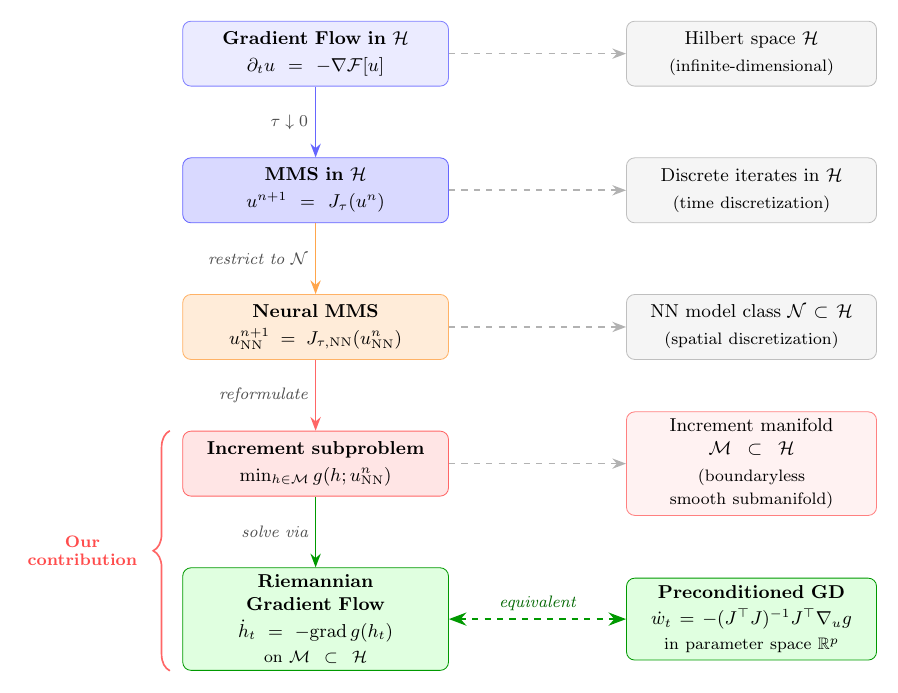}
\caption{Conceptual hierarchy of the framework. Starting from the continuous gradient flow in $\calH$, time discretization yields the MMS, spatial discretization via neural networks gives the neural MMS, and reformulation via increments leads to a Riemannian optimization problem on $\calM\subset\calH$. The Riemannian gradient flow on $\calM$ is equivalent to the preconditioned (Gauss--Newton-type) gradient descent in parameter space $\mathbb{R}^p$.}\label{fig:hierarchy}
\end{figure}

\begin{figure}[t]
\centering
\includegraphics[width=0.75\textwidth]{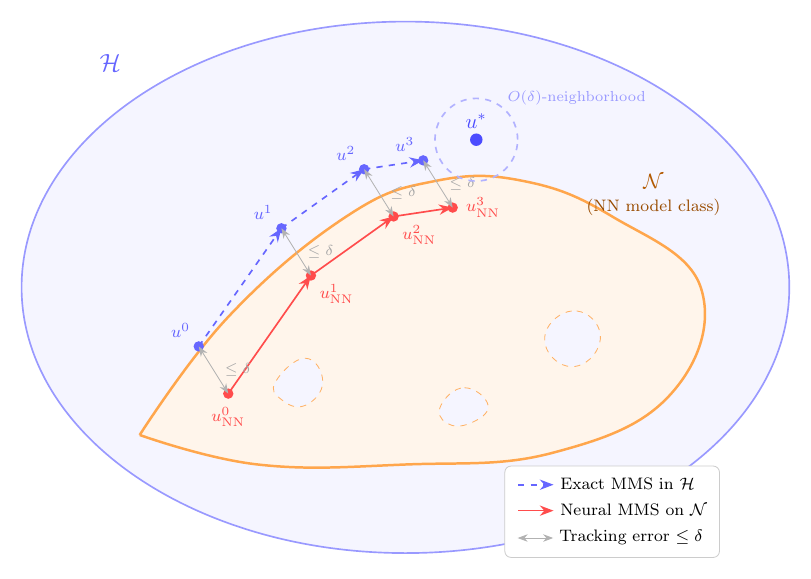}
\caption{Function-space view of the global error propagation. The exact MMS iterates $\{u^n\}$ (blue, dashed) live in the full Hilbert space $\calH$ and converge to the global minimizer $u^*$. The neural MMS iterates $\{u^n_{\mathrm{NN}}\}$ (red, solid) are constrained to the neural network model class $\calN\subset\calH$; the dashed pockets inside $\calN$ indicate that $\calN$ is not a vector subspace (it is not closed under addition or scalar multiplication). Theorem~\ref{thm:global_error} guarantees that $\|u^n_{\mathrm{NN}}-u^n\|_{\calH}\le\delta$ uniformly for all $n$, so the neural iterates converge to an $O(\delta)$-neighborhood of $u^*$.}\label{fig:function-space}
\end{figure}

\subsection{Related Work}
\rev{%
A central question in neural network optimization is whether gradient-based methods converge to global minimizers, and if so, at what rate and under what assumptions on the architecture.
We organize the related literature along three axes that clarify how our result fits into the broader landscape: (i) NTK and over-parameterized convergence, (ii) mean-field and feature-learning approaches, and (iii) second-order and Riemannian geometric methods.
We then discuss related variational and gradient-flow frameworks.
Table~\ref{tab:convergence-comparison} provides a concise comparison.

\paragraph{(i) NTK regime and over-parameterization.}
The neural tangent kernel (NTK) framework of \citet{jacot2018neural} shows that in the infinite-width limit the kernel stays constant during gradient descent, reducing training to a linear ODE and yielding global convergence at a linear rate governed by $\lambda_{\min}(\text{NTK})$.
Subsequent works extended this to finite but large width:
\citet{du2018gradient} proved global convergence for deep over-parameterized ResNets with width $\mathrm{poly}(n)$;
\citet{allen2018convergence} established linear convergence of SGD for FC, CNN, and ResNet architectures with width $\mathrm{poly}(n,L)$;
\citet{zou2019gradient} obtained linear-rate convergence for deep ReLU networks;
and \citet{chen2019much} sharpened the width requirement to $\mathrm{polylog}(n,1/\varepsilon)$ under a data-margin assumption.
For shallow networks, \citet{oymak2019moderate} achieved moderate over-parameterization ($\sqrt{p}>n$).
A key limitation was identified by \citet{hanin2019finite}: the NTK variance scales as $\exp(d/n_{\text{width}})$, so the frozen-kernel approximation breaks down when depth is non-negligible relative to width.
All of these results require width at least polynomial in~$n$, and their convergence rates depend on $\lambda_{\min}(\text{NTK})$, which degrades with depth.
Our framework is complementary: it applies to \emph{fixed, finite-width} architectures and obtains convergence through the Riemannian geometry of the increment manifold, with a rate independent of NTK conditioning.

\paragraph{(ii) Mean-field and feature learning.}
The mean-field perspective \citep{mei2018mean,mei2019mean} treats neurons as particles whose distribution evolves under a Wasserstein gradient flow, yielding global convergence for two-layer networks that---unlike the NTK regime---permits feature learning.
Extensions to three-layer \citep{pham2021global,nguyen2020rigorous} and deep ResNet \citep{lu2020mean} architectures exist, and
\citet{chen2022feature} obtained linear-rate convergence with feature-learning guarantees.
However, all these results require width $\to\infty$ (mean-field limit) and typically characterize limiting dynamics rather than providing non-asymptotic finite-width rates.
Our result operates at fixed finite width with explicit exponential convergence guarantees.

\paragraph{(iii) Second-order and Riemannian geometric methods.}
The closest predecessor to our work is \citet{cayci2025riemannian}, who showed that the Gauss--Newton method induces a Riemannian gradient flow on a low-dimensional submanifold of the function space, achieving geometric convergence at a rate \emph{independent of NTK conditioning} in both underparameterized and overparameterized regimes.
Our work extends this perspective from finite-dimensional Euclidean settings to a continuous Hilbert space~$\calH$ and crucially operates on a \emph{boundaryless increment manifold} (the image of an open parameter ball) within the MMS framework, with a strict interior-localization device keeping the analysis confined to a compact subset of the open chart.
In addition, we provide a \emph{global multi-step error propagation} analysis (Theorem~\ref{thm:global_error}), showing that the inexact neural MMS iterates converge to an $O(\delta)$-neighborhood of the true minimizer---a feature absent from single-subproblem analyses.
\citet{cai2019gram} achieved quadratic convergence for a Gram--Gauss--Newton method but requires sufficient over-parameterization and square loss.
The natural gradient method \citep{amari1998natural,zhang2019fast} preconditions using the Fisher information metric on the \emph{parameter space};
in contrast, our Gauss--Newton preconditioner arises from the \emph{function-space geometry} of the increment manifold embedded in~$\calH$.

\paragraph{Variational and gradient-flow frameworks.}
Neural discretizations of gradient flows via energetic variational principles were introduced by \citet{HuEnergetic2024}.
The MMS has deep connections to the JKO scheme \citep{jordan1998variational} in optimal transport.
The broader Riemannian optimization literature \citep{absil2008optimization,boumal2023intromanifolds} provides algorithms and convergence theory on classical matrix manifolds (Stiefel, Grassmann); our increment manifold $\calM$ differs in that it is defined implicitly through a neural parameterization as the image of an open parameter ball, so it is a \emph{boundaryless} embedded submanifold of $\calH$ whose geometry depends on the architecture and the current parameters.
The theory of gradient flows in metric spaces \citep{ambrosio2005gradient} and the KLS framework \citep{hauer2019kurdyka,attouch2013convergence} give convergence guarantees under tame geometry.
Our analysis uses a different mechanism: we establish geodesic strong convexity on the increment manifold, providing stronger (exponential) convergence rates.

\begin{table}[t]
\caption{Comparison of convergence guarantees for neural network optimization. Here $N$ denotes the number of training samples, $p$ the number of parameters, $H$ (or $L$) the depth, $\lambda_{\min}$ the minimum eigenvalue of the NTK Gram matrix, and $\nu$ the geodesic strong-convexity parameter of the increment subproblem. For our method, $\nu=(1+\tau m_\calF)/\tau$, and the Jacobian non-degeneracy constant $\lambda_{\tilde w}$ (the minimum singular value of the Jacobian) enters the theory through the data condition $C(v)=(4L\kappa/\lambda_{\tilde w}^{2})\max\{\|h^*(v)\|,K(v)\}\le 1$ of Theorem~\ref{thm:g}, where $\kappa=(1+\tau L_\calF)/(1+\tau m_\calF)$, $h^*(v)=J_\tau(v)-v$ is the exact-MMS increment, and $K(v)=\sqrt{2\tau\calF[v]/(1+\tau m_\calF)}$.}\label{tab:convergence-comparison}
\centering
\footnotesize
\setlength{\tabcolsep}{3pt}
\renewcommand{\arraystretch}{1.35}
\begin{tabular}{@{}llll@{}}
\toprule
\makecell[l]{\textbf{Reference}\\\textbf{\scriptsize(Algo.\ / Network)}} & \textbf{Width} & \textbf{Conv.\ rate} & \textbf{Remark} \\
\midrule
\multicolumn{4}{@{}l}{\emph{NTK / over-parameterization}} \\[2pt]
\makecell[l]{\citet{jacot2018neural}\\{\scriptsize GD / FC}} & $\to\infty$ & $e^{-\lambda_{\min} t}$ & Frozen-kernel limit \\[2pt]
\makecell[l]{\citet{du2018gradient}\\{\scriptsize GD / FC, ResNet}} & \makecell[l]{$\Omega(N^4 \cdot 2^{O(H)})$ (FC)\\$\Omega(N^4 H^6)$ (ResNet)} & $(1\!-\!\eta\lambda_{\min}/2)^k$ & Exp.\ in $H$ for FC \\[2pt]
\makecell[l]{\citet{allen2018convergence}\\{\scriptsize SGD / FC, CNN, ResNet}} & $\mathrm{poly}(N,L)$ & $(1\!-\!\eta\lambda_{\min}/N)^k$ & \makecell[l]{Unoptimized poly;\\no exp.\ in $L$} \\[2pt]
\makecell[l]{\citet{zou2019gradient}\\{\scriptsize GD / Deep ReLU}} & $\mathrm{poly}(N)$ & $(1\!-\!\eta\lambda_{\min}/N)^k$ & \\[2pt]
\makecell[l]{\citet{chen2019much}\\{\scriptsize SGD / Deep ReLU}} & $\mathrm{polylog}(N)$ & $(1\!-\!\eta\lambda_{\min}/N)^k$ & Data-margin assumption \\[2pt]
\makecell[l]{\citet{oymak2019moderate}\\{\scriptsize GD, SGD / 2-layer}} & $\sqrt{p}\!>\!N$ & $(1\!-\!\eta\lambda_{\min}/N)^k$ & Moderate over-param. \\
\midrule
\multicolumn{4}{@{}l}{\emph{Mean-field / feature learning}} \\[2pt]
\makecell[l]{\citet{mei2018mean}\\{\scriptsize SGD / 2-layer}} & $\to\infty$ & \makecell[l]{$t\!\to\!\infty$ (PDE);\\ $O(1/p)$ approx.} & \makecell[l]{Mean-field limit;\\no finite-width rate} \\[2pt]
\makecell[l]{\citet{chen2022feature}\\{\scriptsize GF / Multi-layer}} & $O(\log N)$ & $e^{-ct}$ & Feature learning \\
\midrule
\multicolumn{4}{@{}l}{\emph{Second-order / Riemannian}} \\[2pt]
\makecell[l]{\citet{cayci2025riemannian}\\{\scriptsize GN / FC (smooth)}} & \makecell[l]{Non-degenerate $J$} & $e^{-\nu t}$ & \makecell[l]{Finite-dim.\ $\R^d$;\\single subproblem} \\[2pt]
\makecell[l]{\citet{cai2019gram}\\{\scriptsize Gram-GN / Deep NN}} & $\Omega(N^4/\lambda_{\min}^4)$ & $(C/\sqrt{p})\|e_k\|^2$ & \makecell[l]{Quadratic; requires\\Gram matrix inverse} \\
\midrule
\makecell[l]{\textbf{Ours}\\{\scriptsize\textbf{GN-type / FC}}} & \textbf{Fixed, finite} & $\bm{e^{-\nu t}}$ & \makecell[l]{data condition\\ $C(v)\le 1$; see caption} \\
\bottomrule
\end{tabular}
\end{table}
}

\subsection{Notations and Organization of This Paper}

Throughout, $\calH$ denotes a real Hilbert space with inner product $\langle\cdot,\cdot\rangle_\calH$ and norm $\|\cdot\|_\calH$.
We use $\tau>0$ for the MMS step size.
For a neural network $u_{\NN}:\mathbb{R}^p\to \calH$, $\D u_{\NN}(w):\mathbb{R}^p\to \calH$ denotes its Fr\'echet derivative and $\D u_{\NN}(w)^\ast:\calH\to\mathbb{R}^p$ its adjoint.
The increment set  $\calM(u_{\NN}(\tilde w))$ is defined in \eqref{eq:intro:M_def}.
Additional regularity (e.g., Lipschitzness of $\D u_{\NN}$ and local non-degeneracy) will be stated precisely in Section~\ref{sec:embedded-manifold}, while convexity assumptions on $\calF$ needed for MMS stability are summarized in Section~\ref{sec:error-propagation} and Appendix~\ref{sec:conv-mms}.

\rev{%
\begin{table}[h]
\centering
\caption{Summary of principal notation.}
\label{tab:notation}
\begin{tabular}{cl}
\hline
Symbol & Description \\
\hline
$\calH$ & Ambient real Hilbert space (e.g., $L^2(\Omega)$) \\
$\calF$ & Energy functional to minimize \\
$\tau$ & MMS (proximal) step size \\
$u_{\NN}(w)$ & Neural network realization map, $\R^p\to\calH$ \\
$\D u_{\NN}(w)$ & Fr\'echet derivative (Jacobian) of $u_{\NN}$ at $w$ \\
$\calM(v)$ & Boundaryless increment manifold (open chart) centered at $v\in\calH$; see \eqref{eq:intro:M_def} \\
$\calK(v)$ & Compact inner localization image of $\calM(v)$; see \eqref{eq:intro:K_def} \\
$g(h;v)$ & Increment subproblem objective; see \eqref{eq:intro:g_def} \\
$\P_h$ & Orthogonal projector onto $T_h\calM$; see \eqref{eq:Projector} \\
$J_\tau$ & Proximal map of $\calF$ with step $\tau$; see \eqref{eq:prox_map_R} \\
$\lambda_{\tilde w}$ & Non-degeneracy parameter of the Jacobian at $\tilde w$ \\
$L$ & Lipschitz constant of $\D u_{\NN}$ \\
$\nu,\mu$ & Strong convexity/smoothness constants of $g(\cdot;v)$ \\
$\Lambda$ & Second-fundamental-form bound $2L/\lambda_{\tilde w}^2$ for the embedded manifold; see Lemma~\ref{lem:DP_bound} and~\eqref{eq:LambdaBound} \\
$\calG(v)$ & Strongly geodesically convex normal neighborhood of $0$ in $\calM(v)$; see Lemma~\ref{lem:convex-radius} \\
\hline
\end{tabular}
\end{table}
}

The remainder of this paper is organized as follows. Section~\ref{sec:RGF} establishes the Riemannian manifold structure of the increment set and shows that a preconditioned gradient flow in parameter space is equivalent to the Riemannian gradient flow on this manifold. Section~\ref{sec:conv} proves geodesic convexity and strong convexity of the subproblem objective under the interior-localization condition, yielding exponential convergence of the Riemannian gradient flow (Theorem~\ref{thm:RGF-conv}) and a matching linear-rate bound for the discrete Riemannian gradient descent via the exponential map (Theorem~\ref{thm:discrete-RGD}). Section~\ref{sec:error-propagation} propagates finite-time and approximation errors through the MMS iterations, establishing global convergence to an $O(\delta)$-neighborhood in function space for both the continuous and the discrete inner solvers. Section~\ref{sec:param-space} lifts this to a parameter-space convergence result with an explicit trajectory budget. Section~\ref{sec:experiments} presents numerical experiments illustrating the algorithmic implications of the Riemannian interpretation; the results suggest that the Gauss--Newton-type inner solver can provide higher-quality inner solves than Adam and L-BFGS in the tested examples, while exhibiting a competitive, though problem-dependent, computational cost. Section~\ref{sec:conclusions} summarizes our contributions and discusses limitations and future directions. Appendix~\ref{sec:conv-mms} reviews the classical MMS convergence theory, Appendix~\ref{app:suff-locality} gives a sufficient condition for Assumption~\ref{ass:intro:locality}, Appendix~\ref{app:deferred-proofs} collects deferred technical proofs, and the remaining appendices discuss the standing assumptions in detail.

\section{From Preconditioned Gradient Flow to Riemannian Gradient Flow}\label{sec:RGF}
\rev{%
The goal of this section is twofold: (i) to show that the increment set $\calM(u_{\NN}(\tilde w))$
carries the structure of an embedded Riemannian submanifold of $\calH$, and (ii) to demonstrate that
a natural Gauss--Newton-type preconditioned gradient flow in parameter space induces
exactly the Riemannian gradient flow on this manifold.
The key insight is that the Gauss--Newton preconditioner $(J^*J)^{-1}J^*$ in parameter space
is precisely the orthogonal projector onto the tangent space of $\calM$ in $\calH$; thus,
the preconditioned flow is geometrically natural---it performs steepest descent on the manifold
with respect to the ambient Hilbert metric.
Figure~\ref{fig:increment-manifold} illustrates this equivalence for a single MMS step: the neural network map $u_{\mathrm{NN}}$ sends the ball $B(w^n, r_{w^n})$ in parameter space to the increment manifold $\calM(u^n_{\mathrm{NN}})\subset\calH$, and the preconditioned gradient descent (PGD) in $\mathbb{R}^p$ induces exactly the Riemannian gradient flow (RGF) on $\calM$.
}

\begin{figure}[t]
\centering
\includegraphics[width=\textwidth]{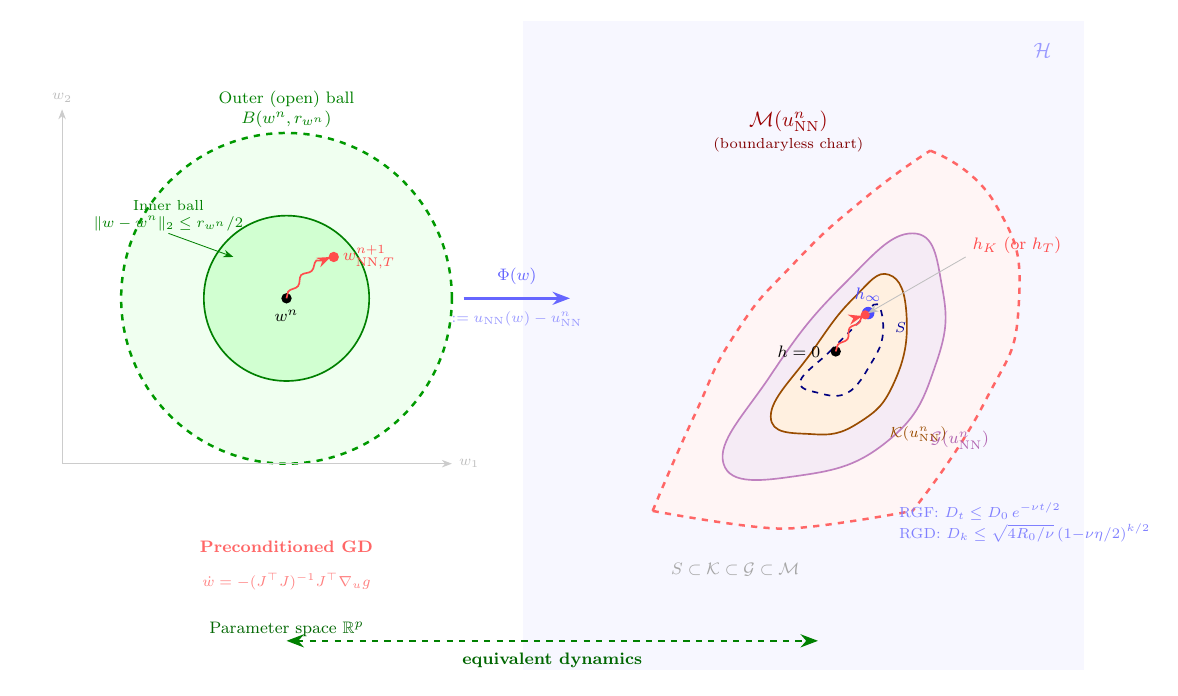}
\caption{A single neural MMS step under the two-radius construction. \textbf{Left:} Parameter space $\R^p$ with the open outer ball $B(w^n, r_{w^n})$ (dashed boundary) and the closed inner ball of radius $r_{w^n}/2$ (solid boundary); the preconditioned gradient descent (PGD) trajectory stays inside the inner ball under the interior-localization condition (Lemma~\ref{lem:interior-loc}). \textbf{Right:} The map $\Phi(w) := u_\NN(w) - u^n_{\NN}$ sends the open ball to the boundaryless increment manifold $\calM(u^n_{\NN}) \subset \calH$. The closed inner ball image $\calK(u^n_{\NN})$ (orange) sits inside the strongly geodesically convex normal neighborhood $\calG(u^n_{\NN})$ (purple, Lemma~\ref{lem:convex-radius}). The Riemannian gradient flow (or discrete RGD), starting at $h=0$, remains in the sublevel set $S \subset \calK \subset \calG$ and converges exponentially to $h_\infty$: $D_t \le D_0 e^{-\nu t/2}$ (continuous, Theorem~\ref{thm:RGF-conv}) or $D_k \le \sqrt{4R_0/\nu}\,(1-\nu\eta/2)^{k/2}$ (discrete, Theorem~\ref{thm:discrete-RGD}).}\label{fig:increment-manifold}
\end{figure}

This section adapts the analysis in \citet{cayci2025riemannian} to a continuous Hilbert space $\mathcal{H}$, generalizing the results originally established in the finite-dimensional Euclidean space.
\subsection{The Increment Set as an Embedded Riemannian Submanifold}\label{sec:embedded-manifold}

\begin{definition}[$C^2$ with $L$-Lipschitz Jacobian]\label{def:NN-Lip-L}
We say the neural network $u_\NN:\R^p\to\calH$ has an \emph{$L$-Lipschitz Jacobian} if $u_\NN\in C^2(\R^p;\calH)$ and its Fr\'echet derivative $\D u_\NN:\R^p\to\mathcal{L}(\R^p,\calH)$ satisfies
\begin{equation}\label{eq:lipschitz-jacobian}
\nm{\D u_\NN(w_1) - \D u_\NN(w_2)}_{\mathrm{op}} \leq L \nm{w_1 -  w_2}_2 \qquad\text{for all }w_1,w_2\in\R^p,
\end{equation}
with constant $L>0$ depending on the activation and architecture.
\end{definition}

\begin{definition}[Non-degeneracy Jacobian]\label{def:non-degen-jacobian} 
We call the neural network $u_\NN$ to have a $\lambda$-non-degenerate Jacobian at $w \in \mathbb R^p$ if 
there exists $\lambda >0$ such that 
\[
\D u_\NN(w)^*\D u_\NN(w) \succeq 4 \lambda^2 {I}_p
\]
where $\D u_\NN(w)$ denotes the Fr\'echet derivative \cite{} and $\D u_\NN (w)^*: \mathcal{H}\rightarrow\mathbb R^p$ denotes the adjoint, 
i.e.,
\[
\langle \D {u}_\NN (w)z,  \D {u}_\NN (w)z\rangle_{\mathcal H}
 \ge  4 \lambda^2 \|z\|_{\mathbb R^p}^2,
\qquad \forall z\in\mathbb R^p .
\]
\end{definition}

Here $\D u_\NN(w)^*\D u_\NN(w)\in\R^{p\times p}$ is the Gram matrix of the partial derivatives $\partial_{w_j}u_\NN(\cdot;w)$ with respect to the $\calH$-inner product, so the non-degeneracy condition amounts to uniform positive definiteness of this Gram matrix.

\begin{lemma}[Local preservation of positive definiteness]\label{lem:gram-lower-bound}
Assume $u_\NN$ is a neural network that has an $L$-Lipschitz Jacobian and a $\lambda_{\tilde w}$-non-degenerate Jacobian at $\tilde w \in\mathbb R^p$.
Define 
\begin{equation}\label{eq:B}
    B(\tilde w, r_{\tilde w}) := \{w\in \mathbb R ^p: \nm{w - \tilde w}_2 < r_{\tilde w} \}
\end{equation}
where
\begin{equation}\label{eq:rwdef}
r_{\tilde w} := \min \left\{\frac{\lambda_{\tilde w}}{L},\ \frac{\lambda_{\tilde w}^2}{2L\bigl(\sqrt{\nm{\D u_\NN(\tilde w)}_{\mathrm{op}}^2 + \lambda_{\tilde w}^2} + \nm{\D u_\NN(\tilde w)}_{\mathrm{op}}\bigr)},\ \frac{\lambda_{\tilde w}^2(1+\tau m_\calF)}{4L\,\Lip_{u_\NN}(1+\tau L_\calF)} \right\}.
\end{equation}
All three entries depend only on the local Jacobian data $(\lambda_{\tilde w},\,\nm{\D u_\NN(\tilde w)}_{\mathrm{op}},\,L,\,\Lip_{u_\NN})$ and the $\calF$-condition ratio $\kappa:=(1+\tau L_\calF)/(1+\tau m_\calF)$. Equivalently, with $\nu=(1+\tau m_\calF)/\tau$, $\mu=(1+\tau L_\calF)/\tau$ (Lemma~\ref{lem:smooth-strongconv}), the third entry reads $\nu\lambda_{\tilde w}^{2}/(4\mu L\,\Lip_{u_\NN})$, manifestly a function of the data alone.
Then for all $w \in B(\tilde w,r_{\tilde w})$, we have 
\begin{equation}\label{eq:gram-lower-bound}
\D u_\NN (w)^{*}\D u_\NN (w)
\succeq \lambda_{\tilde w}^2  I_p.
\end{equation}
Therefore, $u_\NN$ is an immersion on $B(w_v,r_v)$. 
\end{lemma}

\begin{proof}
For any unit vector $ z\in\mathbb S^{p-1}$, we have
\[
\|\D  u_\NN (w)[z]\|_{\mathcal H}
\ge \|\D  u_\NN (\tilde w)[z]\|_{\mathcal H}
   -\|(\D  u_\NN (w)-\D  u_\NN (\tilde w))[z]\|_{\mathcal H}
\ge \|\D u_\NN (\tilde w)[z]\|_{\mathcal H} -  L\|w-\tilde w\|_2 .
\]
Taking the minimum over $z \in\mathbb S^{p-1}$ gives
\[
\sqrt{\lambda_{\min}\!\big(\D u_\NN (w)^{*}\D u_\NN (w)\big)}
 \ge 
\sqrt{\lambda_{\min}\!\big(\D u_\NN (\tilde w)^{*}\D u_\NN (\tilde w)\big)}
-  L \|w-\tilde w\|_2 .
\]
Since $u_\NN$ has $\lambda_{\tilde w}$-non-degenerate Jacobian and
$\|w-\tilde w\|_2 < r_{\tilde w} \le \tfrac{\lambda_{\tilde w}}{L}$,
we have
\[
\sqrt{\lambda_{\min}(\D u_\NN (w)^{*}\D u_\NN (w))} > \sqrt{4 \lambda_{\tilde w}^2} - L\cdot\frac{\lambda_{\tilde w}}{L} = \lambda_{\tilde w},
\]
which implies \eqref{eq:gram-lower-bound}.
\end{proof}

Set $r^{\rm in}_{\tilde w} := r_{\tilde w}/2$, and define the \emph{increment manifold}
\begin{equation}\label{eq:Mn}
    \mathcal{M}( u_\NN(\tilde w)) := \left\{ u_\NN(w) - u_\NN(\tilde w) : w \in B(\tilde w, r_{\tilde w}) \right\} \subset \calH
\end{equation}
together with its \emph{compact localization image}
\begin{equation}\label{eq:Kn}
    \mathcal{K}( u_\NN(\tilde w)) := \left\{ u_\NN(w) - u_\NN(\tilde w) : w \in \R^p,\,\|w-\tilde w\|_2 \le r^{\rm in}_{\tilde w} \right\} \subset \mathcal{M}( u_\NN(\tilde w)).
\end{equation}

\begin{proposition}\label{prop:manifold}
$\mathcal{M}( u_\NN(\tilde w))$ is a $p$-dimensional boundaryless smooth embedded submanifold of $\mathcal{H}$, and $\calK(u_\NN(\tilde w))$ is a compact subset of $\calM(u_\NN(\tilde w))$.
\end{proposition}
\rev{%
\begin{proof}[Proof sketch (full proof in Appendix~\ref{app:proof-manifold})]
Let $\Phi(w) := u_\NN(w) - u_\NN(\tilde w)$, so $\calM(u_\NN(\tilde w)) = \Phi(B(\tilde w, r_{\tilde w}))$. The Lipschitz argument behind Lemma~\ref{lem:gram-lower-bound} gives the bi-Lipschitz lower bound $\|u_\NN(w) - u_\NN(w')\|_{\calH} \ge (2\lambda_{\tilde w} - Lr_{\tilde w})\|w - w'\|_2$, strictly positive by~\eqref{eq:rwdef}, together with the full-rank differential $\D\Phi(w) = \D u_\NN(w)$ on $B(\tilde w, r_{\tilde w})$. Hence $\Phi$ is a smooth embedding of the open ball $B(\tilde w, r_{\tilde w})$ into $\calH$ \citep[Prop.~5.30]{Lee2003}, so $\calM(u_\NN(\tilde w))$ is a boundaryless smooth embedded submanifold. Compactness of $\calK(u_\NN(\tilde w))$ follows from continuity of $\Phi$ on the compact set $\{w : \|w-\tilde w\|_2 \le r^{\rm in}_{\tilde w}\}$.
\end{proof}
}

\begin{proposition}[Tangent Space]\label{prop:tangent_metric}
Let $h=u_\NN(w)-u_\NN(\tilde w)\in \mathcal{M}(u_\NN(\tilde w))$. Then
\[
T_h\mathcal{M}(u_\NN(\tilde w))=\mathrm{Im}(\mathrm D u_\NN(w))\subset \mathcal H.
\]
Moreover, the Riemannian metric on $\mathcal{M}(u_\NN(\tilde w))$ can be chosen as the metric
induced by the ambient inner product on $\mathcal H$, i.e.
\[
\ip{\xi,\zeta}_{\mathcal{M}(u_\NN(\tilde w))}:=\ip{\xi,\zeta}_{\mathcal H},
\qquad \forall \xi,\zeta\in T_h\mathcal{M}(u_\NN(\tilde w)).
\]
In particular, $(\mathcal{M}(u_\NN(\tilde w)),\ip{\cdot,\cdot}_{\mathcal{M}(u_\NN(\tilde w))})$
is an embedded Riemannian submanifold of $\mathcal H$.
\end{proposition}
\begin{proof}
See Appendix~\ref{app:proof-tangent}.
\end{proof}

\begin{proposition}\label{prop:projection}
For $h = u_\NN(w) - u_\NN(\tilde w)\in\mathcal{M}( u_\NN(\tilde w))$,
define the bounded linear map
\begin{equation}\label{eq:Projector}
\P_h:=\D u_\NN(w) \left(\D u_\NN (w)^*\D u_\NN (w) \right)^{-1} \D u_\NN (w)^*. 
\end{equation}
    Then we have: 
\begin{enumerate}
    \item[(a) ]$\P_h$ is a projection from $\mathcal{H}$ onto $T_h\mathcal{M}( u_\NN(\tilde w))$. In other words, 
    \[
    \P_h(z) = \argmin_{y \in T_h\mathcal{M}( u_\NN(\tilde w))} \nm{y - z}^2_{\mathcal{H}}. 
    \]
    
\item[(b)] The Riemannian gradient of $g(\cdot; v)$ at $h \in \mathcal{M}( u_\NN(\tilde w))$ is 
\begin{equation*}
\operatorname{grad}_{\mathcal{M}( u_\NN(\tilde w))} g(h; v)
   =\P_h \left( \nabla g(h; v) \right).
\end{equation*}
\end{enumerate}
\end{proposition}
\begin{proof}
See Appendix~\ref{app:proof-projection}.
\end{proof}

\subsection{Preconditioned Gradient Flow as Riemannian Gradient Flow.}
Recall \eqref{eq:intro:evnn_increment_manifold}:
\begin{equation*}
\begin{cases}
h ^{n}
= \displaystyle\argmin_{h\in \{u_\NN(w) - u_\NN(w^n): w \in \mathbb R^p \}} g(h; u_\NN^n)\\[1em]
u_\NN^{n+1} = h^{n} + u_\NN^n
\end{cases},
\end{equation*}

In parameter space $\mathbb R ^p$, this corresponds

\begin{equation}\label{eq:EVNN-param}
\begin{cases}
w^{n+1}=\argmin_{w\in\R^p}
g(u_\NN(w) - u_\NN(w^n);u_\NN^n),\\[1em]
u_\NN^{n+1}= u_\NN(w^{n+1}).
\end{cases}
\end{equation}

Now consider a preconditioned gradient flow for solving \eqref{eq:EVNN-param}:
\begin{equation}\label{eq:GN-param}
\frac{\mathrm d w_t}{\mathrm dt}
   =- \left(\D u_\NN(w_t)^*\D u_\NN(w_t) \right)^{-1} \D u_\NN(w_t)^* \nabla g( u_\NN(w_t) - u_\NN(w^n);u_\NN^n),
   \qquad w_0=w^n,
\end{equation}

where $\left(\D u_\NN(w_t)^*\D u_\NN(w_t) \right)^{-1}$ is the preconditioner, and the gradient term $\nabla g( u_\NN(w_t) - u_\NN(w^n);u_\NN^n)$ is the gradient $\nabla g(h; u_\NN^n)$ evaluated at $h = u_\NN(w_t) - u_\NN(w^n)$.

Define the induced trajectory in $\mathcal{H}$ as
\[
h_t:= u_\NN(w_t) - u_\NN(w^n) \in\mathcal H.
\]
By the chain rule and~\eqref{eq:GN-param},
\begin{eqnarray}
\frac{\mathrm d h_t}{\mathrm dt}
 \notag &=& \D u_\NN(w_t)\frac{\mathrm d w_t}{\mathrm dt}
 \\
 \notag &=&-\D u_\NN(w_t) \left(\D u_\NN (w_t)^*\D u_\NN (w_t) \right)^{-1} \D u_\NN (w_t)^* \nabla g(h_t;u_\NN^n)\\ 
 &=& - \P_{h_t} \left( \nabla g(h_t;u_\NN^n) \right), \notag
\end{eqnarray}
which is a Riemannian gradient flow on $\mathcal{M}(u_\NN(w^n))$.

\begin{remark}
The theoretical framework assumes the objective functional $g$ depends primarily on the function values $h$, rather than their spatial derivatives $\nabla h$. Consequently, variational problems governed by high-order differential operators, such as solving the Poisson equation via minimization of the Dirichlet energy ($H^1$-norm), do not naturally fit within the proposed analysis. If one takes $\mathcal{H} = H^1$ and uses the Sobolev norm, the assumptions on the neural network (e.g., Jacobian Lipschitz continuity) may no longer hold, as the network is much more sensitive in the $H^1$ norm. In a discrete setting, the Lipschitz constant of the gradient would scale with the inverse square of the grid spacing (i.e., $L \propto \mathcal{O}(dx^{-2})$).
\end{remark}

\section{Convergence of the Increment Subproblem}\label{sec:conv}
\rev{%
Having established that the increment set $\calM$ is a Riemannian submanifold and the
preconditioned gradient flow is the Riemannian gradient flow, we now turn to the convergence
of this flow. The main challenge is that $\calM$ is curved (it is not a linear subspace of $\calH$),
so standard Euclidean convexity arguments do not directly apply.
This section develops three key properties: (i) strict interior localization of the reached sublevel set inside the compact image $\calK(u_\NN(\tilde w))$ under a small-step output-isolation condition (Lemma~\ref{lem:interior-loc}); (ii) existence of a strongly geodesically convex normal neighborhood $\calG(u_\NN(\tilde w))$ of $0$ containing $\calK(u_\NN(\tilde w))$ (Lemma~\ref{lem:convex-radius}); and (iii) geodesic strong convexity and smoothness of the subproblem
objective $g(\cdot;v)$ on the reached sublevel set inside $\calG(u_\NN(\tilde w))$ (Theorem~\ref{thm:g}).
Together, these yield exponential convergence of the Riemannian gradient flow (Theorem~\ref{thm:RGF-conv}).
}

\subsection{Geodesic Convexity: Definitions and Basic Tools}

The next lemma bounds how the tangent projector $\P_h$ varies along $\calM(u_\NN(\tilde w))$; this controls the extrinsic curvature (second fundamental form) of the manifold.

\begin{lemma}[Differential bound of the projector]\label{lem:DP_bound}
Let $h=\Phi(w)=u_\NN(w)-u_\NN(\tilde w)\in \mathcal M(u_\NN(\tilde w))$, and let
$\P_h$ be defined by \eqref{eq:Projector}. Assume that on $B(\tilde w,r_{\tilde w})$,
$\D u_\NN$ is $L$-Lipschitz and $\D u_\NN(w)^*\D u_\NN(w)\succeq \lambda_{\tilde w}^2 I_p$
for all $w\in B(\tilde w,r_{\tilde w})$.
Then for any $\xi\in T_h\mathcal M(u_\NN(\tilde w))$,
\[
\|\D \P_{(h)}[\xi]\|
 \le  \frac{2L}{\lambda_{\tilde w}^2} \|\xi\|_{\mathcal H}.
\]
\end{lemma}
\begin{proof}
See Appendix~\ref{app:proof-of-lemmas}.
\end{proof}

\begin{lemma}[Strict interior localization of the reached sublevel set]\label{lem:interior-loc}
Let $v=u_\NN(\tilde w)$ and define the parameter-space sublevel set
\begin{equation}\label{eq:Snw-def}
    S^w_{\tilde w} := \Big\{w \in \R^p : g\bigl(u_\NN(w)-v;\,v\bigr) \le g(0;v) = \calF[v]\Big\}.
\end{equation}
Assume an output-isolation margin around the inner ball,
\begin{equation}\label{eq:rho-in-def}
    \rho^{\rm in}_{\tilde w} := \dist_\calH\!\Big(v,\ u_\NN\bigl(\R^p \setminus B(\tilde w, r^{\rm in}_{\tilde w})\bigr)\Big) > 0.
\end{equation}
If the strict interiority condition
\begin{equation}\label{eq:interior-loc-cond}
    2\tau\bigl(\calF[v] - \calF_{\inf}\bigr) \;<\; \bigl(\rho^{\rm in}_{\tilde w}\bigr)^2
\end{equation}
holds, then $S^w_{\tilde w} \subset B(\tilde w, r^{\rm in}_{\tilde w})$, and consequently the image-side sublevel set $\{h \in \calM(u_\NN(\tilde w)) : g(h;v) \le g(0;v)\}$ is contained in the compact localization image $\calK(u_\NN(\tilde w))$ defined in~\eqref{eq:Kn}.
\end{lemma}

\begin{proof}
Let $w \in S^w_{\tilde w}$. By definition of $g$,
\[
    \tfrac{1}{2\tau}\|u_\NN(w) - v\|^2_\calH + \calF[u_\NN(w)] \;\le\; \calF[v],
\]
and since $\calF[u_\NN(w)] \ge \calF_{\inf}$ (standing assumption $\calF_{\inf} > -\infty$, automatic under $m_\calF$-strong convexity),
\begin{equation}\label{eq:energy-bound-on-Sw}
    \|u_\NN(w) - v\|^2_\calH \;\le\; 2\tau\bigl(\calF[v] - \calF_{\inf}\bigr).
\end{equation}
If $w \notin B(\tilde w, r^{\rm in}_{\tilde w})$, then by~\eqref{eq:rho-in-def}, $\|u_\NN(w) - v\|_\calH \ge \rho^{\rm in}_{\tilde w}$, so
\[
    \bigl(\rho^{\rm in}_{\tilde w}\bigr)^2 \;\le\; \|u_\NN(w) - v\|^2_\calH \;\le\; 2\tau\bigl(\calF[v] - \calF_{\inf}\bigr),
\]
contradicting~\eqref{eq:interior-loc-cond}. Hence $w \in B(\tilde w, r^{\rm in}_{\tilde w})$, proving the parameter-space inclusion.

The image-side inclusion follows: any $h$ in the sublevel set on $\calM(u_\NN(\tilde w))$ writes as $h = u_\NN(w) - v$ for some $w \in B(\tilde w, r_{\tilde w})$ with $g(h;v) \le g(0;v)$, so $w \in S^w_{\tilde w}$ satisfies $\|w-\tilde w\|_2 \le r^{\rm in}_{\tilde w}$, and therefore $h \in \calK(u_\NN(\tilde w))$.
\end{proof}

The strength of the output-isolation hypothesis~\eqref{eq:rho-in-def} is discussed in Appendix~\ref{app:suff-locality} (Remark~\ref{rem:rho-in-strength}).

\begin{lemma}[Local convexity radius]\label{lem:convex-radius}
Let $\Lambda := 2L/\lambda_{\tilde w}^2$, and assume the injectivity radius of $\calM(u_\NN(\tilde w))$ at $0$ satisfies
\begin{equation}\label{eq:inj-hypothesis}
    \mathrm{inj}_{\calM(u_\NN(\tilde w))}(0) \;\ge\; \pi/\Lambda.
\end{equation}
Define
\begin{equation}\label{eq:Gdef}
    \calG(u_\NN(\tilde w)) \;:=\; B_{\calM(u_\NN(\tilde w))}\bigl(0,\,\pi/(2\Lambda)\bigr),
\end{equation}
the open geodesic ball in $\calM(u_\NN(\tilde w))$ of radius $\pi/(2\Lambda)$ centered at $0$. Then:
\begin{enumerate}[label=(\alph*),leftmargin=*,nosep]
    \item $\calG(u_\NN(\tilde w))$ is a strongly geodesically convex normal neighborhood of $0$; in particular, any two points $h_1,h_2 \in \calG(u_\NN(\tilde w))$ are joined by a unique minimizing geodesic of $\calM(u_\NN(\tilde w))$, and this geodesic lies entirely in $\calG(u_\NN(\tilde w))$;
    \item the compact localization image $\calK(u_\NN(\tilde w))$ from~\eqref{eq:Kn} satisfies $\calK(u_\NN(\tilde w)) \subset \calG(u_\NN(\tilde w))$.
\end{enumerate}
\end{lemma}
\rev{%
\begin{proof}[Proof sketch (full proof in Appendix~\ref{app:proof-convex-radius})]
By Lemma~\ref{lem:DP_bound} the second fundamental form of the embedded manifold $\calM(u_\NN(\tilde w))$ satisfies $\|\mathrm{II}\|_{\mathrm{op}} \le \Lambda$. Because the ambient Hilbert space $\calH$ is flat, the Gauss equation \citep[Theorem~1]{alexander2006gauss} gives that the intrinsic sectional curvature of $\calM(u_\NN(\tilde w))$ is bounded above by $\Lambda^2$. The standard convexity-radius result \citep[Theorem~5.14]{cheeger2008comparison} states that for a smooth Riemannian manifold with sectional curvature $\le \Lambda^2$, the geodesic ball $B_\calM(p,r)$ is strongly geodesically convex provided $r \le \min\{\pi/(2\Lambda),\,\mathrm{inj}_\calM(p)/2\}$. Under the injectivity-radius hypothesis~\eqref{eq:inj-hypothesis}, both bounds give at least $\pi/(2\Lambda)$, so the geodesic ball $\calG(u_\NN(\tilde w))$ defined in~\eqref{eq:Gdef} is strongly geodesically convex. This establishes~(a).

For~(b), bound the geodesic distance from $0$ to any point $h \in \calK(u_\NN(\tilde w))$. Writing $h = u_\NN(w) - u_\NN(\tilde w)$ with $w \in B(\tilde w, r^{\rm in}_{\tilde w})$, the parameter-space segment $w(t) := (1-t)\tilde w + tw$ ($t\in[0,1]$) maps to a curve in $\calM(u_\NN(\tilde w))$ of length
\[
    \int_0^1\!\|\D u_\NN(w(t))(w-\tilde w)\|_\calH\,dt \;\le\; \bigl(\|\D u_\NN(\tilde w)\|_{\mathrm{op}} + L\,r_{\tilde w}\bigr)\,\|w-\tilde w\|_2.
\]
Using $\|w-\tilde w\|_2 \le r^{\rm in}_{\tilde w} = r_{\tilde w}/2$ and the choice of $r_{\tilde w}$ in~\eqref{eq:rwdef}, the second entry $r_{\tilde w}\le \lambda_{\tilde w}^2/\bigl[2L\bigl(\sqrt{\|\D u_\NN(\tilde w)\|_{\mathrm{op}}^2+\lambda_{\tilde w}^2}+\|\D u_\NN(\tilde w)\|_{\mathrm{op}}\bigr)\bigr]$ together with the first entry $r_{\tilde w}\le \lambda_{\tilde w}/L$ gives $(\|\D u_\NN(\tilde w)\|_{\mathrm{op}} + L r_{\tilde w})\,r_{\tilde w} \le \lambda_{\tilde w}^2/(2L)$ (the algebra is verified in the appendix proof). Hence
\[
    d_{\calM(u_\NN(\tilde w))}(h, 0) \;\le\; \frac{\lambda_{\tilde w}^2}{4L} \;=\; \frac{1}{\pi}\cdot\frac{\pi}{2\Lambda} \;<\; \frac{\pi}{2\Lambda},
\]
so $h$ lies in the geodesic ball $\calG(u_\NN(\tilde w))$ from~\eqref{eq:Gdef}. This proves~(b).
\end{proof}
}

The injectivity-radius hypothesis~\eqref{eq:inj-hypothesis} is discussed in Appendix~\ref{app:disc-inj-radius} (Remark~\ref{rem:inj-hypothesis}).
\subsection{Geodesic Strong Convexity of the Subproblem Objective}
\begin{lemma}[Strong convexity and Smoothness of $g(\cdot;v)$ on $\mathcal H$]\label{lem:smooth-strongconv}
Let $(\mathcal H,\langle\cdot,\cdot\rangle_\mathcal{H})$ be a real Hilbert space with norm $\|\cdot\|_{\mathcal{H}}$. 
Assume $\mathcal F:\mathcal H\to\mathbb R$ is Fr\'echet differentiable with $L_{\mathcal F}$-Lipschitz gradient, i.e.
\[
\|\nabla \mathcal F[x]-\nabla \mathcal F[y]\|_\mathcal{H}\le L_{\mathcal F} \|x-y\|_\mathcal{H},\quad\forall x,y\in\mathcal H,
\]
and $m_\mathcal{F}$-strongly convex, i.e.
\[
\langle \nabla \mathcal F[x]-\nabla \mathcal F[y], x-y\rangle_\mathcal{H} \ge m_\mathcal{F} \|x-y\|_\mathcal{H}^2,\quad\forall x,y\in\mathcal H.
\]
Fix $\tau>0$ and $v\in\mathcal H$. Recall the definition of $g$ \eqref{eq:intro:g_def}, we have
\[
g(h;v)=\frac{\|h\|^2}{2\tau}+\mathcal F\!\bigl[h+v\bigr],\qquad h\in\mathcal H.
\]
Then $g(\cdot;v)$ has $\mu$-Lipschitz gradient and is $\nu$-strongly convex with
\[
\mu=\frac{1}{\tau}+L_{\mathcal F},
\qquad
\nu=\frac{1}{\tau}+m_\mathcal{F}.
\]

\end{lemma}

\begin{proof}
See Appendix~\ref{app:proof-smooth-strongconv}.
\end{proof}

\begin{lemma}[Properties of $g(\cdot;v)$]\label{lem:g-properties}
Let $g$ be as in~\eqref{eq:intro:g_def} and write $h^*(v):=\argmin_{h\in\calH}g(h;v)$. Then:
\begin{enumerate}[label=(\alph*),leftmargin=*,nosep]
\item $h^*(v)$ is the unique solution of $h/\tau+\nabla\calF[v+h]=0$. Equivalently, $h^*(v)=J_\tau(v)-v$ is the exact-MMS increment at $v$;
\item $g(0;v)=\calF[v]$;
\item $\sqrt{2g(0;v)/\nu}=K(v)$, where
\(
K(v):=\sqrt{2\tau\calF[v]/(1+\tau m_\calF)};
\)
\item $\mu/\nu=(1+\tau L_\calF)/(1+\tau m_\calF)=:\kappa$;
\item the prefactor $4L\mu/(\nu\lambda_{\tilde w}^{2})=4L\kappa/\lambda_{\tilde w}^{2}$.
\end{enumerate}
\end{lemma}

\begin{proof}
Item~(a) follows because $g(\cdot;v)$ is $\nu$-strongly convex with $\nu>0$ (Lemma~\ref{lem:smooth-strongconv}), hence has a unique minimizer characterized by the first-order condition $\nabla g(h;v)=h/\tau+\nabla\calF[v+h]=0$.
Item~(b) is immediate from $g(0;v)=\|0\|^{2}/(2\tau)+\calF[v]$.
Item~(c) follows from $\nu=(1+\tau m_\calF)/\tau$ together with $g(0;v)=\calF[v]$:
\(
\sqrt{2g(0;v)/\nu}=\sqrt{2\calF[v]\cdot\tau/(1+\tau m_\calF)}=K(v).
\)
Items~(d)--(e) are direct algebra.
\end{proof}

The next theorem is the main technical result of this section. Its informal message: under (i) the strict interior-localization condition (Lemma~\ref{lem:interior-loc}), which forces the reached sublevel set $S$ into the compact image $\calK(u_\NN(\tilde w))\subset\calG(u_\NN(\tilde w))$, and (ii) the data condition $C(v)\le 1$, the curvature correction $\langle \D\P[\xi]\nabla g,\xi\rangle$ from ambient to manifold Hessian is controlled by $\|\nabla g\|$ on the relevant set. The bound is delicate: on the normal neighborhood $\calG$ the correction reaches its full size $\nu$ (yielding only geodesic convexity, $\Hess g\ge 0$), whereas on the smaller sublevel set $S$ the correction halves to $\nu/2$ (yielding $\nu/2$-strong convexity --- the strong-convexity constant degrades by a factor of two). No claim of strong convexity is made on the full open chart $\calM(u_\NN(\tilde w))$.
\begin{theorem}[Geodesic strong convexity and smoothness of $g$ on the reached sublevel set]\label{thm:g}
Fix $v\in\calH$ and $\tilde w\in\R^p$ with $u_\NN(\tilde w)=v$. Assume:
\begin{enumerate}[label=\textnormal{(H\arabic*)},leftmargin=*,nosep]
    \item $u_\NN$ has an $L$-Lipschitz Jacobian and a $\lambda_{\tilde w}$-non-degenerate Jacobian at $\tilde w$, so Proposition~\ref{prop:manifold} provides the boundaryless embedded manifold $\calM(u_\NN(\tilde w))$ together with the compact localization image $\calK(u_\NN(\tilde w))$ from~\eqref{eq:Kn}; the injectivity-radius hypothesis~\eqref{eq:inj-hypothesis} of Lemma~\ref{lem:convex-radius} holds at the origin $0\in\calM(u_\NN(\tilde w))$, so Lemma~\ref{lem:convex-radius} provides the strongly geodesically convex normal neighborhood $\calG(u_\NN(\tilde w)) = B_{\calM(u_\NN(\tilde w))}(0,\pi/(2\Lambda))$ with $\Lambda = 2L/\lambda_{\tilde w}^2$;
    \item the strict interior-localization condition~\eqref{eq:interior-loc-cond} of Lemma~\ref{lem:interior-loc} holds, so the parameter-space sublevel set $S^w_{\tilde w}$ from~\eqref{eq:Snw-def} is contained in $B(\tilde w, r^{\rm in}_{\tilde w})$ and consequently $\{h \in \calM(u_\NN(\tilde w)) : g(h;v) \le g(0;v)\} \subset \calK(u_\NN(\tilde w)) \subset \calG(u_\NN(\tilde w))$;
    \item the data condition
    \begin{equation}\label{eq:data-condition}
        C(v) \;:=\; \frac{4L\kappa}{\lambda_{\tilde w}^{2}}\,\max\!\bigl\{\,\|h^*(v)\|_\calH,\ K(v)\,\bigr\}\;\le\;1
    \end{equation}
    holds, where $h^*(v) = J_\tau(v) - v$ is the exact-MMS increment, $K(v) := \sqrt{2\tau\calF[v]/(1+\tau m_\calF)}$ is the sublevel-set radius (Lemma~\ref{lem:g-properties}), and the constants
    \[
    \mu=\frac{1}{\tau}+L_\calF,\qquad
    \nu=\frac{1}{\tau}+m_\calF,\qquad
    \kappa:=\mu/\nu=\frac{1+\tau L_\calF}{1+\tau m_\calF}
    \]
    are from Lemma~\ref{lem:smooth-strongconv}.
\end{enumerate}
Then $g(\cdot;v)$ on $\calM(u_\NN(\tilde w))$ satisfies:
\begin{enumerate}
    \item[(a).] $g(\cdot;v)$ is geodesically convex on the normal neighborhood $\calG(u_\NN(\tilde w))$.

    \item[(b).] The reached sublevel set
    \begin{equation}\label{eq:sublevel-S}
        S \;:=\; \bigl\{\,h\in\calM(u_\NN(\tilde w)):\ g(h;v)\le g(0;v)\,\bigr\}
    \end{equation}
    is contained in $\calK(u_\NN(\tilde w)) \subset \calG(u_\NN(\tilde w))$ by~(H2) and Lemma~\ref{lem:convex-radius}(b), and is geodesically convex. When the theorem is applied at $\tilde w=w^n_{\NN,T}$ along the neural MMS trajectory, the sublevel set $S=S_n$ inherits an $n$-dependence through both $g(0;u^n)=\calF[u^n]$ and the ambient manifold $\calM(u^n_{\NN,T})$; under Assumption~\ref{assm:traj-unif}, $S_n$ is uniformly bounded along the trajectory (see Remark~\ref{rem:traj-unif}).

    \item[(c).] $g(\cdot;v)$ is geodesically $\tfrac{\nu}{2}$-strongly convex on $S$, and $g(\cdot;v)|_{\calM(u_\NN(\tilde w))}$ has exactly one global minimizer, which lies in $S \subset \calK(u_\NN(\tilde w))$.

    \item[(d).] $g(\cdot;v)$ has geodesically $\tfrac{3\mu}{2}$-Lipschitz continuous gradient on $S$.
\end{enumerate}
\end{theorem}

Satisfiability of the data condition~\eqref{eq:data-condition} is discussed in Appendix~\ref{app:disc-data-condition} (Remark~\ref{rem:data-condition}).

\begin{proof}[Proof sketch (full proof in Appendix~\ref{app:proof-thm-g})]
The Riemannian Hessian of $g$ at $h\in\calM(u_\NN(\tilde w))$ decomposes as
$\langle \Hess g[\xi],\xi\rangle = \langle \nabla^2 g[\xi],\xi\rangle + \langle \D\P_{(h)}[\xi]\,\nabla g,\xi\rangle$.
The first term inherits the ambient bounds $\nu\|\xi\|^2 \le \langle \nabla^2 g[\xi],\xi\rangle \le \mu\|\xi\|^2$.
The second (curvature correction) term is bounded in absolute value by $\frac{2L}{\lambda_{\tilde w}^2}\|\nabla g\|\,\|\xi\|^2$ via Lemma~\ref{lem:DP_bound}, so the analysis reduces to controlling $\|\nabla g\|$ on the relevant domain.

\smallskip\noindent\emph{Part (a) — on $\calG$.}
Using $\nabla g(h^*)=0$ and $\mu$-Lipschitz $\nabla g$, $\|\nabla g(h)\|\le\mu\|h-h^*\|\le\mu(\|h\|+\|h^*\|)$. For $h=\Phi(w)$ with $w\in B(\tilde w,r_{\tilde w})$ the Lipschitz continuity of $u_\NN$ gives $\|h\|\le \Lip_{u_\NN}r_{\tilde w}$, so $\|\nabla g\|\le \mu(\Lip_{u_\NN}r_{\tilde w}+\|h^*\|)$ uniformly on $\calM\supset\calG$. The curvature correction therefore splits into two pieces, each $\le \nu/2$:
\[
\frac{2L\mu}{\lambda_{\tilde w}^2}\Lip_{u_\NN}r_{\tilde w}\le\frac{\nu}{2}
\quad\text{by the third entry of~\eqref{eq:rwdef}},\qquad
\frac{2L\mu}{\lambda_{\tilde w}^2}\|h^*\|=\frac{\nu}{2}\cdot\frac{4L\kappa}{\lambda_{\tilde w}^2}\|h^*\|\le\frac{\nu}{2}
\quad\text{by~\eqref{eq:data-condition}}.
\]
Summing, the correction is at most $\nu$, giving $\Hess g\ge\nu-\nu=0$ on $\calG$. This non-negativity together with the unique-geodesic property of Lemma~\ref{lem:convex-radius}(a) yields geodesic convexity of $g(\cdot;v)$ on $\calG(u_\NN(\tilde w))$ --- but not strong convexity at this stage; the strong-convexity bound is recovered in part (c) on the smaller set $S$.

\smallskip\noindent\emph{Part (b).} $S \subset \calK(u_\NN(\tilde w)) \subset \calG(u_\NN(\tilde w))$ by hypothesis~(H2) and Lemma~\ref{lem:convex-radius}(b); sublevel sets of geodesically convex functions on a geodesically convex set are themselves geodesically convex.

\smallskip\noindent\emph{Part (c) — on $S$.}
On $S$ the $\|h\|+\|h^*\|$ split is replaced by the tighter sublevel-set bound $\|h-h^*\|\le\sqrt{2g(0;v)/\nu}=K(v)$ (Lemma~\ref{lem:g-properties}(c)), giving $\|\nabla g\|\le\mu K(v)$ --- a \emph{single} term, independent of $r_{\tilde w}$. The correction is then
\(\frac{2L\mu}{\lambda_{\tilde w}^{2}}K(v)=\frac{\nu}{2}\cdot\frac{4L\kappa}{\lambda_{\tilde w}^2}K(v)\le\frac{\nu}{2}\) by~\eqref{eq:data-condition}, yielding $\Hess g\ge \nu-\tfrac{\nu}{2}=\tfrac{\nu}{2}$ on $S$. Uniqueness of the minimizer follows from geodesic strong convexity on $S$ together with $0\in S$.

\smallskip\noindent\emph{Part (d).} The upper bound $\Hess g \le \tfrac{3\mu}{2}$ on $S$ follows symmetrically from the same $S$-bound on $\|\nabla g\|$.
\end{proof}

\subsection{Convergence of Riemannian Gradient Flow}
\rev{%
The convergence proof combines three standard ingredients from Riemannian optimization:
\begin{enumerate}[leftmargin=*,nosep]
\item \emph{Energy dissipation}: Along the gradient flow $\dot h_t = -\grad g(h_t;v)$,
the energy $g(h_t;v)$ is nonincreasing, so the trajectory remains in the sublevel set $S$.
\item \emph{Polyak--\L{}ojasiewicz inequality}: Geodesic $\nu/2$-strong convexity on $S$ implies
$\|\grad g(h;v)\|^2 \ge \nu(g(h;v) - g(h_\infty;v))$ for all $h \in S$.
\item \emph{Gr\"onwall's inequality}: Combining (1) and (2) yields the differential inequality
$\frac{d}{dt}R_t \le -\nu R_t$, which gives $R_t \le R_0 e^{-\nu t}$.
\end{enumerate}
The distance bound (part~(c)) follows from the first variation formula and a similar Gr\"onwall argument.
}
\begin{theorem}[Convergence of the Riemannian gradient flow]\label{thm:RGF-conv}
Fix $\tilde w\in\R^p$ and $v=u_\NN(\tilde w)$, and assume hypotheses~(H1)--(H3) of Theorem~\ref{thm:g} hold (in particular, the strict interior-localization condition of Lemma~\ref{lem:interior-loc} and the data condition~\eqref{eq:data-condition}). Let $\mathcal M:=\mathcal{M}(u_\NN(\tilde w))$ be the open embedded chart, $\calK(u_\NN(\tilde w))$ its compact localization image, $\calG(u_\NN(\tilde w))$ the strongly geodesically convex normal neighborhood of Lemma~\ref{lem:convex-radius}, and define the reached sublevel set
\[
S:=\{h\in \mathcal M:  g(h;v)\le g(0;v)\}.
\]
By Theorem~\ref{thm:g}(b), $S \subset \calK(u_\NN(\tilde w)) \subset \calG(u_\NN(\tilde w))$, and $S$ is geodesically convex and compact. When this theorem is invoked at $\tilde w=w^n_{\NN,T}$ along the neural MMS trajectory, the geometric data $(\calM(u^n_{\NN,T}), \calK(u^n_{\NN,T}), \calG(u^n_{\NN,T}), S_n, h_\infty^{(n)})$ are iterate-dependent; the constants $\mu,\nu,\kappa$ depend only on the data (Lemma~\ref{lem:g-properties}), and Assumption~\ref{assm:traj-unif} provides the uniform bounds $\inf_n\lambda_{w^n_{\NN,T}}\ge\lambda_\infty>0$, $\sup_n C(u^n_{\NN,T})\le 1$, and the uniform interior-localization condition needed for the downstream chaining (Lemma~\ref{lem:induction}, Theorem~\ref{thm:global_error}).
Let $(h_t)_{t\ge 0}\subset \mathcal M$ be a solution of the Riemannian gradient flow
\begin{equation*}
\frac{\ud h_t}{\ud t}  =  -\grad g(h_t;v),\qquad t\ge 0,
\end{equation*}
with initial condition $h_0=0\in S \subset \mathcal M$. Then the following hold:

\begin{enumerate}
\item[(a)] The trajectory remains in the compact sublevel set $S$, and consequently stays inside $\calK(u_\NN(\tilde w)) \subset \calG(u_\NN(\tilde w))$ at positive distance from $\partial\calM(u_\NN(\tilde w))$ (in particular, no boundary or projected-flow construction is needed). Moreover,
\[
h_t  \to  h_\infty
 = \arg\min_{h\in\mathcal M} g(h;v)
\qquad\text{as }t\to\infty,
\]
and the minimizer $h_\infty$ lies in $S$.

\item[(b)]
Defining the energy gap
\[
R_t:=g(h_t;v)-g(h_\infty;v)\ge 0,
\]
we have for all $t\ge 0$,
\[
R_t \le R_0 e^{-\nu t},
\]
where $R_0=g(0;v)-g(h_\infty;v)$ depends on $v$ (and hence on the iterate $w^n$ when applied at MMS step $n$).

\item[(c)]
Defining the flow gap
\[
D_t:=d_{\mathcal M}(h_t,h_\infty),
\]
we have for all $t\ge 0$,
\[
\|h_t-h_\infty\|_{\mathcal H}  \le  D_t \le D_0 e^{-\nu t/2},
\]
where $D_0=d_{\mathcal M}(0,h_\infty)$ depends on $v$.
\end{enumerate}
\end{theorem}

\begin{proof}[Proof sketch (full proof in Appendix~\ref{app:proof-RGF-conv})]
The proof combines three standard ingredients from Riemannian optimization, all on the geodesically convex sublevel set $S\subset \calK \subset \calG \subset \calM$:
\begin{enumerate}
\item \emph{Energy dissipation \& interior invariance.} Along the gradient flow $\dot h_t=-\grad g(h_t;v)$, the energy $g(h_t;v)$ is nonincreasing, so the trajectory remains in $S$. By Theorem~\ref{thm:g}(b), $S \subset \calK(u_\NN(\tilde w))$, which is a compact subset of the open chart $\calM(u_\NN(\tilde w))$ at positive distance from $\partial\calM(u_\NN(\tilde w))$. Consequently the trajectory never approaches the chart boundary, the flow is well-defined for all $t \ge 0$ on $\calM$ without any boundary or projected-flow construction, and standard Picard--Lindel\"of arguments apply.
\item \emph{Polyak--\L{}ojasiewicz inequality.} Geodesic $\nu/2$-strong convexity on $S$ (Theorem~\ref{thm:g}(c)) implies
$\|\grad g(h;v)\|^2\ge \nu\,(g(h;v)-g(h_\infty;v))$ for all $h\in S$
\citep[Lemma~11.28]{boumal2023intromanifolds}.
Combining with (1) gives $\frac{d}{dt}R_t\le -\nu R_t$, hence $R_t\le R_0 e^{-\nu t}$ by Gr\"onwall.
\item \emph{First variation formula.} Setting $\phi(t):=\frac12 d_\calM^2(h_t,h_\infty)$ and
using the geodesic strong convexity of $g(\cdot;v)|_S$ at both $h_t$ and $h_\infty$ (both in $\calG$, where the connecting geodesic is unique by Lemma~\ref{lem:convex-radius}),
one obtains $\phi'(t)\le -\nu\phi(t)$, yielding $D_t\le D_0 e^{-\nu t/2}$. The ambient
chord-length bound $\|h_t-h_\infty\|_\calH\le D_t$ follows since $\calM$ is embedded.
\end{enumerate}
\end{proof}

The continuous-time exponential bound of Theorem~\ref{thm:RGF-conv} has a natural discrete-time counterpart: on the same sublevel set $S$, discrete Riemannian gradient descent (RGD) on $\calM(u_\NN(\tilde w))$ that takes its steps along the geodesics --- i.e., via the exponential map --- inherits a matching linear contraction rate. We formulate the discrete iteration with the exponential map rather than a generic retraction because it makes the geodesic-smoothness step in the proof immediate and matches the canonical Riemannian descent picture; the practical algorithm we actually implement is parameter-space preconditioned gradient descent (PGD), the trivial-to-implement Euler discretization of the parameter-space flow, and by the parameter--manifold equivalence of Section~\ref{sec:RGF} PGD performs the same descent dynamics on $\calM(u_\NN(\tilde w))$. Theorem~\ref{thm:discrete-RGD} is therefore the discrete-time companion of the practical iteration; we do not formalize a per-step coordinate identification because no implementation actually computes the exponential map.

\begin{theorem}[Discrete Riemannian gradient descent on the sublevel set]\label{thm:discrete-RGD}
Assume the hypotheses of Theorem~\ref{thm:RGF-conv} hold, and let $f := g(\cdot; v)|_{\calM(u_\NN(\tilde w))}$. By Theorem~\ref{thm:g}(c)--(d), $f$ is geodesically $\alpha$-strongly convex and $\beta$-smooth on the sublevel set $S$ with
\[
\alpha := \tfrac{\nu}{2},\qquad \beta := \tfrac{3\mu}{2}.
\]
Consider the discrete Riemannian gradient descent iteration on $\calM(u_\NN(\tilde w))$ that takes its steps along the geodesics, i.e., via the exponential map:
\begin{equation}\label{eq:discrete-RGD}
    h_{k+1} \;=\; \Exp_{h_k}\bigl(-\eta\,\grad f(h_k)\bigr),\qquad k\ge 0,\quad h_0 = 0.
\end{equation}
By Lemma~\ref{lem:convex-radius}, $S\subset\calG$ lies inside a strongly geodesically convex normal neighborhood of $0$, on which $\Exp_{h_k}$ is well defined for all tangent vectors of magnitude $\le\pi/(2\Lambda)$ (in particular for $\eta\,\|\grad f(h_k)\|$ when $\eta\le 1/\beta$ and $h_k\in S$).
For any constant step size $\eta$ with $0 < \eta \le 1/\beta = 2/(3\mu)$, the iterates $\{h_k\}$ remain in $S$, and the energy gap $R_k := f(h_k) - f(h_\infty)$ and the flow gap $D_k := d_\calM(h_k, h_\infty)$ satisfy
\begin{align}
    R_k \;&\le\; (1 - \alpha\eta)^k R_0
    \;=\; \Bigl(1 - \tfrac{\nu\eta}{2}\Bigr)^{\!k}\bigl(g(0;v) - g(h_\infty;v)\bigr), \label{eq:discrete-RGD-bound} \\
    \|h_k - h_\infty\|_\calH \;\le\; D_k \;&\le\; \sqrt{\tfrac{4 R_0}{\nu}}\,\Bigl(1-\tfrac{\nu\eta}{2}\Bigr)^{\!k/2}, \label{eq:discrete-RGD-flow-gap}
\end{align}
where $h_\infty = \arg\min_{h\in\calM(u_\NN(\tilde w))} g(h;v)$ as in Theorem~\ref{thm:RGF-conv}.
In particular, with the maximal step $\eta = 1/\beta = 2/(3\mu)$, the energy contraction factor is $1 - \nu/(3\mu) = 1 - 1/(3\kappa)$, where $\kappa = \mu/\nu = (1+\tau L_\calF)/(1+\tau m_\calF)$ is the data-side condition number.
\end{theorem}

\begin{proof}[Proof sketch]
Since $h_{k+1}=\Exp_{h_k}(-\eta\,\grad f(h_k))$ lies on the geodesic emanating from $h_k$ in direction $\xi_k:=-\eta\,\grad f(h_k)$, geodesic $\beta$-smoothness of $f$ on $S$ \citep[Lemma~11.28]{boumal2023intromanifolds} gives the descent inequality
\[
f(h_{k+1}) \le f(h_k) + \langle \grad f(h_k), \xi_k\rangle + \tfrac{\beta}{2}\|\xi_k\|^2.
\]
Plugging in $\xi_k=-\eta\,\grad f(h_k)$,
\[
    f(h_{k+1}) \;\le\; f(h_k) - \eta\bigl(1 - \tfrac{\beta\eta}{2}\bigr)\|\grad f(h_k)\|^2 \;\le\; f(h_k) - \tfrac{\eta}{2}\|\grad f(h_k)\|^2
\]
for $\eta \le 1/\beta$. In particular $f(h_{k+1}) \le f(h_k)$, so $h_{k+1} \in S$ and the iteration is well-defined.
Geodesic $\alpha$-strong convexity on $S$ implies the Polyak--\L{}ojasiewicz inequality $\|\grad f(h)\|^2 \ge 2\alpha\,(f(h)-f(h_\infty))$ on $S$ \citep[Lemma~11.28]{boumal2023intromanifolds}. Combining,
\[
    f(h_{k+1}) - f(h_\infty) \;\le\; (1 - \alpha\eta)\bigl(f(h_k) - f(h_\infty)\bigr),
\]
which iterates to~\eqref{eq:discrete-RGD-bound}.
For the flow gap~\eqref{eq:discrete-RGD-flow-gap}, geodesic $\alpha$-strong convexity on $S$ yields $f(h_k)-f(h_\infty)\ge \tfrac{\alpha}{2}d_\calM(h_k,h_\infty)^2$ \citep[Theorem~11.21]{boumal2023intromanifolds}, so
$D_k^2 \le (2/\alpha)R_k \le (2/\alpha)(1-\alpha\eta)^k R_0 = (4/\nu)(1-\nu\eta/2)^k R_0$,
and $\|h_k-h_\infty\|_\calH\le D_k$ since $\calM$ is embedded in $\calH$.
\end{proof}

\begin{remark}[Inexact linear-solve and Levenberg--Marquardt damping]\label{rem:inexact-discrete}
In practice the exact Gauss--Newton system $(J^*J)^{-1}J^*\nabla g$ is replaced by an inexact one, e.g., a conjugate-gradient solve truncated to tolerance $\varepsilon_{\rm CG}$ and/or Levenberg--Marquardt damping $(J^*J+\rho I)^{-1}$. The resulting RGD step can be analyzed as a perturbed-gradient iteration $h_{k+1}=\Exp_{h_k}(-\eta\,\grad f(h_k)+e_k)$ with $\|e_k\|\le\gamma\,\eta\,\|\grad f(h_k)\|+\eta\,\delta$ for $\gamma\in[0,1)$ and $\delta\ge 0$ controlled by $\varepsilon_{\rm CG}$ and $\rho$. A standard perturbed-PL argument \citep[Theorem~11.31]{boumal2023intromanifolds} gives linear convergence up to an error floor proportional to $\delta^2$:
\(
    f(h_k)-f(h_\infty)\le(1-(1-\gamma)\alpha\eta)^k(f(h_0)-f(h_\infty))+O(\delta^2).
\)
A precise statement is left to future work; Theorem~\ref{thm:discrete-RGD} is the exact-step counterpart.
\end{remark}

\section{Global Error Propagation for Neural MMS with Inexact Increment Solutions}\label{sec:error-propagation}
\rev{%
Having established exponential convergence of the Riemannian gradient flow for a single
increment subproblem (Theorem~\ref{thm:RGF-conv}), we now address the global behavior
of the neural MMS across multiple iterations.
The key question is: how do the errors from (i)~finite-time truncation of the inner flow
and (ii)~neural approximation of the proximal map accumulate over iterations?
The answer relies on the contraction property of the proximal map $J_\tau$:
at each MMS step, the contraction factor $1/(1+\tau m_\calF)$ absorbs the per-step errors,
yielding a uniform $\delta$-tracking bound (Theorem~\ref{thm:global_error}).
Figure~\ref{fig:error-propagation} illustrates this one-step error decomposition.
}

Recall that $\mathcal{F}:\mathcal{H}\to\mathbb{R}$ is continuously Fr\'echet differentiable, coercive, and $m_\mathcal{F}$-strongly convex. For $\tau>0$, define the proximal map
\begin{equation}\label{eq:prox_map_R}
J_\tau(x) := \mbox{prox}_{\tau \mathcal{F}}(x)
 := \argmin_{u\in \mathcal H}\Big(\frac{1}{2\tau}\|u-x\|^2+\mathcal F[u]\Big).
\end{equation}
Since $m_\calF>0$ and $\tau>0$, the objective in \eqref{eq:prox_map_R} is $(\frac1\tau+ m_\calF)$--strongly convex and hence $J_\tau$ is single-valued and well-defined.

Let $\{u^n\}_{n\ge 0}$ denote the exact MMS sequence defined by
\begin{equation*}
u^{n+1}   =  J_\tau(u^n),\qquad n\ge 0.
\end{equation*}

\begin{lemma}[Lipschitz continuity of $J_\tau$]\label{lem:lip_J_tau}
Since $\mathcal{F}$ is $m_\mathcal{F}$--strongly convex with $m_\calF>0$, the proximal map satisfies
\begin{equation*}
\|J_\tau (x)-J_\tau(y)\| \le \frac{1}{1+\tau m_\mathcal{F}} \|x-y\|,\qquad \forall x,y\in \mathcal{H}.
\end{equation*}
\end{lemma}
\begin{proof}
Write $\hat x=J_\tau(x)$ and $\hat y=J_\tau(y)$. By the optimality conditions, $x-\hat x=\tau\nabla\mathcal{F}(\hat x)$ and $y-\hat y=\tau\nabla\mathcal{F}(\hat y)$. Since $\mathcal{F}$ is $m_\mathcal{F}$--strongly convex, $\nabla\mathcal{F}$ is $m_\mathcal{F}$--monotone, so
$\langle (x-\hat x)-(y-\hat y),\hat x-\hat y\rangle \ge \tau m_\mathcal{F}\|\hat x-\hat y\|^2$.
Rearranging and applying Cauchy--Schwarz gives $(1+\tau m_\mathcal{F})\|\hat x-\hat y\|\le\|x-y\|$.
\end{proof}

Also recall 
\begin{equation}
    J_{\tau,\NN}(x) := \argmin_{u_\NN \in \NN} \left\{\frac{1}{2\tau}\nm{u_\NN -x}^2 + \mathcal{F}[u_\NN]\right\}. 
\end{equation}

With Assumption \ref{ass:intro:locality}, Lemma~\ref{lem:interior-loc} (which contains the minimizer strictly inside the inner ball), and Theorem~\ref{thm:RGF-conv} (which gives existence and uniqueness of the minimizer on $\calM(x)$), we have
\begin{equation}\label{eq:J-tau-NN-well-defined}
J_{\tau,\NN}(x) =  h_\infty + x,
\end{equation}
where 
\[
h_\infty = \argmin_{h \in \mathcal{M}(x)} g(h;x) 
\]
is uniquely determined by $x$.  Thus, $J_{\tau, \NN}$ is well-defined. 

Let $\{u^n_{\NN,T}\}_{n\geq 0}$ be the sequence of Riemannian gradient flow (RGF) running until $T$ starting from the previous $u_{\NN,T}^n$. That is, 
\begin{equation}\label{eqn:RGF-update}
    u^{n+1}_{\NN,T} := u^n_{\NN,T} + h_T,
\end{equation}
where $\{h_t\}_{t\geq 0 } \subset \mathcal{M}(u^n_{\NN,T})$ solves
\[
\frac{\ud h_t}{\ud t} = - \grad g(h_t; u^n_{\NN,T}) 
\]
with $h_0 = 0$. 

In particular, 
\begin{equation}\label{eq:u_infty_def}
u^{n+1}_{\NN,\infty} = u^n_{\NN,T} + h _\infty = J_{\tau,\NN}\big(u^n_{\NN,T}\big).
\end{equation}

\begin{corollary}[Finite-time error of RGF]\label{cor:rgf-finite-time}
There exists a sequence $\{\bar C_n\}_{n\ge 0}$ with $\bar C_n>0$ such that for all $n\ge 1$,
\begin{equation}\label{eq:rgf_exp}
\|u^{n}_{\NN,T}-u^{n}_{\NN,\infty}\|_\calH \;\le\; \bar C_n\, e^{-\nu T/2},
\qquad
\nu=\frac{1+\tau m_\calF}{\tau} \ \text{(Lemma~\ref{lem:g-properties})}.
\end{equation}
Explicitly, $\bar C_n := D_0^{(n)}=d_{\calM(u^n_{\NN,T})}(0,h_\infty^{(n)})$ is the geodesic distance on the increment manifold from the origin to the inner-flow minimizer at step $n$.
\end{corollary}

\begin{proof}
\emph{Step 1: identify $\bar C_n$.}
By~\eqref{eqn:RGF-update} and~\eqref{eq:u_infty_def}, $u^n_{\NN,T}-u^n_{\NN,\infty}=h_T-h_\infty^{(n)}$. Theorem~\ref{thm:RGF-conv}(c) gives $\|h_T-h_\infty^{(n)}\|_\calH\le D_T\le D_0^{(n)} e^{-\nu T/2}$. Setting $\bar C_n:=D_0^{(n)}$ yields~\eqref{eq:rgf_exp}.

\emph{Step 2: explicit bound on $D_0^{(n)}$.}
Under the data condition~\eqref{eq:data-condition} at iterate $u^n_{\NN,T}$, we bound $D_0^{(n)}$ via a CAT chord-to-geodesic comparison. Since $h_\infty^{(n)}\in S_n$ and $0\in S_n$, the sublevel-set bound~\eqref{eq:dist_to_star_on_S} together with Lemma~\ref{lem:g-properties}(c) gives
\(
\|h_\infty^{(n)}-h^*(u^n_{\NN,T})\|_\calH \le \sqrt{2g(0;u^n_{\NN,T})/\nu} = K(u^n_{\NN,T}).
\)
Combined with Lemma~\ref{lem:g-properties}(a) and the triangle inequality,
\begin{equation}\label{eq:h-inf-norm-bound}
\|h_\infty^{(n)}\|_\calH \;\le\; 2\,\max\bigl\{K(u^n_{\NN,T}),\,\|h^*(u^n_{\NN,T})\|_\calH\bigr\} \;\le\; \frac{\lambda_{w^n_{\NN,T}}^{2}}{2L\kappa},
\end{equation}
where the second inequality uses $C(u^n_{\NN,T})\le 1$. Since $\calM(u^n_{\NN,T})$ is locally $\mathrm{CAT}(\Lambda_n^{2})$ with $\Lambda_n:=2L/\lambda_{w^n_{\NN,T}}^{2}$ (Lemma~\ref{lem:DP_bound}), and $\Lambda_n\,\|h_\infty^{(n)}\|_\calH\le 1/\kappa\le 1$ by~\eqref{eq:h-inf-norm-bound}, the chord-to-geodesic comparison~\citep[Part~II, Prop.~1.4 and Thm.~4.1]{bridson2013metric} yields
\(
D_0^{(n)} = d_{\calM(u^n_{\NN,T})}(0,h_\infty^{(n)}) \le c_\Lambda^{(n)}\,\|h_\infty^{(n)}\|_\calH,
\)
where the chord/geodesic constant $c_\Lambda^{(n)}\ge 1$ depends only on the dimensionless product $\Lambda_n\,\|h_\infty^{(n)}\|_\calH$, which is bounded by $1/\kappa$ by~\eqref{eq:h-inf-norm-bound}. Combining:
\begin{equation}\label{eq:bar-C-bound}
\bar C_n \;=\; D_0^{(n)} \;\le\; c_\Lambda^{(n)}\cdot\frac{\lambda_{w^n_{\NN,T}}^{2}}{2L\kappa}.
\end{equation}
\end{proof}

\begin{remark}[Uniform boundedness of $\bar C_n$ along the trajectory]\label{rem:Cbar-bound}
The explicit bound~\eqref{eq:bar-C-bound} from the proof of Corollary~\ref{cor:rgf-finite-time} is pointwise in $n$. Uniform boundedness along the trajectory follows under Assumption~\ref{assm:traj-unif}: by Remark~\ref{rem:traj-unif}, the derived uniform operator-norm bound $\sup_n\|\D u_\NN(w^n_{\NN,T})\|_{\rm op}<\infty$ implies $\sup_n\lambda_{w^n_{\NN,T}}<\infty$ (singular values are dominated by the operator norm), and $\Lambda_n\,\|h_\infty^{(n)}\|_\calH\le 1/\kappa$ remains uniformly $\le 1$ for all $n$. Setting $c_\Lambda:=\sup_n c_\Lambda^{(n)}<\infty$, we obtain
\[
\sup_{n\ge 0}\bar C_n \;\le\; c_\Lambda\cdot\frac{\bigl(\sup_n\lambda_{w^n_{\NN,T}}\bigr)^{2}}{2L\kappa} \;<\;\infty,
\]
so the bound~\eqref{eq:rgf_exp} chains uniformly across the trajectory: $\sup_n\bar C_n\,e^{-\nu T/2}=\Theta(e^{-\nu T/2})$.
\end{remark}

\subsection{One-Step Error Decomposition}

\begin{assumption}[Neural proximal approximation precision]
\label{assm:approx-precision}
The computed neural proximal map $J_{\tau,\NN}$ approximates the exact proximal
map $J_\tau$ at the one-step level. More precisely, there exists a nonnegative
function $\varepsilon_{\NN}:\calH\to\mathbb{R}_+$ such that, for all relevant
inputs $x\in\calH$ arising in the error-propagation analysis,
\begin{equation*}
\|J_{\tau,\NN}(x)-J_\tau(x)\|_\calH
\le
\varepsilon_{\NN}(x).
\end{equation*}
We write
\[
    \varepsilon:=\sup_x \varepsilon_{\NN}(x),
\]
where the supremum is taken over the relevant set of inputs along the exact and
inexact MMS trajectories. The target accuracy $\delta>0$ is chosen so that
\begin{equation}\label{eqn:approximate-error}
\varepsilon
<
\frac{\tau m_{\calF}}{1+\tau m_{\calF}}\,\delta .
\end{equation}
\end{assumption}

Interpretation of the one-step precision $\varepsilon$ and its decomposition into trial-space and sampling/quadrature errors is discussed in Appendix~\ref{app:disc-approx-precision} (Remark~\ref{rem:approx-precision}).

\begin{assumption}[Non-degeneracy on all increment manifolds]\label{assm:center-nondeg}
Let $\{w^n\}_{n\ge 0}$ be the parameters generated by the neural MMS (or the inexact increment solver). Assume $u_\NN$ has $\lambda_{w^n}$-non-degenerate Jacobian at all $w^n$ for all $n$. That is, for all $n\ge 0$,
\[
\lambda_{\min}\!\big(\D u_{\NN}(w^n)^* \D u_{\NN}(w^n)\big)\ge \lambda_{w^n}^2,
\]
and the radii $\{r_{w^n}\}$ are chosen so that Lemma~\ref{lem:gram-lower-bound} applies on each ball
$B(w^n, r_{w^n})$ (cf.~\eqref{eq:B}).
\end{assumption}

\begin{assumption}[Trajectory uniformity]\label{assm:traj-unif}
Along the inexact trajectory $\{w^n_{\NN,T}\}_{n\ge 0}$ generated by the neural MMS, the following uniform bounds hold:
\begin{enumerate}[label=\textnormal{(U\arabic*)},leftmargin=*,nosep]
\item \emph{Uniform non-degeneracy.} There exists $\lambda_\infty>0$ such that
\(
\inf_{n\ge 0}\lambda_{w^n_{\NN,T}}\;\ge\;\lambda_\infty.
\)
This strengthens Assumption~\ref{assm:center-nondeg} from per-iterate positivity to a uniform positive lower bound.
\item \emph{Uniform data condition.} The hypothesis of Theorem~\ref{thm:g} holds at every iterate:
\[
\sup_{n\ge 0}\,C(u^n_{\NN,T})\;\le\;1,
\]
where $C(v):=(4L\kappa/\lambda_{\tilde w}^{2})\,\max\{\|h^*(v)\|_\calH,\,K(v)\}$ is the data condition of Theorem~\ref{thm:g} (see~\eqref{eq:data-condition}; the $\lambda_{\tilde w}$ in $C(v)$ is the local non-degeneracy constant at $w^n_{\NN,T}$).
\item \emph{Uniform output isolation around the inner ball.} There exists $\rho^{\rm in}_\infty>0$ such that
\(
\inf_{n\ge 0}\rho^{\rm in}_{w^n_{\NN,T}}\;\ge\;\rho^{\rm in}_\infty,
\)
where $\rho^{\rm in}_{w^n_{\NN,T}}:=\dist_\calH\bigl(u^n_{\NN,T},\,u_\NN(\R^p\setminus B(w^n_{\NN,T},r^{\rm in}_{w^n_{\NN,T}}))\bigr)$ is the inner output-isolation radius from~\eqref{eq:rho-in-def} at iterate $w^n_{\NN,T}$, and $r^{\rm in}_{w^n_{\NN,T}} := r_{w^n_{\NN,T}}/2$. Combined with $\sup_n(\calF[u^n_{\NN,T}]-\calF_{\inf})<\infty$ (a consequence of the uniform energy bound; see Remark~\ref{rem:traj-unif}), this supplies the $n$-uniform interior-localization hypothesis~\eqref{eq:interior-loc-cond} of Lemma~\ref{lem:interior-loc} under the small-step condition $2\tau\sup_n(\calF[u^n_{\NN,T}]-\calF_{\inf}) < (\rho^{\rm in}_\infty)^2$. It also underwrites the $n$-uniform version of Assumption~\ref{ass:intro:locality} used by Theorem~\ref{thm:small-tau-locality}.
\item \emph{Uniform injectivity radius.} The injectivity-radius hypothesis~\eqref{eq:inj-hypothesis} of Lemma~\ref{lem:convex-radius} holds at every iterate, i.e.,
\[
\mathrm{inj}_{\calM(u^n_{\NN,T})}(0)\;\ge\;\pi/\Lambda_n,\qquad \Lambda_n:=2L/\lambda_{w^n_{\NN,T}}^2,
\]
for all $n\ge 0$, so Lemma~\ref{lem:convex-radius} provides $\calG(u^n_{\NN,T})$ at every iterate. (Under (U1), $\Lambda_n\le 2L/\lambda_\infty^2$ is uniformly bounded, so a single uniform lower bound $\inf_n \mathrm{inj}_{\calM(u^n_{\NN,T})}(0)\ge \pi\lambda_\infty^2/(2L)$ suffices.)
\end{enumerate}
\end{assumption}

Discussion of Assumption~\ref{assm:traj-unif} -- including the three derived uniformities, verifiability of (U1)--(U3) under the initial-gap condition, and the connection of (U3) to global inverse-stability of $u_\NN$ -- is given in Appendix~\ref{app:disc-traj-unif} (Remarks~\ref{rem:traj-unif}, \ref{rem:tracking-uniformity}, and~\ref{rem:globalPGF-traj-unif}).

\begin{figure}[htpb]
    \centering
    \includegraphics[width=0.75\linewidth]{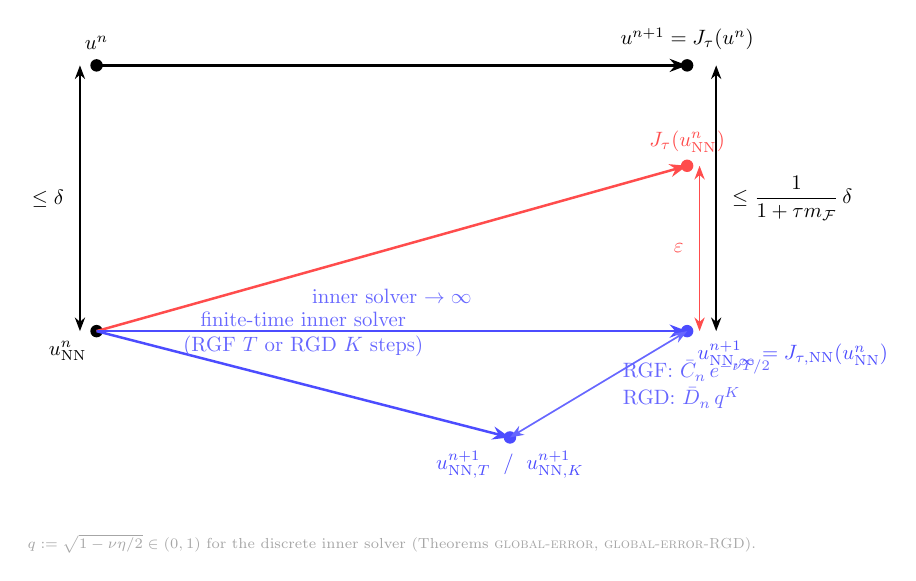}
    \caption{Error propagation across one MMS induction step. The per-step error decomposes into the proximal-map contraction $\delta/(1+\tau m_\calF)$ (right vertical, black), the neural approximation error $\varepsilon$ (red), and the inner-solver truncation error -- $\bar C_n\,e^{-\nu T/2}$ for the continuous RGF (Corollary~\ref{cor:rgf-finite-time}) or $\bar D_n\,q^K$ with $q = \sqrt{1-\nu\eta/2}$ for the discrete RGD (Theorem~\ref{thm:discrete-RGD}). The two diagonal blue arrows mark the same starting point $u^n_\NN$ mapped by the two inner solvers.}
    \label{fig:error-propagation}
\end{figure}

\begin{lemma}[Induction Step]\label{lem:induction}
Suppose $\nm{u^n_{\NN,T} - u^n}_\calH \leq \delta$, and assume Theorem~\ref{thm:g} applies at the iterate $\tilde w=w^n_{\NN,T}$ (i.e., the data condition~\eqref{eq:data-condition} holds at $u^n_{\NN,T}$). If one runs the Riemannian gradient flow starting from $u^n_{\NN,T}$ on the manifold $\calM(u^n_{\NN,T})$ for time $T=T_n$ satisfying
\begin{equation}\label{eqn:T}
T_n \;\ge\; \frac{2}{\nu}\log\!\left(\bar C_n\cdot\frac{1+\tau m_\calF}{\tau m_\calF(\delta-\varepsilon)-\varepsilon}\right),
\end{equation}
where $\varepsilon$ is in Assumption~\ref{assm:approx-precision} and $\bar C_n$ is the step-$n$ constant from Corollary~\ref{cor:rgf-finite-time}, then the error at the next step is bounded:
\[
\nm{u^{n+1}_{\NN,T} - u^{n+1}}_\calH \;\le\;\delta.
\]
By Lemma~\ref{lem:g-properties} and Lemma~\ref{lem:gram-lower-bound}, the constants $\nu,\mu,\kappa$ in~\eqref{eqn:T} depend only on the data. The genuinely iterate-dependent quantities are $\bar C_n$ and the sublevel set $S_n$; under Assumption~\ref{assm:traj-unif} together with Remark~\ref{rem:Cbar-bound}, $\sup_n\bar C_n<\infty$, so $\sup_n T_n<\infty$ and the inductive bound chains uniformly.
\end{lemma}
\begin{proof}
By Lipschitz continuity of $J_\tau$ (Lemma~\ref{lem:lip_J_tau}), if $\|u^n_{\NN,T}-u^n\|_\calH<\delta$ then
\begin{equation}\label{eq:propagate_term}
\|J_\tau(u^n_{\NN,T})-J_\tau(u^n)\|_\calH
\le \frac{1}{1+\tau m_\calF}\|u^n_{\NN,T}-u^n\|_\calH
<\frac{\delta}{1+\tau m_\calF}.
\end{equation}
By the triangle inequality and~\eqref{eq:u_infty_def},~\eqref{eq:rgf_exp},~\eqref{eqn:approximate-error},~\eqref{eq:propagate_term},
\begin{align}\label{eq:error_prop}
\|u^{n+1}_{\NN,T}-u^{n+1}\|_\calH
&\le \|u^{n+1}_{\NN,T}-u^{n+1}_{\NN,\infty}\|_\calH + \|u^{n+1}_{\NN,\infty}-u^{n+1}\|_\calH \nonumber\\
&= \|u^{n+1}_{\NN,T}-u^{n+1}_{\NN,\infty}\|_\calH + \|J_{\tau,\NN}(u^n_{\NN,T})-J_\tau(u^n)\|_\calH \nonumber\\
&\le \|u^{n+1}_{\NN,T}-u^{n+1}_{\NN,\infty}\|_\calH + \|J_{\tau,\NN}(u^n_{\NN,T})-J_\tau(u^n_{\NN,T})\|_\calH + \|J_\tau(u^n_{\NN,T})-J_\tau(u^n)\|_\calH \nonumber\\
&\le \bar C_n\,e^{-\nu T/2}+\varepsilon+\frac{1}{1+\tau m_\calF}\|u^n_{\NN,T}-u^n\|_\calH.
\end{align}
In particular, if $\|u^n_{\NN,T}-u^n\|_\calH<\delta$ then
\[
\|u^{n+1}_{\NN,T}-u^{n+1}\|_\calH \;\le\; \bar C_n\,e^{-\nu T/2}+\frac{\delta}{1+\tau m_\calF}+\varepsilon.
\]
Choosing $T$ so that
\(
\bar C_n\,e^{-\nu T/2}+\delta/(1+\tau m_\calF)+\varepsilon\le\delta
\)
is equivalent to
\(
e^{-\nu T/2}\le \bar C_n^{-1}\bigl(\tau m_\calF\delta/(1+\tau m_\calF)-\varepsilon\bigr),
\)
i.e., to the threshold~\eqref{eqn:T}. This ensures $\|u^{n+1}_{\NN,T}-u^{n+1}\|_\calH\le\delta$, closing the induction.
\end{proof}

\begin{remark}\label{rem:universal_decomposition}
We emphasize that the error decomposition in \eqref{eq:error_prop} is fundamentally universal and independent of the specific optimization algorithm used for the inner MMS problem. Structurally, any numerical inner solver yields a three-part accumulation:
(i) the contraction from the proximal operator,
(ii) the spatial approximation error $\varepsilon$ of the neural network, and
(iii) an inexact optimization error.
The primary theoretical contribution of our preconditioned gradient flow analysis is showing that, under the standing assumptions and Assumption~\ref{assm:center-nondeg} (and uniformly along the trajectory under Assumption~\ref{assm:traj-unif}), the Gauss--Newton-type flow specifically guarantees a rigorously bounded, exponentially decaying estimate $\bar C_n\,e^{-\nu T/2}$ at each step $n$ for this inexact optimization step.
\end{remark}

Lemma~\ref{lem:induction}'s threshold~\eqref{eqn:T} chooses a uniform $T_n$ sufficient for the per-step bound $e_n\le\delta$, but only yields a uniform-in-$n$ bound on $e_n=\|u^n_{\NN,T}-u^n\|_\calH$ -- not a summable one. The downstream parameter-space convergence theorem (Theorem~\ref{thm:global-PGF}) requires a finite trajectory length $\sum_n\Delta_n<\infty$, which in turn requires $\sum_n e_n<\infty$ (entering via Corollary~\ref{cor:param-proxy} and~\eqref{eq:per-step-lambda-inf}). We therefore record below a refinement of Lemma~\ref{lem:induction} obtained by strengthening the choice of $T_n$ to grow linearly in $n$; this is exactly the geometric-tracking bound used in the proof of Theorem~\ref{thm:global-PGF}(b)~\eqref{eq:partial-sum-Delta}. The bounds~\eqref{eq:geometric-tracking}, \eqref{eq:sum-e-bound}, \eqref{eq:sum-e-partial} are stated as a corollary (rather than as a remark or an inline note) so that they are citable as formal equations in that proof.

\begin{corollary}[Geometric tracking under growing $T_n$]\label{cor:geometric-tracking}
Fix $\eta_0\in\bigl(0,\,(1-\rho)\delta-\varepsilon\bigr)$ (a positive interval by Assumption~\ref{assm:approx-precision}) and strengthen Lemma~\ref{lem:induction}'s threshold to require $T_n$ to grow linearly in $n$:
\begin{equation}\label{eqn:T-growing}
T_n \;\ge\; \frac{2}{\nu}\Bigl(\log\!\frac{\bar C_n}{\eta_0}+n\log\!\frac{1}{\rho}\Bigr),
\end{equation}
so that $\bar C_n\,e^{-\nu T_n/2}\le\eta_0\rho^n$. Then Lemma~\ref{lem:induction}'s recurrence~\eqref{eq:error_prop} yields the geometric-tracking bound
\begin{equation}\label{eq:geometric-tracking}
e_n \;:=\;\|u^n_{\NN,T}-u^n\|_\calH \;\le\; \rho^n\,e_0 \,+\, n\,\eta_0\,\rho^{n-1} \,+\, \frac{\varepsilon}{1-\rho}\,(1-\rho^n);
\end{equation}
i.e.\ $e_n$ decays geometrically toward an asymptotic floor of $\varepsilon/(1-\rho)$. Two regimes follow:
\begin{itemize}[leftmargin=*,nosep]
\item \emph{(I) Perfect-approximation regime ($\varepsilon=0$):} the trajectory error sum is finite,
\begin{equation}\label{eq:sum-e-bound}
\sum_{n=0}^{\infty} e_n \;\le\; \frac{e_0}{1-\rho}+\frac{\eta_0}{(1-\rho)^{2}}.
\end{equation}
\item \emph{(II) Uniform-approximation regime ($\varepsilon>0$):} the geometric component decays as in (I), but $\liminf_n e_n=\varepsilon/(1-\rho)>0$. For any finite horizon $N\ge 0$,
\begin{equation}\label{eq:sum-e-partial}
\sum_{n=0}^{N} e_n \;\le\; \frac{e_0}{1-\rho}+\frac{\eta_0}{(1-\rho)^{2}}+\frac{(N+1)\,\varepsilon}{1-\rho}.
\end{equation}
\end{itemize}
\end{corollary}

\begin{proof}
Under~\eqref{eqn:T-growing}, $\bar C_n\,e^{-\nu T_n/2}\le\eta_0\rho^n$, so the per-step accumulator in~\eqref{eq:error_prop} is $\eta_n:=\bar C_n e^{-\nu T_n/2}+\varepsilon\le\eta_0\rho^n+\varepsilon$. Iterating the recurrence $e_{n+1}\le\rho\,e_n+\eta_n$ with $e_0$ as base case,
\(
e_n \le \rho^n e_0 + \sum_{k=0}^{n-1}\rho^{n-1-k}\,\eta_k
\le \rho^n e_0 + \eta_0\sum_{k=0}^{n-1}\rho^{n-1} + \varepsilon\sum_{k=0}^{n-1}\rho^{n-1-k}
= \rho^n e_0 + n\eta_0\rho^{n-1} + \frac{\varepsilon(1-\rho^n)}{1-\rho},
\)
which is~\eqref{eq:geometric-tracking}.
Summing in $n\ge 0$: $\sum_n\rho^n=1/(1-\rho)$, $\sum_n n\rho^{n-1}=1/(1-\rho)^2$, so the geometric components contribute $e_0/(1-\rho)+\eta_0/(1-\rho)^2$. The $\varepsilon$-floor contributes $0$ in regime~(I) and $\varepsilon\sum_{n=0}^N (1-\rho^n)/(1-\rho)\le(N+1)\varepsilon/(1-\rho)$ for finite $N$ in regime~(II).
\end{proof}

\subsection{Accumulated Global Error Bounds}
\rev{%
The following theorem is the culminating result of the paper.
Informally, it states that if each inner solve is accurate enough
(controlled by the flow time $T$ and the approximation precision $\varepsilon$),
then the inexact neural iterates stay within distance $\delta$ of the exact MMS iterates
for all time, and converge to an $O(\delta)$-neighborhood of the global minimizer.
The proof is by induction: at each step, the one-step error decomposes into
(i)~a \emph{contraction term} $\frac{\delta}{1+\tau m_\calF}$ from the Lipschitz constant of $J_\tau$,
(ii)~a \emph{finite-time truncation error} $\bar C_n\,e^{-\nu T/2}$ (with the $n$-indexed constant $\bar C_n$ of Corollary~\ref{cor:rgf-finite-time}; bounded uniformly in $n$ under Assumption~\ref{assm:traj-unif} by Remark~\ref{rem:Cbar-bound}) from stopping the inner flow early, and
(iii)~a \emph{neural approximation error} $\varepsilon$.
Choosing $T$ large enough makes (ii) negligible, and the condition on $\varepsilon$
ensures that (i)$+$(iii)$<\delta$, closing the induction.
}
\begin{theorem}[Global Error Propagation Bound in MMS]\label{thm:global_error}
Under the standing assumptions ($\calF$ continuously Fr\'echet differentiable and $m_\calF$-strongly convex with $m_\calF>0$, $u_\NN$ with $L$-Lipschitz Jacobian), Assumption~\ref{assm:center-nondeg}, and Assumption~\ref{assm:traj-unif} (Trajectory uniformity), let $\{u^n\}_{n\ge0}$ be the exact MMS iterates
\[
u^{n+1}=J_\tau(u^n),
\]
and let $\{u_{\mathrm{NN},T}^n\}_{n\ge0}$ be the inexact iterates generated by running the Riemannian gradient flow~\eqref{eqn:RGF-update} for time $T=T_n$ chosen at each step as in~\eqref{eqn:T} so that the induction condition of Lemma~\ref{lem:induction} is satisfied. Assume the NN approximation error~\eqref{eqn:approximate-error} holds, and that the initialization satisfies $\|u^0_{\NN,T}-u^0\|_\calH\le\delta$ (ensured whenever the neural network can approximate $u^0$ to accuracy $\delta$ by the universal approximation theorem). Then:
\begin{enumerate}
\item[(a).] (Uniform bound tracking of MMS.)
For all $n\ge0$,
\[
\|u_{\mathrm{NN},T}^n-u^n\|_\mathcal{H} \le \delta .
\]

\item[(b).] (Global bound to the minimizer.)
Let $u^* \in \mathcal{H}$ be the global minimizer of $\calF$. Then for all $n\ge0$,
\[
\|u_{\mathrm{NN},T}^n-u^*\|_\calH
\le \delta + \Big(\frac{1}{1+\tau m_\calF}\Big)^n \|u^0-u^*\|_\calH.
\]
In particular, the inexact iterates converge to an $O(\delta)$-neighborhood of $u^*$.
\end{enumerate}
\end{theorem}

The role of (U1)--(U2) of Assumption~\ref{assm:traj-unif} in this theorem is discussed in Appendix~\ref{app:disc-traj-unif} (Remark~\ref{rem:tracking-uniformity}).

\begin{proof}
Lemma~\ref{lem:induction} shows the induction step:
if $\|u_{\mathrm{NN},T}^n-u^n\|_\calH\le\delta$, then with $T$ chosen as in \eqref{eqn:T},
we also have $\|u_{\mathrm{NN},T}^{n+1}-u^{n+1}\|_\calH\le\delta$.
Thus, provided the initialization satisfies $\|u_{\mathrm{NN},T}^0-u^0\|_\calH\le\delta$, an induction yields (a): 
\[
\|u_{\mathrm{NN},T}^n-u^n\|_\calH\le\delta,\qquad \forall n\ge0.
\]

Now for (b): since $u^*$ minimizes $\calF$, it is also the minimizer of
$u\mapsto \frac1{2\tau}\|u-u^*\|_\calH^2+\calF[u]$, hence $u^*=J_\tau(u^*)$ for any $\tau >0$.
By the Lipschitz continuity of $J_\tau$ (Lemma~\ref{lem:lip_J_tau}),
\[
\|u^{n+1}-u^*\|_\calH
=\|J_\tau(u^n)-J_\tau(u^*)\|_\calH
\le \frac{1}{1+\tau m_\calF}\|u^n-u^*\|_\calH.
\]
Iterating the above inequality gives
\[
\|u^{n}-u^*\|_\calH\le \Big(\frac{1}{1+\tau m_\calF}\Big)^n\|u^0-u^*\|_\calH .
\]

Finally, by the triangle inequality,
\[
\|u_{\mathrm{NN},T}^n-u^*\|_\calH
\le \|u_{\mathrm{NN},T}^n-u^n\|_\calH + \|u^n-u^*\|_\calH
\le \delta + \Big(\frac{1}{1+\tau m_\calF}\Big)^n\|u^0-u^*\|_\calH.
\]
This concludes the proof.
\end{proof}

\rev{%
In practice each inner subproblem is solved not by the continuous-time Riemannian gradient flow but by a finite number of discrete Riemannian gradient descent steps~\eqref{eq:discrete-RGD}. The next theorem is the discrete-time counterpart of Theorem~\ref{thm:global_error}: with $K_n$ inner RGD iterations per outer MMS step (in place of inner-flow time $T_n$), the inexact iterates again satisfy a uniform $\delta$-tracking bound. The structure mirrors Theorem~\ref{thm:global_error} exactly; the only change is that the inner-flow truncation error $\bar C_n e^{-\nu T/2}$ from Corollary~\ref{cor:rgf-finite-time} is replaced by the discrete flow-gap bound $\bar D_n q^{K}$ from Theorem~\ref{thm:discrete-RGD}~\eqref{eq:discrete-RGD-flow-gap}.
}

\begin{theorem}[Global error propagation under discrete RGD inner solver]\label{thm:global_error-RGD}
Adopt the standing assumptions and Assumption~\ref{assm:traj-unif} of Theorem~\ref{thm:global_error}, together with the injectivity-radius hypothesis~\eqref{eq:inj-hypothesis} at every iterate. Let $\{u^n_{\NN,K}\}_{n\ge 0}$ denote the inexact iterates obtained by running the discrete Riemannian gradient descent~\eqref{eq:discrete-RGD} (via the exponential map) with constant step $0<\eta\le 2/(3\mu)$ for $K=K_n$ steps at each MMS step, starting from $h_0=0\in\calM(u^n_{\NN,K})$. Define
\begin{equation}\label{eq:Dbar-def}
    q := \sqrt{1-\nu\eta/2}\in(0,1),\qquad
    \bar D_n := \sqrt{\tfrac{4\bigl(g(0;u^n_{\NN,K}) - g(h^{(n)}_\infty;u^n_{\NN,K})\bigr)}{\nu}},
\end{equation}
where $h^{(n)}_\infty=\arg\min_{h\in\calM(u^n_{\NN,K})}g(h;u^n_{\NN,K})$. Assume the NN approximation error~\eqref{eqn:approximate-error} of Assumption~\ref{assm:approx-precision} holds, the initialization satisfies $\|u^0_{\NN,K}-u^0\|_\calH\le\delta$, and at each step the inner iteration count $K_n$ is chosen so that
\begin{equation}\label{eqn:Kn}
    K_n \;\ge\; \frac{1}{\log(1/q)}\,\log\!\left(\bar D_n\cdot\frac{1+\tau m_\calF}{\tau m_\calF(\delta-\varepsilon)-\varepsilon}\right).
\end{equation}
Then:
\begin{enumerate}
\item[(a).] (Uniform tracking of MMS.) For all $n\ge 0$,
\[
\|u^n_{\NN,K}-u^n\|_\calH \;\le\; \delta.
\]
\item[(b).] (Global bound to the minimizer.) With $u^*$ the global minimizer of $\calF$ and $\rho:=1/(1+\tau m_\calF)$,
\[
\|u^n_{\NN,K}-u^*\|_\calH \;\le\; \delta + \rho^n\,\|u^0-u^*\|_\calH,\qquad \forall n\ge 0.
\]
\end{enumerate}
The constants $\nu,\mu,\kappa$ in~\eqref{eqn:Kn} are data-dependent (Lemma~\ref{lem:g-properties}); under Assumption~\ref{assm:traj-unif}, $\bar D_n\le \sqrt{4(\sup_n\calF[u^n_{\NN,K}]-\calF_{\inf})/\nu}<\infty$, so $\sup_n K_n<\infty$ and the bounds chain uniformly.
\end{theorem}

\begin{proof}
Let $h^n_K\in\calM(u^n_{\NN,K})$ denote the $K$-th discrete-RGD iterate at outer step $n$, so $u^{n+1}_{\NN,K}=u^n_{\NN,K}+h^n_K$, and let $u^{n+1}_{\NN,\infty}:=u^n_{\NN,K}+h^{(n)}_\infty=J_{\tau,\NN}(u^n_{\NN,K})$. Theorem~\ref{thm:discrete-RGD}~\eqref{eq:discrete-RGD-flow-gap} applied at $v=u^n_{\NN,K}$ gives
\[
\|u^{n+1}_{\NN,K}-u^{n+1}_{\NN,\infty}\|_\calH \;=\; \|h^n_K-h^{(n)}_\infty\|_\calH \;\le\; \bar D_n\,q^{K_n}.
\]
By the triangle inequality, the contraction of $J_\tau$ (Lemma~\ref{lem:lip_J_tau}), and Assumption~\ref{assm:approx-precision},
\begin{equation}\label{eq:error_prop_RGD}
    \|u^{n+1}_{\NN,K}-u^{n+1}\|_\calH \;\le\; \bar D_n\,q^{K_n} + \varepsilon + \rho\,\|u^n_{\NN,K}-u^n\|_\calH,
\end{equation}
exactly mirroring~\eqref{eq:error_prop}. Condition~\eqref{eqn:Kn} is equivalent to $\bar D_n\,q^{K_n}\le \tau m_\calF(\delta-\varepsilon)/(1+\tau m_\calF) -\varepsilon/(1+\tau m_\calF)\cdot(1+\tau m_\calF)/(1+\tau m_\calF) = (1-\rho)\delta-\varepsilon$, where we used $(1-\rho) = \tau m_\calF/(1+\tau m_\calF)$. Hence $\bar D_n\,q^{K_n}+\varepsilon\le (1-\rho)\delta$, and if $\|u^n_{\NN,K}-u^n\|_\calH\le\delta$, then~\eqref{eq:error_prop_RGD} gives
\[
\|u^{n+1}_{\NN,K}-u^{n+1}\|_\calH \;\le\; (1-\rho)\delta + \rho\delta \;=\; \delta.
\]
With the initial bound $\|u^0_{\NN,K}-u^0\|_\calH\le\delta$, an induction yields~(a).

For~(b), the exact MMS iterates contract: $\|u^n-u^*\|_\calH\le\rho^n\|u^0-u^*\|_\calH$ (Lemma~\ref{lem:lip_J_tau} together with $u^*=J_\tau(u^*)$). The triangle inequality then gives
\[
\|u^n_{\NN,K}-u^*\|_\calH \;\le\; \|u^n_{\NN,K}-u^n\|_\calH + \|u^n-u^*\|_\calH \;\le\; \delta + \rho^n\|u^0-u^*\|_\calH.
\]
Finally, $\bar D_n = \sqrt{(4/\nu)(g(0;u^n_{\NN,K})-g(h^{(n)}_\infty;u^n_{\NN,K}))} \le \sqrt{(4/\nu)(\calF[u^n_{\NN,K}]-\calF_{\inf})}$ since $g(0;v)=\calF[v]$ and $g(h^{(n)}_\infty;v)\ge\calF_{\inf}$ by $\calF\ge\calF_{\inf}$. The uniform energy bound (derived uniformity from Assumption~\ref{assm:traj-unif}, Remark~\ref{rem:traj-unif}) yields $\sup_n\bar D_n<\infty$, and hence $\sup_n K_n<\infty$ from~\eqref{eqn:Kn}.
\end{proof}

\rev{%
\begin{remark}[Continuous vs.\ discrete inner solver]\label{rem:discrete-RGD-vs-RGF}
Comparing~\eqref{eqn:Kn} with the continuous threshold~\eqref{eqn:T}, the discrete required iteration count is $K_n \sim (1/\log(1/q))\log(\bar D_n/\delta)$, while the continuous-time horizon is $T_n \sim (2/\nu)\log(\bar C_n/\delta)$. Up to the data-dependent constants $\bar D_n, \bar C_n$, the discrete rate $\log(1/q) = -\tfrac12\log(1-\nu\eta/2) \approx \nu\eta/4$ recovers the continuous rate $\nu/2$ as $\eta\downarrow 0$, consistent with Euler discretization of the flow.
\end{remark}
}

\section{Parameter-Space Convergence}\label{sec:param-space}

Sections~\ref{sec:RGF}--\ref{sec:error-propagation} establish a function-space convergence theory by interpreting the parameter-space preconditioned gradient flow as the Riemannian gradient flow on the increment manifold $\calM(u_\NN(w^n))\subset\calH$ and propagating the inexactness through the neural MMS iterations. We now translate this function-space guarantee back to parameter space. Since neural networks may admit symmetries (e.g., permutations of hidden units), convergence is most naturally stated as convergence of $\{u_\NN(w^n_{\NN,T})\}$ to an $O(\delta)$-neighborhood of $u^*$, rather than as convergence of $\{w^n_{\NN,T}\}$ to a single point.

Figure~\ref{fig:param-space} illustrates the picture: along the inexact trajectory $\{w^n_{\NN,T}\}$, each iterate sits inside its non-degeneracy ball $B(w^n_{\NN,T},r_{w^n_{\NN,T}})$ (Lemma~\ref{lem:gram-lower-bound}); under Assumption~\ref{assm:traj-unif} these radii are uniformly bounded below by $r_\infty := \inf_n r_{w^n_{\NN,T}} > 0$, and the entire trajectory remains confined to a single bounded set in $\R^p$ (Theorem~\ref{thm:global-PGF}(b)). Each parameter step is controlled by the corresponding function-space step via the inverse-stability bound $\|w^{n+1}_{\NN,T}-w^n_{\NN,T}\|_2 \le \lambda_\infty^{-1}\|u^{n+1}_{\NN,T}-u^n_{\NN,T}\|_\calH$ (Corollary~\ref{cor:param-proxy} below).

\begin{figure}[!ht]
\centering
\includegraphics[width=0.85\textwidth]{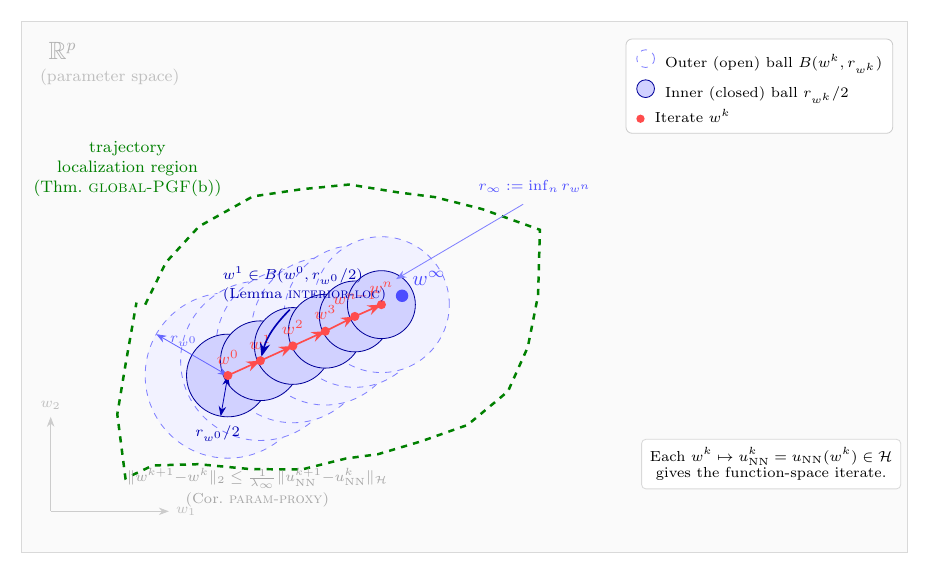}
\caption{Parameter-space picture of the neural MMS trajectory. The iterates $\{w^n_{\NN,T}\}$ (red, solid) sit each inside its own open non-degeneracy ball $B(w^n_{\NN,T}, r_{w^n_{\NN,T}})$ (dashed boundary; Lemma~\ref{lem:gram-lower-bound}) and, under the interior-localization condition (Lemma~\ref{lem:interior-loc}), the next iterate $w^{n+1}_{\NN,T}$ lies strictly inside the closed inner ball of radius $r_{w^n_{\NN,T}}/2$ (solid boundary). Under Assumption~\ref{assm:traj-unif}, the radii are uniformly bounded below by $r_\infty := \inf_n r_{w^n_{\NN,T}}>0$, and the trajectory stays inside a single bounded localization region (green dashed; Theorem~\ref{thm:global-PGF}(b)). Each parameter step is bounded by the corresponding function-space step via $\lambda_\infty^{-1}$ (Corollary~\ref{cor:param-proxy}).}
\label{fig:param-space}
\end{figure}

Let $u^*\in\arg\min_{\calH}\calF$, which is unique by $m_\calF$-strong convexity. We translate the function-space neighborhood convergence into per-step parameter displacement using the local non-degeneracy of the Jacobian.

\begin{corollary}[Per-step parameter displacement bound]\label{cor:param-proxy}
Adopt the standing assumptions in Sections~\ref{sec:RGF}--\ref{sec:error-propagation}. Let $\{w^n_{\NN,T}\}_{n\ge 0}$ be the inexact parameter iterates and $u^n_{\NN,T}:=u_\NN(w^n_{\NN,T})$. Suppose $u_\NN$ has a $\lambda_{w^n_{\NN,T}}$-non-degenerate Jacobian at $w^n_{\NN,T}$ and the next iterate satisfies $w^{n+1}_{\NN,T}\in B(w^n_{\NN,T},r_{w^n_{\NN,T}})$, with $r_{w^n_{\NN,T}}$ as in Lemma~\ref{lem:gram-lower-bound}. Then
\begin{equation}\label{eq:per-step-param-bound}
\|w^{n+1}_{\NN,T}-w^n_{\NN,T}\|_2
\;\le\; \frac{1}{\lambda_{w^n_{\NN,T}}}\,\|u^{n+1}_{\NN,T}-u^n_{\NN,T}\|_\calH,
\qquad n\ge 0.
\end{equation}
Under Assumption~\ref{assm:traj-unif}~(U1) and the derived uniform local-radius bound (Remark~\ref{rem:traj-unif}), $\inf_n r_{w^n_{\NN,T}}\ge r_\infty>0$, so the bound chains uniformly along the trajectory.
\end{corollary}

\begin{proof}
The bi-Lipschitz bound established in the proof of Proposition~\ref{prop:manifold} gives
$\|u_\NN(w)-u_\NN(w')\|_\calH \ge (2\lambda_{w^n_{\NN,T}}-L r_{w^n_{\NN,T}})\|w-w'\|_2 \ge \lambda_{w^n_{\NN,T}}\|w-w'\|_2$
for all $w,w'\in B(w^n_{\NN,T},r_{w^n_{\NN,T}})$, where the last step uses $r_{w^n_{\NN,T}}\le \lambda_{w^n_{\NN,T}}/L$ from~\eqref{eq:rwdef}. Applying this with $w=w^{n+1}_{\NN,T}$ and $w'=w^n_{\NN,T}$ yields~\eqref{eq:per-step-param-bound}.
\end{proof}

Corollary~\ref{cor:param-proxy} bypasses the nonconvexity of the parameter-space objective by leveraging the function-space Riemannian gradient flow on the geodesically strongly convex increment manifold (Sections~\ref{sec:conv}--\ref{sec:error-propagation}) and translating its convergence into a per-step parameter displacement via the local non-degeneracy of $\D u_\NN$ alone---no global bi-Lipschitz constant is required.

We next combine this per-step bound with the function-space convergence (Theorem~\ref{thm:global_error}) to obtain a global convergence theorem in parameter space. The key idea: exponential decay of $\|u^n_{\NN,T}-u^*\|_\calH$ yields a summable total parameter trajectory length, which (i) ensures the iterates form a Cauchy sequence in $\R^p$, and (ii) bounds the depletion of the non-degeneracy budget along the trajectory.

\begin{lemma}[Non-degeneracy decay along the trajectory]\label{lem:lambda-decay}
Suppose $u_\NN$ has $L$-Lipschitz Jacobian and a $\lambda_{w^0}$-non-degenerate Jacobian
at $w^0$. Let $\{w^n_{\NN,T}\}_{n\ge 0}$ be the inexact parameter iterates and write
$\Delta_n := \|w^{n+1}_{\NN,T}-w^n_{\NN,T}\|_2$. As long as
\begin{equation}\label{eq:budget}
L\sum_{k=0}^{n-1}\Delta_k < 2\lambda_{w^0},
\end{equation}
the Jacobian at $w^n_{\NN,T}$ remains non-degenerate, with
\begin{equation*}
\sigma_{\min}\!\big(\D u_\NN(w^n_{\NN,T})\big)
\ \ge\ 2\lambda_{w^0} - L\sum_{k=0}^{n-1}\Delta_k\ =:\ 2\lambda_n,
\end{equation*}
so $u_\NN$ is $\lambda_n$-non-degenerate at $w^n_{\NN,T}$ (in the sense of Definition~\ref{def:non-degen-jacobian}).
\begin{remark}[$\lambda_n$ as Weyl proxy vs.\ actual non-degeneracy]\label{rem:lambda-decay}
The quantity $\lambda_n$ above is a \emph{Weyl-perturbation lower bound} on the actual non-degeneracy constant $\lambda_{w^n_{\NN,T}}$ of Assumption~\ref{assm:center-nondeg}; in particular, $\lambda_n\le\lambda_{w^n_{\NN,T}}$. The lemma certifies $\lambda_n>0$ step-wise as long as~\eqref{eq:budget} holds up to step $n$, but does \emph{not} by itself guarantee $\inf_{n\ge 0}\lambda_n>0$ for the entire infinite trajectory. The uniform positivity $\inf_n\lambda_{w^n_{\NN,T}}\ge\lambda_\infty>0$ is the content of Assumption~\ref{assm:traj-unif}(U1); under the initial-gap condition of Theorem~\ref{thm:global-PGF}(iii), the proof of Theorem~\ref{thm:global-PGF}(b) establishes both that~\eqref{eq:budget} holds for all $n$ and that $\inf_n\lambda_n\ge\lambda_\infty$.
\end{remark}
\end{lemma}

\begin{proof}
By induction on $n$. At $n=0$, $\sigma_{\min}(\D u_\NN(w^0))\ge 2\lambda_{w^0}$ by hypothesis.
Suppose the bound holds at step $n$. By the $L$-Lipschitz continuity of $\D u_\NN$ and the Weyl-type inequality
$|\sigma_{\min}(A)-\sigma_{\min}(B)|\le \|A-B\|_{\rm op}$,
\[
\sigma_{\min}\!\big(\D u_\NN(w^{n+1}_{\NN,T})\big)
\ \ge\ \sigma_{\min}\!\big(\D u_\NN(w^n_{\NN,T})\big) - L\Delta_n
\ \ge\ 2\lambda_{w^0} - L\sum_{k=0}^n\Delta_k\ =\ 2\lambda_{n+1},
\]
which is positive by \eqref{eq:budget}.
\end{proof}

We are now ready to state the main convergence theorem in parameter space.

\begin{theorem}[Global convergence of the preconditioned gradient flow in parameter space]
\label{thm:global-PGF}
Assume the standing assumptions of Sections~\ref{sec:RGF}--\ref{sec:error-propagation}, and let
$u^*\in\arg\min_{\calH}\calF$ be the unique global minimizer of $\calF$ ($m_\calF>0$).
Suppose:
\begin{enumerate}
\item[(i)] $u_\NN$ has an $L$-Lipschitz Jacobian and a $\lambda_{w^0}$-non-degenerate Jacobian at $w^0$;
\item[(ii)] the inner-flow time $T=T_n$ is chosen to grow with $n$ according to~\eqref{eqn:T-growing} of Corollary~\ref{cor:geometric-tracking} (with some fixed $\eta_0\in(0,(1-\rho)\delta-\varepsilon)$), so that the function-space tracking error~\eqref{eq:geometric-tracking} decays geometrically. The chained tracking bound of Theorem~\ref{thm:global_error}(a) then applies under Assumption~\ref{assm:traj-unif};
\item[(iii)] writing $\rho:=1/(1+\tau m_\calF)$, $\lambda_\infty>0$ for a chosen lower bound on the asymptotic non-degeneracy constant (the same $\lambda_\infty$ in Assumption~\ref{assm:traj-unif}(U1)), $r_\infty:=\inf_{n\ge 0}r_{w^n_{\NN,T}}>0$ (the derived uniform local radius from Remark~\ref{rem:traj-unif}), and the trajectory budget
\begin{equation}\label{eq:S-def}
S \;:=\; \frac{2+\tau m_\calF}{\tau m_\calF}\,\|u^0-u^*\|_\calH \;+\; \frac{2\delta}{1-\rho} \;+\; \frac{2\eta_0}{(1-\rho)^2},
\end{equation}
the initial function-space gap is small enough that
\begin{equation}\label{eq:initial-gap-condition}
S \;<\; \min\!\left\{\frac{2(\lambda_{w^0}-\lambda_\infty)\,\lambda_\infty}{L},\;\;\lambda_\infty\,r_\infty\right\}.
\end{equation}
The $\eta_0$-term in~\eqref{eq:S-def} vanishes in the limit $\eta_0\to 0^+$ (perfect-inner-solver limit); the simpler condition $\frac{2+\tau m_\calF}{\tau m_\calF}\|u^0-u^*\|_\calH+\frac{2\delta}{1-\rho}<\min\{2(\lambda_{w^0}-\lambda_\infty)\lambda_\infty/L,\,\lambda_\infty r_\infty\}$ then suffices for all sufficiently small $\eta_0$.
\end{enumerate}

The apparent circularity in hypothesis~(ii) -- Theorem~\ref{thm:global_error} requires (U1)--(U2), and the present theorem's conclusion (b) supplies them -- is discussed in Appendix~\ref{app:disc-traj-unif} (Remark~\ref{rem:globalPGF-traj-unif}).

Then the inexact parameter iterates $\{w^n_{\NN,T}\}_{n\ge 0}$ satisfy:
\begin{enumerate}
\item[(a)] (Function-space convergence.) For all $n\ge 0$,
\begin{equation*}
\|u^n_{\NN,T}-u^*\|_\calH \ \le\ \delta + \rho^n\|u^0-u^*\|_\calH .
\end{equation*}

\item[(b)] (Trajectory length is finite.) The total parameter trajectory length is bounded by the trajectory budget $S$:
\begin{equation}\label{eq:trajectory-bound}
\sum_{n=0}^{\infty}\Delta_n \;\le\; \frac{S}{\lambda_\infty} \;<\; \infty.
\end{equation}
In particular, the Jacobian non-degeneracy is preserved along the entire trajectory:
$\sigma_{\min}(\D u_\NN(w^n_{\NN,T}))\ge 2\lambda_\infty>0$ for all $n\ge 0$ (so Assumption~\ref{assm:traj-unif}(U1) holds with the same $\lambda_\infty$).
Moreover, the per-step displacement satisfies $\Delta_n\le r_\infty$ for all $n$,
so Corollary~\ref{cor:param-proxy} applies at every step.

\item[(c)] (Cauchy in parameter space.) $\{w^n_{\NN,T}\}_{n\ge 0}$ is a Cauchy sequence in $\R^p$
and converges to a limit $w^\infty\in\R^p$.

\item[(d)] (Parameter-space neighborhood convergence.) The limit satisfies
\begin{equation*}
\|u_\NN(w^\infty)-u^*\|_\calH \ \le\ \delta.
\end{equation*}
\end{enumerate}
\end{theorem}

\begin{proof}
We prove parts~(b), (a), (c), (d) in that order. The crux is part~(b), a clean stepwise induction that establishes Assumption~\ref{assm:traj-unif}(U1) along the trajectory.

\smallskip
\noindent\textbf{A-priori telescope bound on exact-MMS displacements.}
For the exact MMS iterates of the $m_\calF$-strongly convex functional, the contraction property of the proximal map gives $\|u^{k+1}-u^k\|_\calH\le\rho^k\|u^1-u^0\|_\calH$. Since
$\|u^1-u^0\|_\calH=\|J_\tau(u^0)-u^0\|_\calH\le\|J_\tau(u^0)-u^*\|_\calH+\|u^*-u^0\|_\calH\le(1+\rho)\|u^0-u^*\|_\calH$,
we obtain $\|u^{k+1}-u^k\|_\calH\le\rho^k(1+\rho)\|u^0-u^*\|_\calH$. Summing the geometric series,
\begin{equation}\label{eq:exact-MMS-telescope}
\sum_{k=0}^{\infty}\|u^{k+1}-u^k\|_\calH
\;\le\; \frac{1+\rho}{1-\rho}\|u^0-u^*\|_\calH
\;=\; \frac{2+\tau m_\calF}{\tau m_\calF}\|u^0-u^*\|_\calH .
\end{equation}
The bound~\eqref{eq:exact-MMS-telescope} does \emph{not} depend on the inexact iterates and is the a-priori telescope that closes the induction below.

Recall the trajectory budget $S$ from~\eqref{eq:S-def}; the initial-gap condition reads $S<\min\{2(\lambda_{w^0}-\lambda_\infty)\lambda_\infty/L,\,\lambda_\infty r_\infty\}$.

\smallskip
\noindent\textbf{(b) Stepwise induction.} We show, by induction on $n\ge 0$, the joint statement
\[
(\mathrm{IH}_n):\quad \lambda_k\ge\lambda_\infty \ \text{for all $0\le k\le n$,} \quad \Delta_k\le r_\infty \ \text{for all $0\le k\le n-1$.}
\]

\emph{Base case} ($n=0$): $\lambda_0=\lambda_{w^0}\ge\lambda_\infty$ by hypothesis~(i); the bound on $\Delta_k$ is vacuous.

\emph{Inductive step.} Assume $(\mathrm{IH}_n)$. By hypothesis (ii) and the inductive hypothesis (U1 holding up to step $n$, U2 along the trajectory up to step $n$ via Remark~\ref{rem:traj-unif} under the initial-gap condition), Lemma~\ref{lem:induction} applies at each step $k\le n$ and yields the function-space tracking $\|u^k_{\NN,T}-u^k\|_\calH\le\delta$ for $0\le k\le n+1$.

By the triangle inequality and the function-space tracking,
\[
\|u^{n+1}_{\NN,T}-u^n_{\NN,T}\|_\calH \le \|u^{n+1}_{\NN,T}-u^{n+1}\|_\calH+\|u^{n+1}-u^n\|_\calH+\|u^n-u^n_{\NN,T}\|_\calH \le 2\delta+\|u^{n+1}-u^n\|_\calH.
\]
Applying Corollary~\ref{cor:param-proxy} with the Weyl bound $\lambda_n\ge\lambda_\infty$ from $(\mathrm{IH}_n)$,
\begin{equation}\label{eq:per-step-lambda-inf}
\Delta_n \;\le\; \frac{\|u^{n+1}_{\NN,T}-u^n_{\NN,T}\|_\calH}{\lambda_n} \;\le\; \frac{2\delta+\|u^{n+1}-u^n\|_\calH}{\lambda_\infty}.
\end{equation}

\emph{Per-step containment at step $n$.}
Under hypothesis~(ii) (growing $T_n$ from~\eqref{eqn:T-growing}), Corollary~\ref{cor:geometric-tracking} provides the geometric-tracking bound~\eqref{eq:geometric-tracking}, and---working in regime~(I) (perfect approximation $\varepsilon=0$, see Remark~\ref{rem:phase4-regime} below for the $\varepsilon>0$ case)---\eqref{eq:sum-e-bound} gives the summable tracking error
\begin{equation}\label{eq:sum-e-rigorous}
\sum_{k=0}^{\infty} e_k \;\le\; E_\infty,\qquad E_\infty:=\frac{e_0}{1-\rho}+\frac{\eta_0}{(1-\rho)^{2}}.
\end{equation}
Summing~\eqref{eq:per-step-lambda-inf} for $k=0,\ldots,n$ and using $\|u^{k+1}_{\NN,T}-u^k_{\NN,T}\|_\calH\le e_{k+1}+\|u^{k+1}-u^k\|_\calH+e_k$ together with the a-priori telescope~\eqref{eq:exact-MMS-telescope} and the geometric-tracking bound $E_\infty = e_0/(1-\rho)+\eta_0/(1-\rho)^2 \le \delta/(1-\rho)+\eta_0/(1-\rho)^2$,
\begin{equation}\label{eq:partial-sum-Delta}
\sum_{k=0}^{n}\Delta_k \;\le\; \frac{1}{\lambda_\infty}\Bigl(\sum_{k=0}^{n}(e_k+e_{k+1})+\sum_{k=0}^{n}\|u^{k+1}-u^k\|_\calH\Bigr)
\;\le\; \frac{1}{\lambda_\infty}\Bigl[2E_\infty+\frac{2+\tau m_\calF}{\tau m_\calF}\|u^0-u^*\|_\calH\Bigr]
\;\le\; \frac{S}{\lambda_\infty},
\end{equation}
where $S$ is the trajectory budget from~\eqref{eq:S-def}. In particular, $\Delta_n\le S/\lambda_\infty < r_\infty$ by the second bound in hypothesis~(iii), so $\Delta_n < r_\infty \le r_{w^n_{\NN,T}}$, validating Corollary~\ref{cor:param-proxy}.

\emph{Non-degeneracy at step $n+1$.}
By Lemma~\ref{lem:lambda-decay}, $\lambda_{n+1}=\lambda_{w^0}-(L/2)\sum_{k=0}^{n}\Delta_k\ge\lambda_{w^0}-(L/2)\cdot S/\lambda_\infty>\lambda_\infty$, by the first bound in hypothesis~(iii). Since $\lambda_{n+1}\le\lambda_{w^{n+1}_{\NN,T}}$ (Weyl proxy is a lower bound; Remark~\ref{rem:lambda-decay}), also $\lambda_{w^{n+1}_{\NN,T}}\ge\lambda_{n+1}>\lambda_\infty$.

This proves $(\mathrm{IH}_{n+1})$, closing the induction. Taking $n\to\infty$ in~\eqref{eq:partial-sum-Delta} yields the trajectory-length bound~\eqref{eq:trajectory-bound}.

\smallskip
\noindent\textbf{(a) Function-space convergence.}
Part~(b) establishes Assumption~\ref{assm:traj-unif}(U1) along the trajectory; Assumption~\ref{assm:traj-unif}(U2) is verified along the way under the initial-gap condition (Remark~\ref{rem:traj-unif} and Remark~\ref{rem:globalPGF-traj-unif}). With Assumption~\ref{assm:traj-unif} established, Theorem~\ref{thm:global_error}(b) gives
\(
\|u^n_{\NN,T}-u^*\|_\calH\le\delta+\rho^n\|u^0-u^*\|_\calH.
\)

\smallskip
\noindent\textbf{(c) Cauchy in parameter space.}
Since $\{\Delta_n\}$ is summable by part~(b), for any $\epsilon>0$ there exists $N$ such that $\sum_{n\ge N}\Delta_n<\epsilon$. Therefore $\|w^m_{\NN,T}-w^n_{\NN,T}\|_2\le\sum_{k=n}^{m-1}\Delta_k<\epsilon$ for all $m>n\ge N$, so $\{w^n_{\NN,T}\}$ is Cauchy in $\R^p$ and converges to a limit $w^\infty\in\R^p$.

\smallskip
\noindent\textbf{(d) Parameter-space neighborhood convergence.}
By continuity of $u_\NN$ and part~(a),
\[
\|u_\NN(w^\infty)-u^*\|_\calH \;=\; \lim_{n\to\infty}\|u^n_{\NN,T}-u^*\|_\calH \;\le\;\delta+\lim_{n\to\infty}\rho^n\|u^0-u^*\|_\calH \;=\;\delta. \qedhere
\]
\end{proof}

Discussion of the initial-gap condition~\eqref{eq:initial-gap-condition} -- its geometric interpretation, the validity regimes of the parameter-space conclusions in the perfect- and uniform-approximation cases, and its connection to ``warm start'' practice -- is given in Appendix~\ref{app:disc-initial-gap} (Remarks~\ref{rem:initial-gap}, \ref{rem:phase4-regime}, and~\ref{rem:warm-start}).

\begin{remark}[Practical relevance and numerical evidence]\label{rem:practical-jacobian}
Theorem~\ref{thm:global-PGF} provides a \emph{worst-case} guarantee in which the
non-degeneracy budget shrinks monotonically as iterations proceed.
However, this pessimistic scenario need not arise in practice:
the actual non-degeneracy constant may stay roughly constant (or even improve)
along the optimization trajectory, well above the worst-case lower bound.

The numerical experiments (see Figure~\ref{fig:app_Jacobian}) monitor the minimum singular value $s_{\min}(J(w))$
throughout the MMS iterations and provide supporting evidence:
\begin{itemize}
\item For the Gauss--Newton (GN) inner solver, $s_{\min}$ remains bounded away from zero
at an $O(10^{-2})$ level throughout the iterations, indicating that the Jacobian
non-degeneracy is \emph{maintained} along the GN trajectory---far better than the
worst-case decay predicted by Lemma~\ref{lem:lambda-decay}.
\item In contrast, Adam and L-BFGS drive $s_{\min}$ to near-machine-precision levels
($\sim 10^{-14}$), indicating rapid loss of Jacobian non-degeneracy.
\end{itemize}
These observations suggest that the preconditioned (Gauss--Newton) gradient flow
naturally preserves the manifold structure that underlies the convergence theory,
making the effective convergence horizon much larger than the worst-case bound.
Understanding why the GN flow preserves Jacobian non-degeneracy---and whether this
can be proved rigorously---is an interesting direction for future work.
\end{remark}

\section{Numerical Experiments}\label{sec:experiments}

The theory developed above establishes a convergence analysis of the
Gauss--Newton-type preconditioned flow as the inner solver for each neural MMS proximal
subproblem: under local non-degeneracy and curvature-control assumptions, the induced
parameter-space dynamics coincides with the Riemannian gradient flow of the increment
objective on the neural increment manifold and can converge to a global minimum of the subproblem. In realistic neural-network training,
however, several of these technical assumptions, such as uniform Jacobian non-degeneracy
and geodesic convexity of the increment manifold, are difficult to verify directly. Therefore,
the goal of this section is not to provide a numerical verification of all assumptions
in the theory. Instead, the experiments are designed to address two complementary goals.
First, we measure how closely the neural MMS iterates track the exact Hilbert-space
trajectory (supporting the global error-propagation result of Theorem~\ref{thm:global_error}).
Second, we examine whether the Gauss--Newton-type inner solver is an effective method
for solving the neural MMS subproblems across a range of benchmarks, both in the
underparameterized regime where the theoretical assumptions are expected to hold,
and in the overparameterized regime where Levenberg--Marquardt (LM) regularization is often needed.

We mainly focus on a general nonlinear regression problem in $\mathbb{R}^d$ in numerical experiments. At the population level, we may view nonlinear regressions as an optimization problem in a Hilbert space $\mathcal H=L^2(\mu)$ with energy
\[
\mathcal F[u]=\frac12\|u-f^\ast\|_{L^2(\mu)}^2 = \int_{\mathbb{R}^d} \frac{1}{2} (u(\x) - f^*(\x))^2 \dd \mu \ ,
\] 
where $f^*: \mathbb{R}^d \to \mathbb{R}$ is the target function, $\mu(\x)$ is the distribution of data. Given training data $\{(\x_i,y_i)\}_{i=1}^N \subset \mathbb{R}^d\times\mathbb{R}$ with $y_i=f^\ast(\x_i)$, $\mu$ is replaced by the empirical measure \( \mu_N=\frac1N\sum_{i=1}^N \delta_{x_i}\). Let $u_{w}:\mathbb{R}^d\to\mathbb{R}$ denote a neural network parameterized by $w\in\mathbb{R}^p$, the goal is to determine parameters $w$ that minimize the empirical risk
\begin{equation}
    \mathcal F[u_w]
    =
    \frac{1}{2N}\sum_{i=1}^N \bigl(u_w(\x_i)-y_i\bigr)^2.
\end{equation}
Let $u^n(\cdot)=u_{w^n}(\cdot)$ denote the neural network function at time step $n$. The minimizing movement scheme updates the parameters by solving the optimization problem
\begin{equation}
    w^{n+1}
    =
    \argmin_{w\in\mathbb{R}^p} \mathcal J^n(w),
\end{equation}
at each step, where
\begin{equation}\label{J_Reg}
    \mathcal J^n(w)
    =
    \frac{1}{2\tau N}\sum_{i=1}^N \bigl(u_w(\x_i)-u^n(\x_i)\bigr)^2
    +
    \frac{1}{2N}\sum_{i=1}^N \bigl(u_w(\x_i)-y_i\bigr)^2.
\end{equation}
Here, $\tau>0$ is the temporal stepsize.

\paragraph{Gauss--Newton flow and its discretization} We first define the  Gauss-Newton flow in this specific setting.  Let $\mathbf{u}(w) = (u_{w}(\x_1), \dots, u_{w}(\x_N))^\top \in \mathbb{R}^{N}$  and the target vector $\mathbf{y} = (y_1, \dots, y_N)^\top$.
We introduce the auxiliary mapping $\Phi(w): \mathbb{R}^p \rightarrow \mathbb{R}^N$ defined as
\begin{equation}
 \Phi(w) =  \mathbf{u}(w) - \mathbf{u}^n .
\end{equation}
The MMS objective \eqref{J_Reg} can be rewritten as a composite function $g^n(\Phi(w))$, where $g^n: \mathbb{R}^N \to \mathbb{R}$ is the strictly convex quadratic function:
\begin{equation}\label{def_g}
  g^n (\mathbf{z}) = \frac{1}{2 \tau  N}  \| \mathbf{z} \|^2 + \frac{1}{2 N} \left \|   \mathbf{z}  + \mathbf{u}^n - \mathbf{y}  \right \|^2.
\end{equation}
The continuous Gauss--Newton flow for minimizing this composite objective is defined by
\begin{equation}\label{eq:GN-flow}
    \frac{\mathrm{d}w}{\mathrm{d}t}
    =
    - \left(
        \mathrm{D}\Phi(w)^{\!\top} \nabla^2_{\mathbf{z}} g^n
        \mathrm{D}\Phi(w)
    \right)^{-1} \nabla_{w} \mathcal{J}(w),
\end{equation}
where $\mathrm{D}\Phi(w) = J(w)$ is the standard Jacobian matrix $J(w) \in \mathbb{R}^{N \times p}$ with entries $J_{ij} = \partial u_{w}(x_i) / \partial w_j$. The Hessian of the convex function $g^n$ with respect to $\mathbf{z}$ is
\begin{equation}
\nabla^2_{\mathbf{z}} g^n =  \left( \frac{1}{\tau  N} + \frac{1}{ N} \right) I_N = \frac{1}{ N} \left(1 + \frac{1}{\tau} \right) I_N.
\end{equation}
Substituting this into \eqref{eq:GN-flow}, the approximate Hessian matrix becomes:
\begin{equation}
    H_{\text{GN}} = J^\top \left[ \frac{1}{ N} \left(1 + \frac{1}{\tau} \right) I \right] J
    = \frac{1}{N} \left(1 + \frac{1}{\tau} \right) J^\top J.
\end{equation}
The gradient of the objective $\mathcal{J}(w)$ is given by:
\begin{equation}
    \nabla_w \mathcal{J}(w) = J^\top \left[ \frac{1}{\tau N}(\mathbf{u} - \mathbf{u}^n) + \frac{1}{N}(\mathbf{u} - \mathbf{y}) \right].
\end{equation}
Combining these, the Gauss--Newton flow simplifies to:
\begin{equation}\label{GN_for_u_final}
    \frac{\mathrm{d}w}{\mathrm{d}t}
    = - \left( \left(1 + \frac{1}{\tau} \right) J^\top J \right)^{-1} J^\top \left( \frac{1}{\tau}(\mathbf{u} - \mathbf{u}^n) + (\mathbf{u} - \mathbf{y}) \right).
\end{equation}
This flow is used in all numerical experiments below.

\begin{remark}
The theoretical analysis in Sections~2--4 is based on the
preconditioned flow
\begin{equation*}
    \frac{dw}{dt}
    =
    -
    \Bigl(
        \mathrm{D}\Phi^n(w)^{\!\top}
        \mathrm{D}\Phi^n(w)
    \Bigr)^{-1} \nabla_{w} g^n\bigl(\Phi^n(w)\bigr),
    \qquad
    w(0) = w^n,
\end{equation*}
which uses $(D\Phi^\top D\Phi)^{-1}$ as preconditioner rather than
incorporating $\nabla^2_z g^n$ explicitly.
Since $\mathrm{D}\Phi^n(w) = J (w)$ and
$\nabla_{w} g^n(\Phi(w)) =  (D \Phi^n)^{\!\top} \nabla g^n$,
the flow can be written as
\begin{equation}\label{GN_for_u}
\begin{aligned}
\frac{\dd w}{\dd t} & = - (J^{\!\top} J)^{-1}  J^{\!\top} \left( \frac{1}{\tau  N} \Phi(w) + \frac{1}{N}  \left( \Phi(w) + {\bm u}^n - {\bm y} \right) \right) \\
& = - ( J^{\!\top} J)^{-1}  J^{\!\top} \left( \frac{1}{\tau N} ({\bm u}(w) - {\bm u}^n) + \frac{1}{N}  \left( {\bm u}(w) - {\bm y} \right) \right)\ .
\end{aligned}
\end{equation}
Comparing \eqref{GN_for_u} with \eqref{GN_for_u_final}, the two
flows differ only by the scalar factor
$\frac{1}{1+1/\tau}$ in the preconditioner, which amounts to a reparametrization of time and leaves the parameter-space trajectory unchanged.
\end{remark}

Although the theoretical framework is based on the continuous Riemannian gradient flow, obtaining the accurate solution for Eq. \eqref{GN_for_u_final} is computationally prohibitive for deep neural networks due to the high dimensionality of the Jacobian. To balance computational efficiency with the benefits of curvature information, we discretize the flow using an explicit Euler method. Rather than fixing $\delta t$, we adaptively select the step size via a monotone backtracking line search combined with a trust-region-style norm clip on the search direction (contraction factor $0.5$, at most $8$ backtracking steps, and maximum step norm $5.0$ unless stated otherwise). The underlying linear system is solved using the conjugate gradient (CG) method, limited to a maximum of 30 iterations. This combination balances the efficiency of an inexact linear solve with the robustness afforded by adaptive step sizes.

The theoretical framework relies on the non-degeneracy of the Jacobian at the initial condition along the MMS trajectory. This assumption guarantees that the pullback metric $G(w)=D u_{\rm NN}(w)^\ast D u_{\rm NN}(w)$ remains strictly positive definite, allowing the Gauss--Newton dynamics to be interpreted as a Riemannian gradient flow. However, as noted by \cite{cayci2025riemannian}, in the over-parameterized regime $(p>N)$. even if the non-degeneracy condition $\lambda_{w_0}>0$ holds generically, it is often very small, leading to severe ill-conditioning of $J^{\rm T} J$. This practical reality necessitates the use of LM regularization to stabilize the GN updates. Because LM damping modifies the effective metric (replacing $J^\top J$ with $J^\top J+\rho I$), the resulting regularized dynamics diverge from the exact Riemannian gradient flow of the original pullback metric. To systematically address this gap between our idealized theoretical analysis and practical over-parameterized landscapes, we consider both the underparameterized and overparameterized regimes. In the latter regime, LM regularization may be
applied to stabilize the GN updates and the goal is to demonstrate stability and performance in practice.

To evaluate the practical performance of Gauss--Newton as a solver for the MMS subproblems, we compare it with several widely used optimization algorithms, namely Adam \citep{kingma2014adam}, AdamW \citep{loshchilov2017decoupled}, L-BFGS \citep{nocedal2006numerical}, Muon \citep{jordan2024muon}, Shampoo \citep{gupta2018shampoo}, and SOAP \citep{vyas2024soap}. In all experiments, these methods are employed solely as inner solvers within the same MMS time-marching framework, so that the outer discretization is kept fixed and the comparison isolates the effect of the subproblem solver itself. For Adam, AdamW, Muon, Shampoo , and SOAP, we perform $500$ inner iterations for each MMS subproblem unless state differently. In this experiment, the learning rates for Adam, AdamW, Muon, Shampoo, and SOAP are all set to $10^{-3}$ unless state differently; AdamW additionally uses weight decay \(10^{-2}\). L-BFGS is implemented with strong Wolfe line search, with maximum inner iterations set to $100$ and history size set to $50$. For the adaptive second-order variants, Shampoo uses momentum $0.9$, numerical regularization parameter $\epsilon=10^{-6}$, and preconditioner update frequency $1$, while SOAP uses Adam-type momentum parameters $(0.9,0.999)$, $\epsilon=10^{-8}$, Shampoo decay parameter $0.95$, and eigenspace update frequency $10$. Since the effective numerical scale and conditioning vary across experiments, the learning rates are adjusted slightly from one experiment to another; however, the overall comparison protocol is kept unchanged throughout.

\subsection{Error accumulation in MMS steps}

In this subsection, we numerically track the error at each MMS step at the trajectory level, which corresponds to the error analysis in Theorem~\ref{thm:global_error}. A solver that resolves the neural proximal subproblem more
accurately produces a smaller per-step error, which in turn leads to tighter tracking
under the contractive outer dynamics.  We compare GN against several first-order and
quasi-Newton baselines on this \emph{trajectory-tracking} diagnostic, which is distinct
from the standard measure of final training energy.

\paragraph{Exact MMS reference for nonlinear regression.}
For the quadratic energy
\[
    F[u] = \frac12 \|u-f^\ast\|^2,
\]
the exact MMS update in the ambient Hilbert space is available in closed form:
\begin{equation}
    u^{n+1}_{\mathrm{exact}}
    =
    \frac{u^n_{\mathrm{exact}}+\tau f^\ast}{1+\tau},
    \qquad
    u^0_{\mathrm{exact}}=u^0_{\mathrm{prescribed}} .
    \label{eq:mms-exact}
\end{equation}
Equivalently,
\begin{equation}
    u^n_{\mathrm{exact}}
    =
    (1+\tau)^{-n} u^0_{\mathrm{prescribed}}
    +
    \Big(1-(1+\tau)^{-n}\Big) f^\ast .
    \label{eq:mms-exact-closed}
\end{equation}
This closed-form trajectory provides a direct functional-space reference against which the neural MMS trajectory can be compared. This setup directly instantiates the scenario depicted in Figure~\ref{fig:function-space}: the exact MMS iterates $\{u^n_{\mathrm{exact}}\}$ converge to $f^*$ in $\mathcal{H}$, while Theorem~\ref{thm:global_error}(a) predicts that the neural iterates $\{u^n_{\mathrm{NN}}\}$ remain within a uniform $\delta$-band of the exact trajectory at every step.  The key quantity $\delta$ is controlled by how accurately each inner solver resolves the proximal subproblem: a solver that minimizes the subproblem more precisely produces a smaller effective $\delta$, and hence tighter tracking of the exact MMS trajectory.

Let \(u_{\mathrm{NN}}^n\) denote the neural iterate obtained by approximately solving the neural proximal subproblem at step \(n\), and define the tracking error
\begin{equation}
    e^n := u_{\mathrm{NN}}^n - u_{\mathrm{exact}}^n .
\end{equation}
In the experiments, this error is evaluated on the same empirical grid used to define the discrete \(L^2\) inner product. We consider a one-dimensional deterministic regression problem on \([-1,1]\) for which the exact MMS trajectory is available in closed form. The target function is
\[
    f^\ast(x)
    =
    x^2
    +0.30\sin(2\pi x)
    +0.20\cos(3\pi x),
\]
and the prescribed initial condition is chosen as
\[
    u^0_{\rm prescribed}(x)=x^2 .
\]
The training grid consists of \(256\) uniformly spaced points on \([-1,1]\). We first pretrain a neural network to approximate the initial condition.  The primary diagnostic is the training-grid discrete \(L^2\) error
\(
    \|u^n_{\rm NN}-u^n_{\rm exact}\|_N = \sqrt{\tfrac{1}{N} \textstyle \sum_{i=1}^N |u^n_{\rm NN}(x_i)-u^n_{\rm exact}(x_i)|^2},
\)
where \(\|\cdot\|_N\) denotes the discrete norm over the \(256\) training points.

We consider three values of the MMS step size,
\(\tau\in\{0.1,\,0.01,\,0.001\}\); results are shown in Figure~\ref{fig:mms-tracking}. The hyperparameters of all inner solvers are kept fixed across these runs, except for the Muon learning rate, which is set to $10^{-3}$ for $\tau=0.1$, $10^{-4}$ for $\tau=0.01$, and $10^{-5}$ for $\tau=0.001$ to ensure stability.
The top row of each column reports the tracking error
\(\|u^n_{\rm NN}-u^n_{\rm exact}\|_N\), while the bottom row reports the corresponding
training energy.  The two rows together make it possible to assess whether an inner
solver that performs well on energy also tracks the functional-space trajectory
accurately.

\begin{figure}[!ht]
    \centering
    \includegraphics[width=\linewidth]{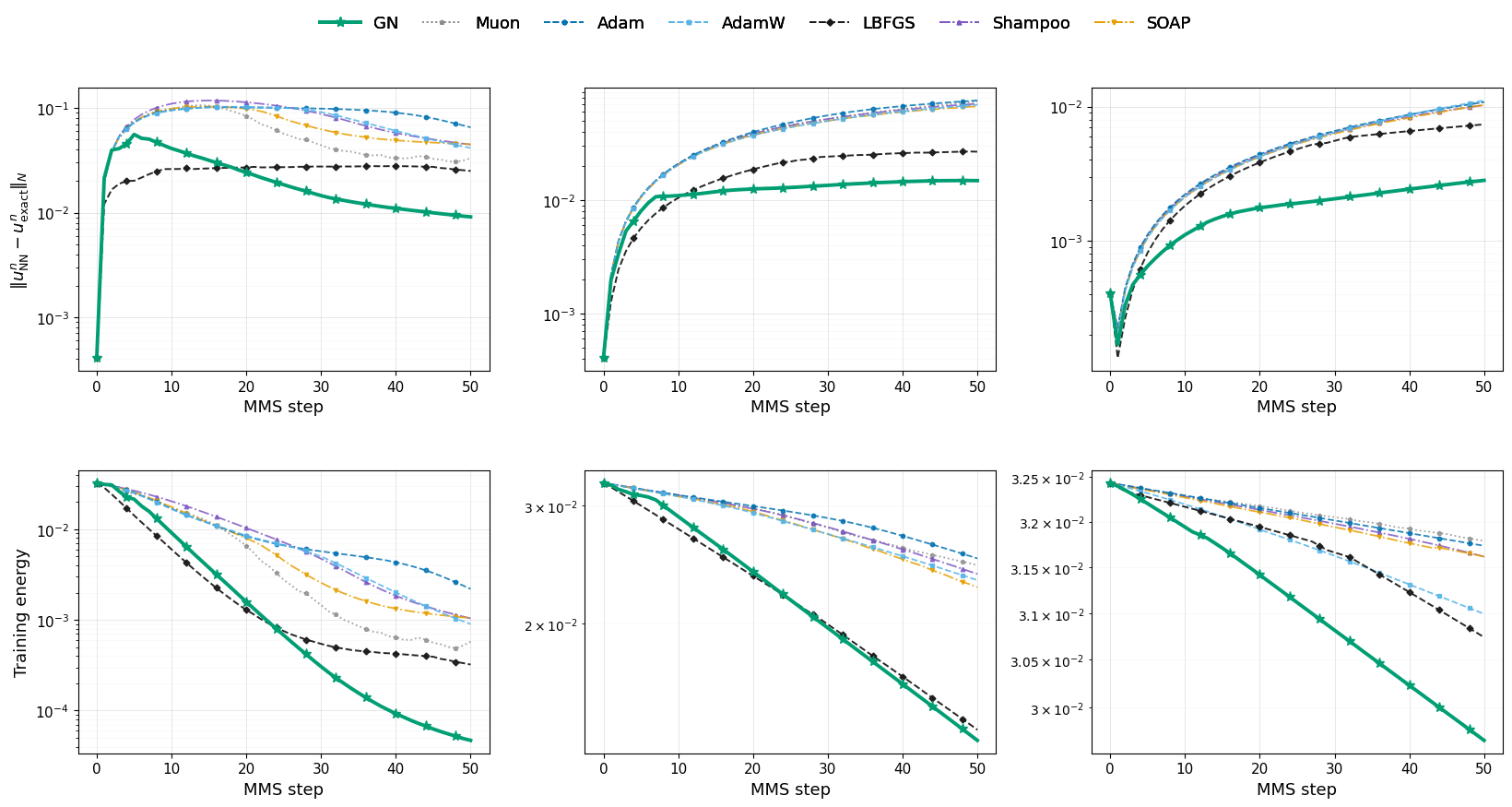}
    \caption{\textbf{Top row:} Stepwise tracking error $\|u^n_{\rm NN}-u^n_{\rm exact}\|_N$ between the neural MMS trajectory and the exact Hilbert-space MMS trajectory, for step sizes $\tau=0.1,\,0.01,\,0.001$ (left to right). GN (green) achieves the smallest tracking error across all~$\tau$, lying uniformly below all first-order and quasi-Newton baselines, particularly for $\tau=0.01$ and $\tau=0.001$. \textbf{Bottom row:} Corresponding training energy. Several baselines decrease energy at rates comparable to GN, yet exhibit substantially larger tracking errors (top row), illustrating that energy decrease alone does not characterize a structure-compatible inner solver. }
    \label{fig:mms-tracking}
\end{figure}

Figure~\ref{fig:mms-tracking} reveals a clear and consistent separation
between inner solvers on the trajectory-tracking diagnostic (top row).
\textbf{GN achieves the smallest, or among the smallest, tracking errors across all
three values of~$\tau$}, particularly after the first few outer iterations. These observations \emph{partially support} Theorem~\ref{thm:global_error}. Theorem~\ref{thm:global_error}(a) guarantees a uniform tracking bound $\|u^n_{\mathrm{NN}}-u^n\|_\calH\le\delta$ for all $n$, where the bound $\delta$ is controlled by the accuracy with which each inner solver resolves the proximal subproblem (Lemma~\ref{lem:induction}).
The separation is most pronounced for $\tau=0.01$ and $\tau=0.001$ (middle and right
columns), where GN's tracking error lies uniformly below that of every other solver
throughout the entire trajectory.
By contrast, several first-order and quasi-Newton baselines
(Adam, AdamW, Muon, L-BFGS, Shampoo, SOAP) decrease the training energy at rates
comparable to GN (bottom row), yet their neural MMS trajectories deviate substantially
from the exact Hilbert-space MMS trajectory (top row).
This dissociation between energy decrease and trajectory tracking is the central empirical
message: a good optimizer for the subproblem energy is \emph{not} the same as a
structure-compatible inner solver for MMS.  It supports the interpretation of GN as
the structure-compatible choice, whose second-order curvature information allows it to
approximate the exact proximal step more faithfully than gradient-based alternatives.

The dependence on $\tau$ provides a consistency check across operating
regimes.  For $\tau=0.01$ and $\tau=0.001$ (middle and right columns), the exact MMS
trajectory evolves slowly per outer step, and GN maintains the tightest tracking from
the first outer iteration onward, with first-order solvers consistently trailing by a
visible margin throughout the trajectory.  For $\tau=0.1$ (left column), the exact
target moves more rapidly per step; GN exhibits a brief initial transient in which the
tracking error rises during the first few MMS steps, reflecting accumulated error from
the pretraining initialization.  After roughly $10$ outer steps, GN recovers and attains
tracking errors below all other solvers for the remainder of the run.  Across all three
step sizes, smaller~$\tau$ yields a more slowly evolving exact trajectory and smaller
absolute tracking errors for all solvers, consistent with the reduced per-step error
amplification predicted by the theory.

In the remaining experiments, we study the practical performance of GN as an inner solver for MMS across a range of benchmarks, including both underparameterized and overparameterized regimes. We emphasize that the purpose of these experiments is to demonstrate that \emph{MMS with GN as the inner solver constitutes a reliable and competitive practical algorithm}.

\begin{remark}
In is worth mentioning that Theorem~\ref{thm:global_error} does \emph{not} assert that GN achieves the lowest final training energy among all inner solvers; its guarantee is that GN's function-space iterates remain uniformly close to the exact MMS trajectory throughout the optimization (as illustrated in Figure~\ref{fig:function-space}).
Other solvers may temporarily deviate from the exact MMS trajectory yet still converge to a neighborhood of~$u^*$, and may attain competitive or lower final energies in certain settings.
The purpose of the experiments below is therefore to demonstrate that MMS with GN as the inner solver constitutes a reliable and competitive practical algorithm.
Accordingly, we report training energy as the primary optimization metric; test-set generalization errors are included in the summary tables for reference, but lie outside the scope of our theoretical guarantees.
\end{remark}

\subsection{Underparameterized regime}
We consider the underparameterized regime for a two-layer neural network $u_w:\mathbb{R}^d \to\mathbb{R}$,
\begin{equation}
u_{w}(x) = a_0 + \sum_{j=1}^m a_j \,\sigma(w_j \cdot x + b_j), \quad w_j \in \mathbb{R}^d,~b_j \in \mathbb{R},~a_j \in \mathbb{R}
\end{equation}
with $\sigma=\tanh$, trained on samples $\{(x_i,y_i)\}_{i=1}^N$.
We quantify Jacobian non-degeneracy via the smallest eigenvalue of the Gram matrix $J(w)^{\top}J(w)$:
\begin{equation}
\lambda_{\min}\!\big(J(w)^{\top}J(w)\big) > \lambda_0
\end{equation}
for some $\lambda_0>0$. In the underparameterized setting ($p\le N$), it is possible for certain initializations $w_{\rm ini}$ to satisfy $\lambda_{\min}\!\big(J(w_{\rm ini})^{\top}J(w_{\rm ini})\big) > \lambda_0.$
In what follows, we examine how $\lambda_{\min}\!\big(J(w)^{\top}J(w)\big)$ evolves over MMS iterations under different optimization methods. In practice, we monitor the minimum singular value of the Jacobian,
\begin{equation}
s_{\min}(w) := \sigma_{\min}\!\big(J(w)\big),
\end{equation}
which is related to the Gram eigenvalue via
\begin{equation}
\lambda_{\min}\!\big(J(w)^{\top}J(w)\big) = s_{\min}(w)^2.
\end{equation}
We emphasize that $\lambda_{\min}$ (and hence $s_{\min}$) depends not only on the parameters $w$ but also on the sampling set $\{x_i\}_{i=1}^N$. Moreover, the initialization must satisfy certain non-degeneracy conditions; otherwise $s_{\min}(w_{\rm ini})$ can be very small (or numerically indistinguishable from zero) even in the underparameterized regime.

\begin{figure}[!htb]
    \centering
    \begin{overpic}[width=\linewidth]{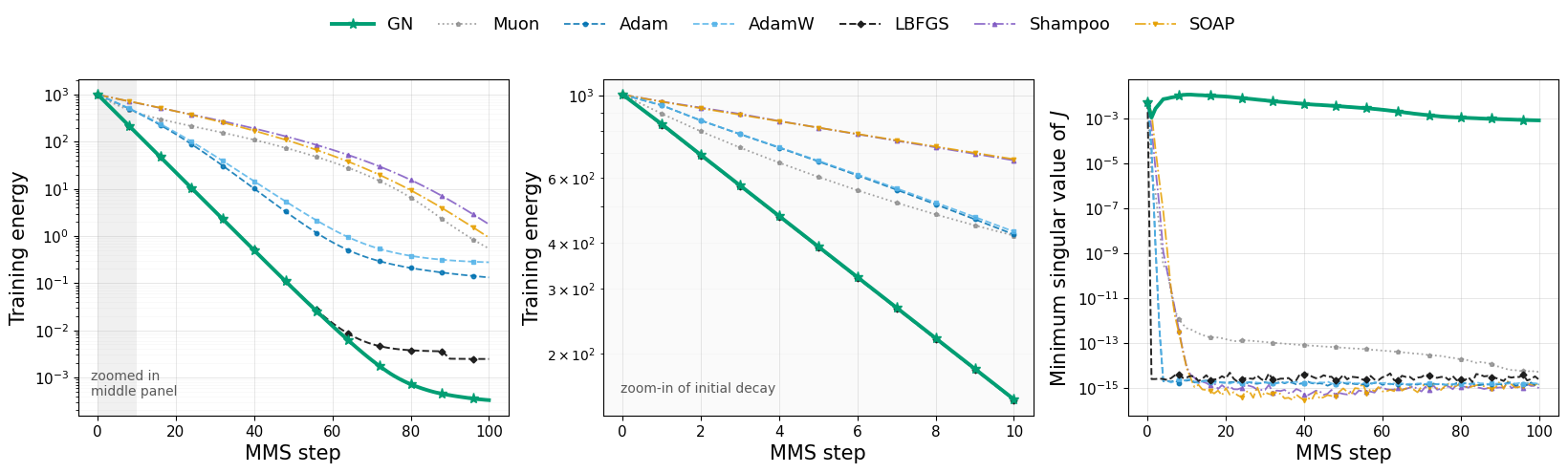}
    \put(-2, 23){{\small (a)}}
    \end{overpic}
    
      \vspace{1em}
    \begin{overpic}[width = \linewidth]{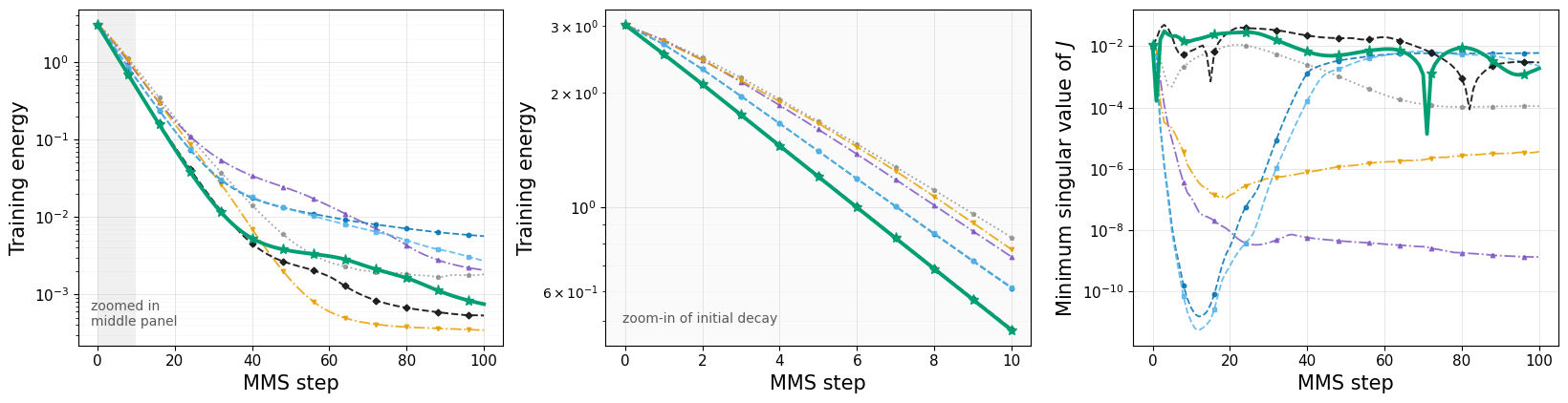}
    \put(-2, 22){{\small (b)}}
    \end{overpic}

    \vspace{1em}
    \begin{overpic}[width = \linewidth]{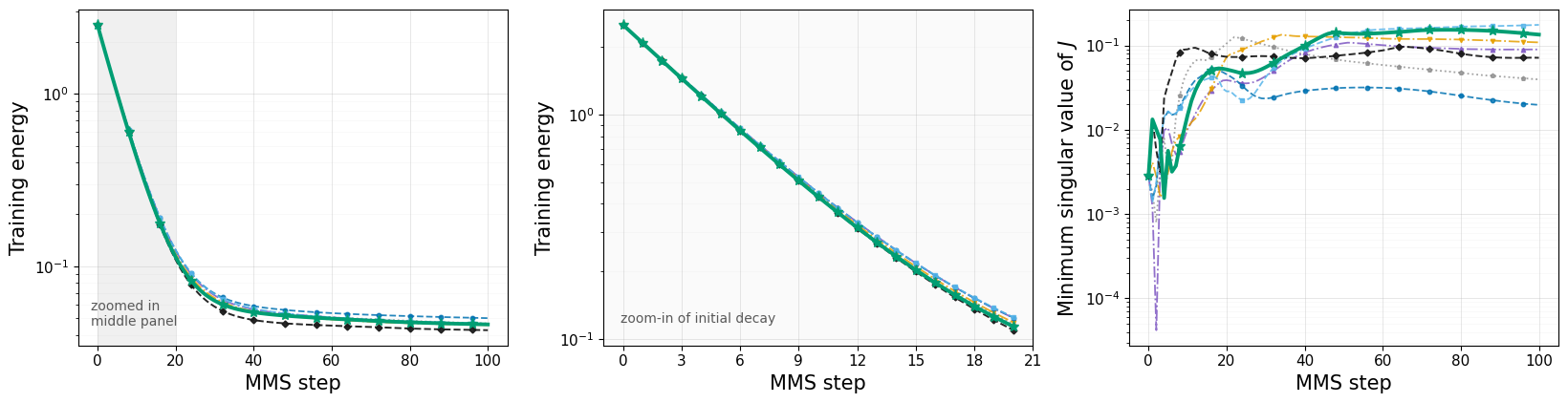}
    \put(-2, 22){{\small (c)}}
    \end{overpic}

    \caption{Numerical results in underparameterized regime with large $s_{\min}(J(w_{\rm ini}))$ with $\tau = 0.1$ for (a) 1D regression with $f^*(x) = x^2$, (b) 10D regression with $f^\ast(\x)  =  \sum_{k=1}^d \cos(\pi x_k)$, and (c) California Housing data. Left: full training-energy decay over MMS steps.  Middle: zoom-in of the early-stage training-energy decay over the first $K$ MMS steps (shaded region in the left panel), highlighting the convergence rate in the initial phase.
Right: minimum singular value of the sampled Jacobian $J$, reported as a diagnostic of Jacobian non-degeneracy along the optimization trajectory.}\label{fig:app_Jacobian}
\end{figure}

We first consider a simple target function $f^*(x)=x^2$ and use a neural network with $m=10$ hidden units, corresponding to a total of $31$ parameters. The training samples $\{x_i\}_{i=1}^N$ are chosen as uniformly spaced grid points on $[-10,10]$ with $N=256$, and a test set of $N_{\mathrm{test}}=1000$ uniformly spaced points on the same interval is used to compute the generalization errors reported in Table~\ref{tab:optimizer-summary}. Since the target function $x^2$ admits many different neural-network representations, different inner solvers may converge to different parameter realizations that produce similarly accurate function approximations. Figure~\ref{fig:app_Jacobian}(a) shows that this is indeed the case numerically. Interestingly, for all methods except GN (i.e., Adam, AdamW, L-BFGS, Muon, Shampoo, and SOAP), $s_{\min}(w)$ rapidly decreases to a near-machine-precision level (around $10^{-14}$) and remains there, indicating that the Jacobian becomes numerically near-degenerate along their optimization trajectories. This collapse of $s_{\min}$ is precisely the failure mode that the non-degeneracy condition of Definition~\ref{def:non-degen-jacobian} guards against. In contrast, the undamped Gauss--Newton (GN) iterates quickly move to a regime in which $s_{\min}(w)$ stays at an $\mathcal{O}(10^{-2})$ level, so that the Jacobian remains well away from numerical singularity. This suggests that all compared methods are able to fit the target function, but they select markedly different parameter representations, with GN systematically favoring a realization whose Jacobian geometry is well-conditioned and consistent with the theoretical assumptions.


Next we consider an example with $d = 10$. In general,  $s_{\min}(w_{\rm ini})$ can be larger for larger $d$ \citep{du2018gradient}. We consider target function 
\begin{equation}\label{HighD_sy}
f^\ast(\x)  =  \textstyle \sum_{k=1}^d \cos(\pi x_k), 
\qquad x=(x_1,\ldots,x_d)\in[-1,1]^d.
\end{equation}
The training inputs are sampled i.i.d.\ uniformly on the hypercube:
\(
x_i^{\mathrm{train}} \overset{\mathrm{i.i.d.}}{\sim}\mathrm{Unif}([-1,1]^d),~i=1,\dots,n_{\mathrm{train}},
\)
with $N_{\mathrm{train}} = 1000$.  The test inputs are independently sampled from the same distribution with size $N_{\mathrm{test}}=1000$. We consider a two layer neural network with $m = 32$, so the total number of parameters is $385$. The results are shown in Figure~\ref{fig:app_Jacobian}(b). Clearly, different optimization methods take different paths in the parameter space to approximate the same function. In particular, the minimum singular value of the Jacobian, $s_{\min}(J)$, evolves very differently across methods. As shown in Table~\ref{tab:optimizer-summary}, GN achieves competitive training energy on this 10D problem, outperforming all first-order gradient-based baselines (Adam, AdamW, Muon, Shampoo), while SOAP and L-BFGS attain slightly lower final energy values. The training-energy curves in Figure~\ref{fig:app_Jacobian}(b) further show that GN reaches low energy levels within the first few MMS steps.

We also consider a real-data regression task using the California Housing dataset \citep{pace1997sparse}. 
The dataset consists of $d=8$ input features and a scalar target (median house value). 
We randomly split the full dataset into training and test sets with an $80\%/20\%$ split using a fixed random seed, and then subsample $N_{\rm train}=1000$ training samples and $N_{\rm test}=1000$ test samples for efficiency and reproducibility. 
Following standard practice, we standardize the input features by fitting a \texttt{StandardScaler} on the training inputs and applying it to both training and test sets. 
We also standardize the targets using a separate scaler fitted on the training targets (and applied to the test targets). 
We use the same two-layer $\tanh$ network with width $m=32$, which gives $p=321$ trainable parameters in this setting. 
The MMS and inner solvers are configured identically to the synthetic example unless otherwise stated. The results are summarized in Figure~\ref{fig:app_Jacobian}(c). As shown in Table~\ref{tab:optimizer-summary}, GN delivers competitive training energy on this real-world tabular dataset, with roughly half the CPU time of Adam. The training-energy zoom-in in Figure~\ref{fig:app_Jacobian}(c) shows that GN makes rapid progress in the early MMS steps, consistent with the fast initial convergence observed in the synthetic experiments.

\begin{table}[t]
\centering
\caption{
Comparison of optimizers across three regression benchmarks.
Each entry reports the final value after the MMS iteration. CPU time is measured in seconds on a MacBook Pro with an Apple M2 chip.}
\label{tab:optimizer-summary}
\resizebox{\linewidth}{!}{%
\begin{tabular}{llccccccc}
\toprule
 & Metric 
& Adam & AdamW & L-BFGS & Muon & Shampoo & SOAP & GN \\
\midrule
\multirow{3}{*}{1D}

&Energy
& $1.33\mathrm{e}{-1}$
& $2.77\mathrm{e}{-1}$
& $2.46\mathrm{e}{-3}$
& $5.45\mathrm{e}{-1}$
& $1.78\mathrm{e}{0}$
& $9.29\mathrm{e}{-1}$
& $3.30\mathrm{e}{-4}$ \\
&Rel. $L^2$
& $1.11\mathrm{e}{-2}$
& $1.57\mathrm{e}{-2}$
& $1.51\mathrm{e}{-3}$
& $2.19\mathrm{e}{-2}$
& $3.99\mathrm{e}{-2}$
& $2.84\mathrm{e}{-2}$
& $5.46\mathrm{e}{-4}$ \\
&Time (s)
& $20.65$
& $21.51$
& $6.55$
& $28.40$
& $25.43$
& $21.90$
& $11.99$ \\
\midrule
\multirow{3}{*}{10D}
& Energy
& $5.64\mathrm{e}{-3}$
& $2.70\mathrm{e}{-3}$
& $5.32\mathrm{e}{-4}$
& $1.81\mathrm{e}{-3}$
& $2.06\mathrm{e}{-3}$
& $3.45\mathrm{e}{-4}$
& $7.48\mathrm{e}{-4}$ \\
& Rel. $L^2$
& $8.95\mathrm{e}{-2}$
& $5.68\mathrm{e}{-2}$
& $2.43\mathrm{e}{-2}$
& $3.46\mathrm{e}{-2}$
& $5.07\mathrm{e}{-2}$
& $1.99\mathrm{e}{-2}$
& $2.90\mathrm{e}{-2}$ \\
& Time (s)
& $47.87$
& $49.28$
& $25.13$
& $65.12$
& $62.56$
& $54.98$
& $25.72$ \\
\midrule
\multirow{3}{*}{Cal. Housing}
& Energy
& $5.01\mathrm{e}{-2}$
& $4.65\mathrm{e}{-2}$
& $4.27\mathrm{e}{-2}$
& $4.68\mathrm{e}{-2}$
& $4.62\mathrm{e}{-2}$
& $4.57\mathrm{e}{-2}$
& $4.61\mathrm{e}{-2}$ \\
& Rel. $L^2$
& $6.10\mathrm{e}{-1}$
& $6.18\mathrm{e}{-1}$
& $6.40\mathrm{e}{-1}$
& $6.48\mathrm{e}{-1}$
& $6.61\mathrm{e}{-1}$
& $6.04\mathrm{e}{-1}$
& $7.03\mathrm{e}{-1}$ \\
& Time (s)
& $46.08$
& $45.22$
& $23.77$
& $54.97$
& $55.51$
& $43.24$
& $24.75$ \\
\bottomrule
\end{tabular}%
}
\end{table}

\subsection{Overparameterized regime}

In this subsection, we test the performance of the Gauss--Newton method for overparameterized two-layer neural networks and deep neural networks. All settings are the same as in the previous 
section. When needed, we apply Levenberg--Marquardt (LM) 
regularization to stabilize the GN updates, replacing the standard 
GN linear system with
\begin{equation}\label{eq:LM-update}
\left( \left(1 + \frac{1}{\tau}\right) J^\top J  + \rho I \right) \eta=J^\top \left( \frac{1}{\tau} (\mathbf{u}^n - \mathbf{u}(w^k)) + (\mathbf{y} - \mathbf{u}(w^k)) \right),
\end{equation}
and update the parameters via $w^{k+1} = w^k + \delta t\, \eta $, where $\delta t > 0$ is the step-size for the explicit Euler discretization and $\rho > 0$ serves as a damping coefficient. 
We choose $\delta t$ using line search as before, unless stated otherwise.
To solve the linear system \eqref{eq:LM-update} efficiently, we utilize the Conjugate Gradient (CG) method. Due to the computational cost, we do not simulate the flow to equilibrium. Instead, we perform only 5 discrete Euler steps.

\subsubsection{Nonlinear Regression}

In overparameterized regime, we first consider a challenging one dimensional regression problem considered in \citet{dahmen2025expansive}, where the target function is defined as
\begin{equation}\label{complex_1d}
f^*(x) = \exp(\sin(k \pi x)) + x^3 - x - 1.0
\end{equation}
with $k = 10$. As discussed in \citet{dahmen2025expansive}, it is difficult to fit this function using a two-layer neural network, even in the over-parameterized regime. So we consider a four-layer fully connected neural network with $\tanh$ activation and 10 neurons per hidden layer, initialized with Xavier normal initialization (gain for $\tanh$) and zero biases. The training and test data are generated in the same way as before. Adam uses learning rate $10^{-3}$ for $500$ iterations

\begin{figure}[!htpb]

     \centering
     \includegraphics[width=0.75\linewidth]{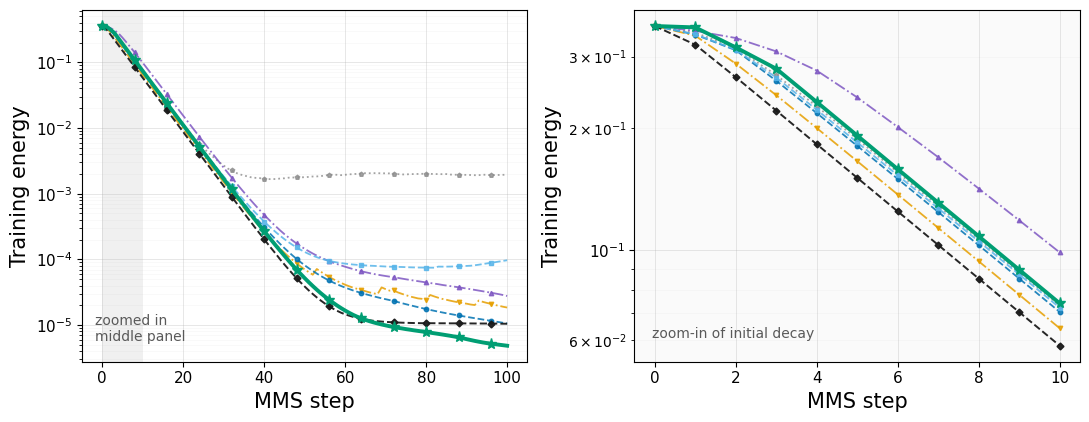}
    \caption{Results for deep network regression on $f^*(x) = \exp(\sin(10\pi x)) + x^3 - x - 1$ in the overparameterized regime ($901$ parameters). Left: full training-energy decay over MMS steps. Right: zoom-in of the training-energy decay during the first $10$ MMS steps (shaded region in the left panel).}
    \label{fig:1d_regression_challenging}
\end{figure}

\begin{table}[t]
\centering
\caption{
Comparison of optimizers on the 1D complex benchmark
(901 parameters, $n_{\text{train}}=256$).
Each entry reports the final value after the MMS iteration. CPU time is measured in seconds on a MacBook Pro with an Apple M2 chip.}
\label{tab:optimizer-1d-complex}
\resizebox{\linewidth}{!}{%
\begin{tabular}{lccccccc}
\toprule
Metric 
& Adam & AdamW & L-BFGS & Muon & Shampoo & SOAP & GN \\
\midrule
Energy
& $1.06\mathrm{e}{-5}$
& $9.70\mathrm{e}{-5}$
& $1.05\mathrm{e}{-5}$
& $1.94\mathrm{e}{-3}$
& $2.77\mathrm{e}{-5}$
& $1.85\mathrm{e}{-5}$
& $4.85\mathrm{e}{-6}$ \\
Rel. $L^2$
& $5.35\mathrm{e}{-3}$
& $1.60\mathrm{e}{-2}$
& $5.33\mathrm{e}{-3}$
& $7.48\mathrm{e}{-2}$
& $8.85\mathrm{e}{-3}$
& $7.33\mathrm{e}{-3}$
& $3.67\mathrm{e}{-3}$ \\
Time (s)
& $44.10$
& $46.05$
& $11.33$
& $75.01$
& $62.96$
& $50.89$
& $23.19$ \\
\bottomrule
\end{tabular}%
}
\end{table}

Figure~\ref{fig:1d_regression_challenging} and Table~\ref{tab:optimizer-1d-complex} show that GN attains the lowest final training energy and relative $L^2$ error among all compared solvers on this challenging benchmark, while remaining competitive in wall-clock time. The zoom-in confirms that GN makes rapid progress in the early MMS steps, demonstrating that it constitutes a highly competitive inner solver even in the overparameterized regime.

Next, we consider the high-dimensional example~\eqref{HighD_sy} with an overparameterized network of width $m=128$ ($1537$ parameters). As shown in Figure~\ref{Res:10D} and Table~\ref{tab:optimizer-10d-wide}, GN again attains the lowest training energy with competitive CPU time. The zoom-in further confirms that GN reaches low energy levels within the first few outer iterations, demonstrating that five GN inner steps can match or surpass hundreds of first-order iterations. These results consistently support the conclusion that MMS with a GN inner solver is a practical and competitive algorithm across both regimes.

\begin{figure}[!ht]
\centering
    \includegraphics[width=0.75\linewidth]{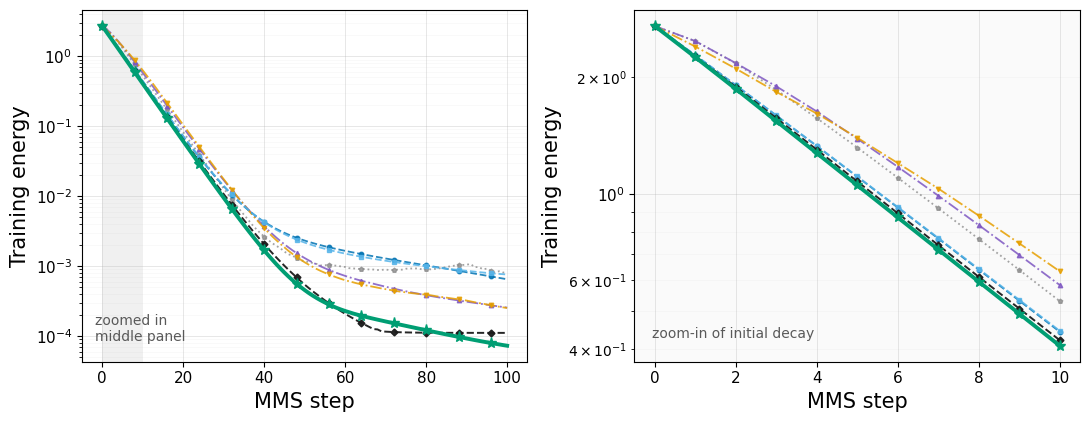}
    \caption{Results for the 10D synthesis data in the overparameterized regime ($m=128$, $1537$ parameters). Left: full training-energy decay over MMS steps. Right: zoom-in of the training-energy decay during the first $10$ MMS steps (shaded region in the left panel).}\label{Res:10D}
\end{figure}

\begin{table}[t]
\centering
\caption{
Comparison of optimizers on the 10D regression benchmark with a wider network
(1537 parameters, $n_{\text{train}}=1000$).
Each entry reports the final value after the MMS iteration. CPU time is measured in seconds on a MacBook Pro with an Apple M2 chip.}
\label{tab:optimizer-10d-wide}
\resizebox{\linewidth}{!}{%
\begin{tabular}{lccccccc}
\toprule
Metric 
& Adam & AdamW & L-BFGS & Muon & Shampoo & SOAP & GN \\
\midrule
Energy
& $6.49\mathrm{e}{-4}$
& $7.60\mathrm{e}{-4}$
& $1.11\mathrm{e}{-4}$
& $8.19\mathrm{e}{-4}$
& $2.56\mathrm{e}{-4}$
& $2.52\mathrm{e}{-4}$
& $7.28\mathrm{e}{-5}$ \\
Rel. $L^2$
& $6.85\mathrm{e}{-2}$
& $5.48\mathrm{e}{-2}$
& $2.90\mathrm{e}{-2}$
& $7.87\mathrm{e}{-1}$
& $4.10\mathrm{e}{-2}$
& $3.30\mathrm{e}{-2}$
& $3.30\mathrm{e}{-2}$ \\
Time (s)
& $98.11$
& $93.62$
& $51.56$
& $124.88$
& $218.27$
& $130.70$
& $70.80$ \\
\bottomrule
\end{tabular}%
}
\end{table}

\subsubsection{MNIST Classification with One-Hot Labels}

To demonstrate applicability beyond scalar regression, we consider 
a vector-valued MMS problem on the MNIST dataset. The energy 
functional is the one-hot mean-squared-error loss,
\begin{equation}
    F[u_w] = \frac{1}{2N}\sum_{i=1}^{N} 
    \|u_w(x_i) - e_{y_i}\|^2,
\end{equation}
where $e_{y_i} \in \mathbb{R}^{10}$ is the one-hot encoding of the 
label $y_i \in \{0,\ldots,9\}$. This preserves the least-squares 
structure of~\eqref{J_Reg} while extending the output dimension 
from $\mathbb{R}$ to $\mathbb{R}^C$ with $C = 10$, so that the Gauss--Newton framework of~\eqref{GN_for_u_final}, with $\mathbf{u}(w)\in\mathbb{R}^{NC}$ interpreted as the stacked output vector and $J(w)\in\mathbb{R}^{NC\times p}$ the corresponding Jacobian, applies without modification.
Classification accuracy is reported as a secondary metric to
assess practical utility.

We train a fully connected neural network with input dimension 
$d = 784$, depth $L = 3$, hidden width $m = 64$, and $\tanh$ 
activation, giving $p = 55{,}050$ trainable parameters. All weights 
are initialized with Xavier normal initialization (gain for $\tanh$) 
and all biases are set to zero. The training set consists of 
$N = 5000$ samples and the test set of $1000$ samples, with pixel 
values normalized to $[0, 1]$.

At each MMS step, we compare three inner solvers: Adam (learning rate 
$10^{-3}$, $100$ inner iterations), L-BFGS (strong Wolfe line search, 
maximum $20$ iterations per step), and Gauss--Newton with 
Levenberg--Marquardt regularization ($\rho = 10^{-5}$, $\delta t =1$, $5$ inner 
steps, CG limited to $20$ iterations). We perform $80$ MMS steps in total
with proximal parameter $\tau = 0.1$. Fig.~\ref{fig:MNIST} summarizes the results.
\begin{figure}[!ht]
    \centering
    \includegraphics[width= \linewidth]{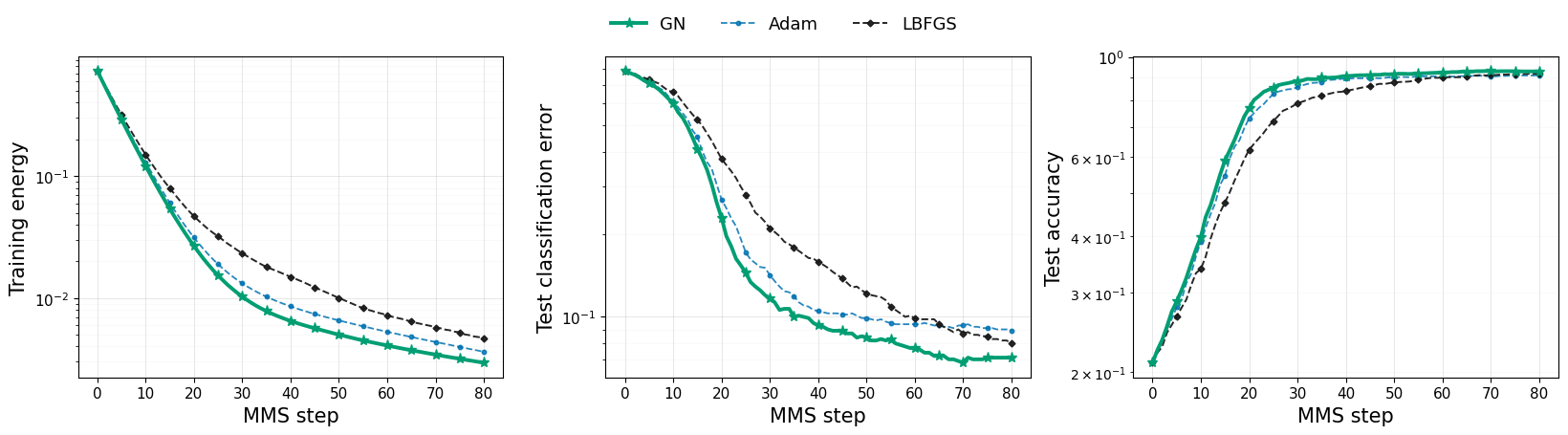}
    \caption{Results for MNIST one-hot regression with $\tau = 0.1$
    and $80$ MMS steps. }
    \label{fig:MNIST}
\end{figure}

\begin{table}[!h]
\centering
\caption{Summary of inner-solver performance for MNIST one-hot regression with $\tau=0.1$ and $80$ MMS steps. CPU time is measured in seconds on a MacBook Pro with an Apple M2 chip.}
\label{tab:mms_summary_mnist}
\begin{tabular}{lcccc}
\hline
Method & Final energy & Final rel.\ $L^2$ & Test accuracy & Time (s) \\
\hline
Adam   & $3.625\times 10^{-3}$ & $8.900\times 10^{-2}$ & $91.1\%$ & $405.08$ \\
L-BFGS & $4.621\times 10^{-3}$ & $8.000\times 10^{-2}$ & $92.0\%$ & $328.68$ \\
GN     & $2.945\times 10^{-3}$ & $7.100\times 10^{-2}$ & $92.9\%$ & $338.67$ \\
\hline
\end{tabular}
\end{table}

The MNIST one-hot experiment is intended as an applicability test beyond scalar regression,
rather than a benchmark optimized for classification accuracy. The least-squares structure is
preserved by using one-hot mean-squared error, so the same Gauss--Newton formulation
applies. As reported in Table~\ref{tab:mms_summary_mnist}, GN works well in practice
on this task, with competitive optimization quality and runtime behavior among the
three solvers. These results suggest that the
Gauss--Newton inner solver remains competitive in a higher-dimensional, vector-valued
setting. 
state-of-the-art classification performance.

\subsubsection{Latent diffusion model on MNIST}

We further test the MMS inner solvers in a simple latent diffusion \citep{rombach2022high} experiment
on MNIST \citep{lecun1998gradient}, since score matching is itself a nonlinear regression problem: the score network $s_\theta$ is trained to regress the noise $\varepsilon$ from noisy latent codes $(z_t, t)$, so the score-matching energy $E(\theta)$ has the same least-squares structure as the regression objectives considered earlier in this section. The Gauss--Newton framework therefore applies directly, with the Jacobian now taken with respect to $\theta$ and the residual vector given by the stacked noise-prediction errors. The purpose of this experiment is not to obtain state-of-the-art
generative performance, but to examine whether the Gauss--Newton-type inner
solver remains competitive with AdamW in a stochastic, higher-dimensional
neural MMS subproblem. 

We first train a convolutional variational autoencoder (VAE) to obtain a
low-dimensional latent representation of the MNIST images \citep{kingma2013auto}.  The encoder follows a standard convolutional autoencoder design \citep{radford2016unsupervised}: two
stride-two convolutional layers are used to progressively downsample the
\(28\times 28\) image to a \(7\times 7\) spatial feature map, while increasing
the number of channels from \(1\) to \(w/2\) and then to \(w\). A third
convolutional layer with stride one further processes the \(w\)-channel feature
map without changing its spatial resolution. In our experiments \(w=64\), so
the encoder has the channel progression
\(
    1 \longrightarrow 32 \longrightarrow 64 \longrightarrow 64 .
\)
All convolutional layers use SiLU nonlinearities. The final feature map of size
\(w\times 7\times 7\) is then flattened and mapped by two fully connected layers
to the latent mean and log-variance,
\(
    \mu_\phi(x),\ \log\sigma_\phi^2(x)\in\mathbb R^{d_z},
    \qquad d_z=16.
\)
The decoder first maps \(z\in\mathbb R^{d_z}\) back to a \(w\times 7\times 7\)
feature map and then applies two transposed convolutional layers and a final
convolutional layer to produce the image logits. The VAE is trained with the
standard reconstruction--KL objective \citep{kingma2013auto}:
\[
    \mathcal L_{\rm VAE}(\phi,\psi)
    =
    \frac{1}{B}\sum_{i=1}^B
    \Bigl[
    {\rm BCE}\bigl(D_\psi(z_i),x_i\bigr)
    +
    \beta_{\rm VAE}\,
    D_{\rm KL}
    \bigl(q_\phi(z_i|x_i)\,\|\,\mathcal N(0,I)\bigr)
    \Bigr],
\]
where \(z_i=\mu_\phi(x_i)+\sigma_\phi(x_i)\xi_i\), \(\xi_i\sim\mathcal N(0,I)\),
and we use \(\beta_{\rm VAE}=0.1\). After training the VAE, we use the encoder
mean \(\mu_\phi(x)\) as the latent code and normalize the latent variables by
their empirical mean and standard deviation:
\( \tilde z_0   = (\mu_\phi(x)-\bar z)/s_z. \)
This produces a small-scale latent-diffusion testbed in the spirit of latent
diffusion models \citep{rombach2022high}, while keeping the experiment simple
enough for controlled optimizer comparisons.

On this normalized latent space, we use the standard forward diffusion process
of denoising diffusion probabilistic models \citep{ho2020denoising}
\[
    z_t
    =
    \sqrt{\bar\alpha_t}\,\tilde z_0
    +
    \sqrt{1-\bar\alpha_t}\,\varepsilon,
    \qquad
    \varepsilon\sim\mathcal N(0,I),
\]
where \(\bar\alpha_t=\prod_{k=1}^t(1-\beta_k)\), and \(\{\beta_k\}_{k=1}^{T}\)
is a linear noise schedule with \(T=1000\). The score network is trained in the
usual noise-prediction form
\(
    s_\theta(z_t,t)\approx \varepsilon .
\)
In our implementation, \(s_\theta\) is a residual MLP with sinusoidal time
embedding. More precisely, the time index \(t\) is first mapped to a sinusoidal
embedding \(e(t)\in\mathbb R^{d_t}\), with \(d_t=128\). The input
\((z_t,e(t))\in\mathbb R^{d_z+d_t}\) is projected to a hidden representation of
width \(64\). We then apply two residual blocks of the form
\(
    h \mapsto h +
    W_2\,\sigma\!\left(W_1\,{\rm LN}(h)\right),
\)
where \({\rm LN}\) denotes layer normalization and \(\sigma\) is the SiLU
activation. A final layer normalization and linear output layer map the hidden
state back to \(\mathbb R^{d_z}\). The last linear layer is initialized to zero,
so that the initial score prediction is close to zero. For a mini-batch \(\{(z_t^i,t_i,\varepsilon_i)\}_{i=1}^B\), the basic
score-matching energy is
\[
    E(\theta)
    =
    \frac{1}{2B d_z}
    \sum_{i=1}^B
    w(t_i)\left\|
    s_\theta(z_t^i,t_i)-\varepsilon_i
    \right\|^2 .
\]
where $w(t)$ is a weighting function. The precise choice of \(w(t)\) is not essential for our optimizer-comparison
purpose. We use a mild decreasing weight \(w(t)=3-2t/(T-1)\), which assigns
larger weights to lower-noise denoising tasks while keeping the overall loss
scale comparable across time levels.
The MMS version uses the previous score network \(s_{\theta^n}\) as the proximal
reference and defines the inner objective at outer step \(n\) by
\[
    \mathcal L_{\rm MMS}^{n}(\theta)
    =
    \frac{1}{2\tau B d_z}
    \sum_{i=1}^B
    \left\|
    s_\theta(z_t^i,t_i)-s_{\theta^n}(z_t^i,t_i)
    \right\|^2
    + E(\theta) .
\]

\begin{figure}[!ht]
    \centering
    \begin{minipage}[c]{0.3 \linewidth}
        \centering

        \vspace{1em}
        \includegraphics[width=\linewidth]{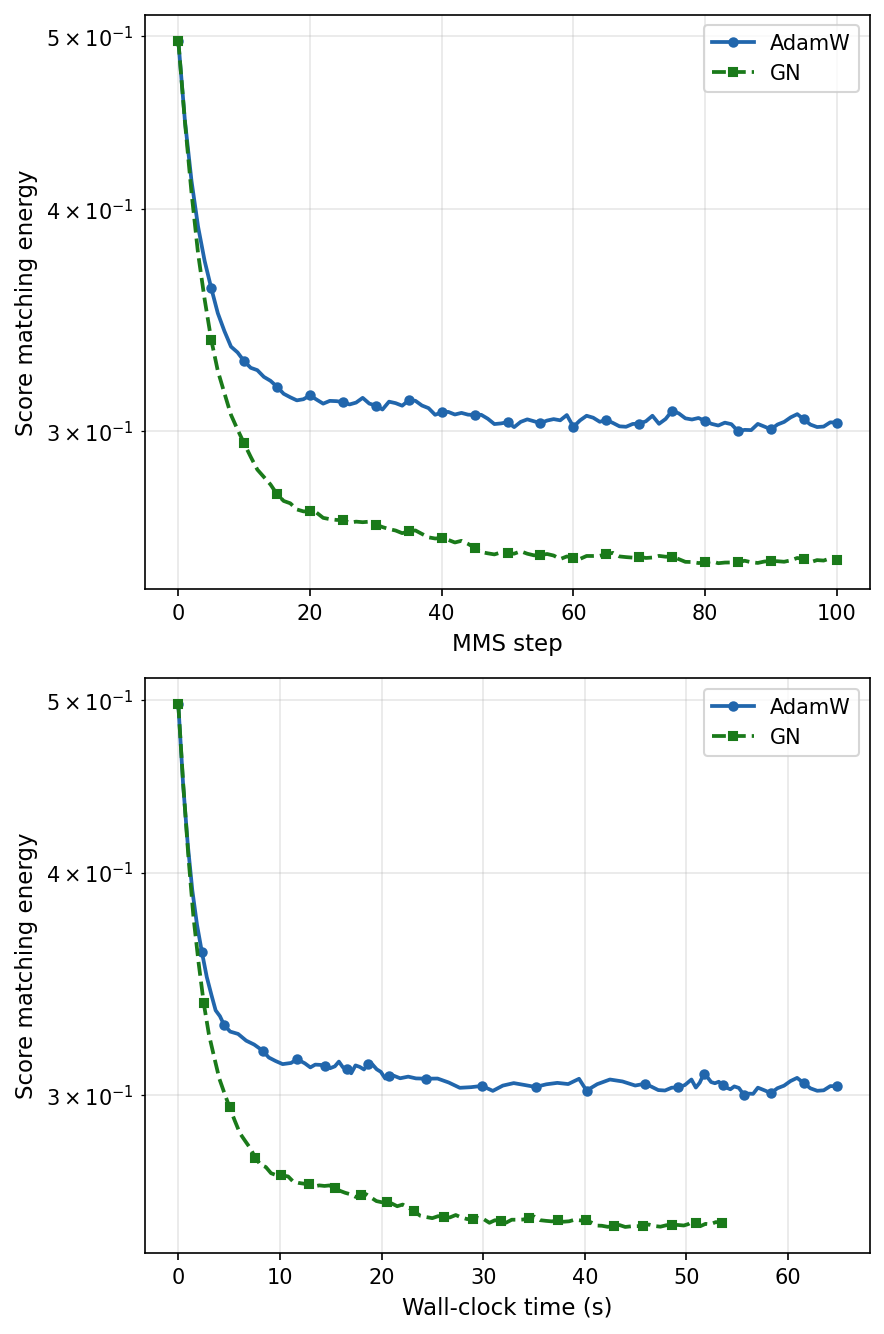}
    \end{minipage}
    \hfill
    \begin{minipage}[c]{0.65\linewidth}
        \centering
        \includegraphics[width=\linewidth]{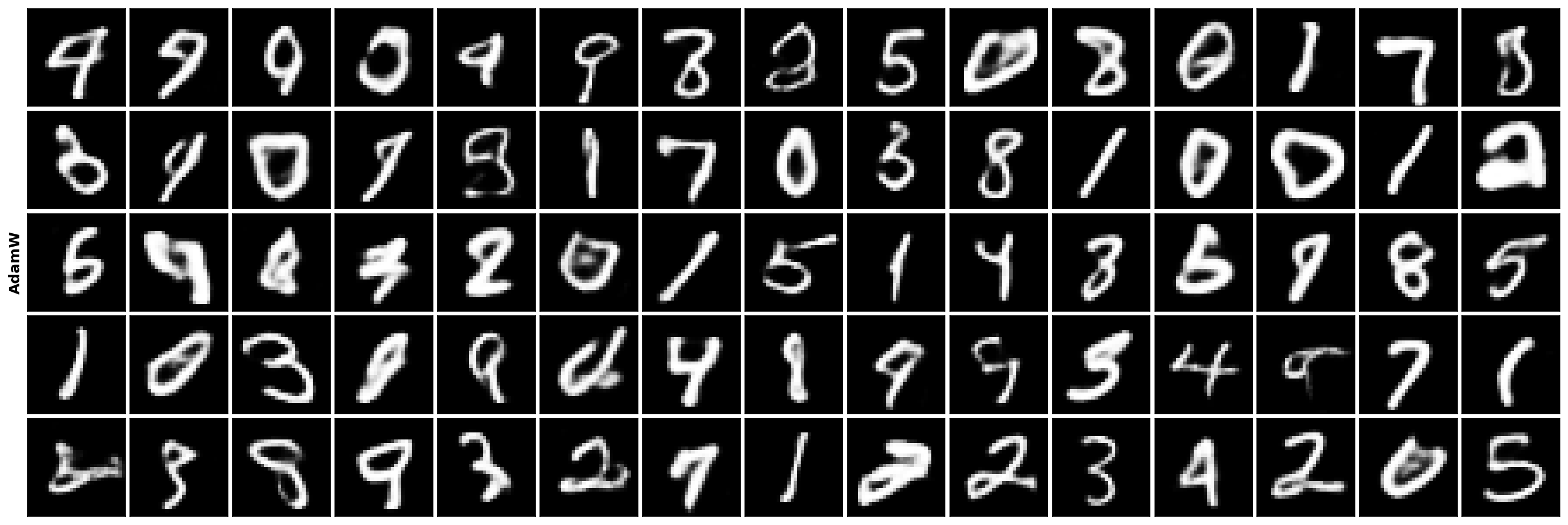}

        \vspace{0.8em}

        \includegraphics[width=\linewidth]{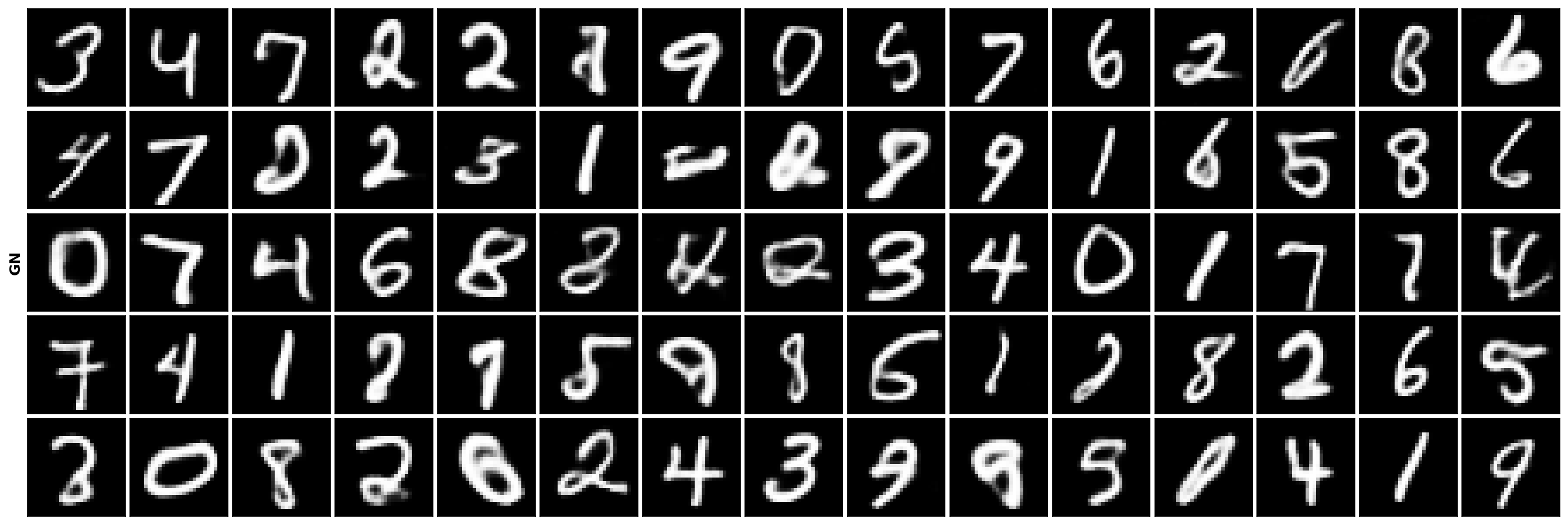}
    \end{minipage}
    \caption{Latent diffusion experiment on MNIST under the MMS formulation. Left: score-matching energy versus MMS steps and wall-clock time (measured in seconds on ASUS GX10) for Adam and Gauss--Newton (GN). Right: representative generated samples obtained after DDPM sampling in latent space and decoding through the VAE decoder.}\label{fig:ddpm}
\end{figure}

We compare AdamW and a damped Gauss--Newton-type method for solving this inner
problem. In the Gauss--Newton update, we use the Levenberg--Marquardt linear system~\eqref{eq:LM-update} with damping $\rho=10^{-3}$, solved by CG with at most $20$ iterations, and a line search as described above; we perform $5$ inner GN steps per MMS step. Figure~\ref{fig:ddpm} shows that GN attains a substantially lower final score-matching energy than AdamW (approximately $2.4\times10^{-1}$ vs.\ $3.1\times10^{-1}$) while converging faster both in terms of MMS steps and wall-clock time.
In particular, GN reaches its plateau within roughly $20$ outer MMS steps and $25$ seconds, whereas AdamW requires the full $100$ steps and over $65$ seconds to stabilize at a noticeably higher energy level.
This advantage is consistent with the findings in the regression benchmarks: the second-order curvature information exploited by GN allows each proximal subproblem to be resolved more accurately with only $5$ inner steps, whereas AdamW requires many more first-order iterations to achieve a comparable (but here still inferior) reduction in the subproblem objective.

Since this experiment is intended primarily as an optimization diagnostic, we
report the score-matching energy along MMS steps and wall-clock time. The
generated samples obtained by DDPM sampling from the learned latent score model
are included only as a qualitative sanity check. Thus, this experiment should be
interpreted as a comparison of inner solvers for a latent-space neural MMS
problem, rather than as a full benchmark for generative modeling performance.

In summary, across the numerical experiments considered here, solving the MMS subproblems
with the Gauss--Newton-type method generally yields results that are comparable to, and in
most cases better than, those obtained with Adam, AdamW, L-BFGS, Muon, Shampoo, and SOAP. From the viewpoint of computational
cost, Gauss--Newton is often competitive with or faster than Adam in our implementation. Each MMS
subproblem is solved only approximately, using $5$ Gauss--Newton steps in the regression benchmarks. By contrast, gradient-based baselines (Adam, AdamW, Muon, Shampoo, SOAP) are run for hundreds of inner iterations.
Thus, although each Gauss--Newton step is more expensive per iteration, the substantially smaller number of inner
iterations can lead to a lower overall CPU time.

\section{Conclusions}\label{sec:conclusions}

We have developed a geometric convergence theory for neural network optimization
within the minimizing movement scheme framework.
Our contributions are threefold.

First, we showed that by reformulating each neural MMS step as an increment
optimization problem, the feasible set naturally carries the structure of a
boundaryless smooth embedded Riemannian submanifold of the ambient Hilbert space (using an open parameter ball as the chart domain), together with a compact inner localization image used to control the gradient-flow trajectory.
Under a locally non-degenerate Jacobian condition together with a strict interior-localization condition, we established that the reached sublevel set of the increment objective lies in a strongly geodesically convex normal neighborhood of the origin, and we identified the manifold's tangent spaces and orthogonal projectors.

Second, we proved that a preconditioned (Gauss--Newton-type) gradient flow in
parameter space induces exactly the Riemannian gradient flow on the increment
manifold. By establishing geodesic strong convexity and smoothness of the
subproblem objective on an appropriate sublevel set, we showed that this
Riemannian gradient flow converges exponentially to the unique subproblem minimizer.

Third, we quantified how finite-time inexactness and neural approximation error
propagate through MMS iterations, proving that the inexact neural iterates
track the exact MMS iterates uniformly and converge to an $O(\delta)$-neighborhood
of the global minimizer when the energy functional is strongly convex.

Several limitations of the current framework suggest directions for future work.
The analysis assumes that the objective functional depends on the function values
in the $L^2$-sense; extending to functionals involving spatial derivatives (e.g., 
$H^1$-norms arising in PDE problems) would require controlling the Jacobian
Lipschitz constant in stronger norms, where it may scale unfavorably with spatial
discretization.
The locality assumption (Assumption~\ref{ass:intro:locality}), while shown to hold
for sufficiently small step sizes, may be restrictive for large-step or adaptive
methods in practice.
Finally, the convergence theory is formulated for the continuous-time Riemannian
gradient flow; a rigorous analysis of the discretized Gauss--Newton updates
(including Levenberg--Marquardt regularization and inexact linear solves) remains
an important open problem.

\acks{H.\ Yang and S. \ Zheng was partially funded by NSF under awards DMS-2244988, IIS-2520978, GEO/RISE-5239902, the Office of Naval Research Award N00014-23-1-2007, DOE (ASCR) Award DE-SC0026052, and the DARPA D24AP00325-00. Y.~Wang is partially supported by NSF DMS-2153029 and DMS-2410740.
}

\newpage
\appendix
\section{Deferred Proofs from the Main Text}\label{app:deferred-proofs}
\subsection{Proof of Proposition~\ref{prop:manifold}}\label{app:proof-manifold}
\begin{proof}
Let $B := B(\tilde w, r_{\tilde w})$, $\overline B := \{w:\|w-\tilde w\|_2\le r_{\tilde w}\}$, and
\[
\Phi:\overline B \to \mathcal{H}, \qquad \Phi(w) := u_\NN(w) - u_\NN(\tilde w),
\]
so that $\calM(u_\NN(\tilde w)) = \Phi(B)$ and $\calK(u_\NN(\tilde w)) = \{\Phi(w) : \|w-\tilde w\|_2 \le r^{\rm in}_{\tilde w}\}$ by~\eqref{eq:Mn}--\eqref{eq:Kn}.

Because $u_\NN$ is smooth, $\Phi$ is smooth, and $\D\Phi(w) = \D u_\NN(w)$ has full rank $p$ on $\overline B$ by Lemma~\ref{lem:gram-lower-bound} (the bound extends to the closure by continuity). Hence $\Phi$ is a smooth immersion on $\overline B$.

For any $w,w'\in\overline B$ and $\gamma(t) := tw + (1-t)w'$, the fundamental theorem of calculus and the $L$-Lipschitz continuity of $\D u_\NN$ give
\begin{equation}\label{eq:bi-lip-lower}
\|u_\NN(w)-u_\NN(w')\|_\calH \;\ge\; (2\lambda_{\tilde w} - L r_{\tilde w})\|w-w'\|_2 \;\ge\; \lambda_{\tilde w}\|w-w'\|_2,
\end{equation}
where the last step uses $r_{\tilde w}\le \lambda_{\tilde w}/L$ from~\eqref{eq:rwdef}. So $\Phi$ is injective on $\overline B$, and its inverse is Lipschitz with constant $1/\lambda_{\tilde w}$. Therefore $\Phi|_B$ is a homeomorphism from $B$ onto $\Phi(B) = \calM(u_\NN(\tilde w))$ with the subspace topology, and combined with the immersion property is a smooth embedding of the open ball $B$ \citep[Prop.~5.30]{Lee2003}. This makes $\calM(u_\NN(\tilde w))$ a $p$-dimensional boundaryless smooth embedded submanifold of $\calH$.

Finally, $\calK(u_\NN(\tilde w))$ is the continuous image of the compact set $\{w : \|w-\tilde w\|_2 \le r^{\rm in}_{\tilde w}\}\subset B$, hence compact in $\calM(u_\NN(\tilde w))$.
\end{proof}

\subsection{Proof of Proposition~\ref{prop:tangent_metric}}\label{app:proof-tangent}
\begin{proof}
The result follows directly from the definitions of the tangent space of an embedded submanifold and the induced Riemannian metric; see, e.g., \citep[Chapter~3]{absil2008optimization}. For completeness, we include a short proof below.

Let $\Phi:B(\tilde w,r_{\tilde w})\to\mathcal H$ be the embedding
\[
\Phi(w):=u_\NN(w)-u_\NN(\tilde w),
\qquad \mathcal M(u_\NN(\tilde w))=\Phi(B(\tilde w,r_{\tilde w})).
\]

By the definition of the tangent space of an embedded submanifold,
for $h=\Phi(w)\in\mathcal M(u_\NN(\tilde w))$ we have
\[
T_h\mathcal M(u_\NN(\tilde w))
=\Big\{\gamma'(0) : \gamma:(-\varepsilon,\varepsilon)\to \mathcal M(u_\NN(\tilde w))
\text{ smooth with }\gamma(0)=h\Big\}.
\]
Since $\Phi$ is a local diffeomorphism onto its image, any such curve can be written as
$\gamma(t)=\Phi(\eta(t))$ for a smooth $\eta(t)$ with $\eta(0)=w$.
Therefore, by the chain rule,
\[
\gamma'(0)=\D \Phi(w) \eta'(0)\in \mathrm{Im}(\D \Phi(w)).
\]
Conversely, for any $v\in\mathbb R^p$, taking $\eta(t)=w+tv$ yields a smooth curve
$\gamma(t)=\Phi(w+tv)$ in $\mathcal M(u_\NN(\tilde w))$ with $\gamma(0)=h$ and
$\gamma'(0)=\D\Phi(w)v$. Hence
\[
T_h\mathcal M(u_\NN(\tilde w))=\mathrm{Im}(\D\Phi(w)) =\mathrm{Im}(\D u_\NN (w)).
\]

Finally, since $\mathcal M(u_\NN(\tilde w))$ is an embedded submanifold of the Hilbert space
$\mathcal H$, the restriction of the ambient inner product to each tangent space defines a
Riemannian metric, i.e.,
\[
\langle \xi,\zeta\rangle_{\mathcal M(u_\NN(\tilde w))}:=\langle \xi,\zeta\rangle_{\mathcal H},
\qquad \forall \xi,\zeta\in T_h\mathcal M(u_\NN(\tilde w)).
\]
\end{proof}

\subsection{Proof of Proposition~\ref{prop:projection}}\label{app:proof-projection}
\begin{proof}
Let $A:=\D u_\NN(w):\mathbb R^p\to\mathcal H$. By Proposition~\ref{prop:tangent_metric},
\[
T_h\mathcal M(u_\NN(\tilde w))=\mathrm{Im}(A).
\]
Fix $z\in\mathcal H$. Any $y\in \mathrm{Im}(A)$ can be written as $y=A\xi$ for some $\xi\in\mathbb R^p$.
Hence, the projection problem is equivalent to the least-squares problem
\[
\min_{y\in \mathrm{Im}(A)}\|y-z\|_{\mathcal H}^2
\quad\Longleftrightarrow\quad
\min_{\xi\in\mathbb R^p}\|A\xi-z\|_{\mathcal H}^2.
\]
Since $A$ has full column rank (by the nondegeneracy assumption in Definition~\ref{def:non-degen-jacobian}),
the normal equation has the unique solution
\[
\xi^*=(A^*A)^{-1}A^*z.
\]
Therefore, the unique minimizer in $\mathrm{Im}(A)$ is
\[
y^*=A\xi^*=A(A^*A)^{-1}A^*z =: \P_h z,
\]
which proves the stated formula.

Moreover, $\P_h$ is a projection onto $\mathrm{Im}(A)$: for any $y=A\xi\in \mathrm{Im}(A)$,
\[
\P_h y = A(A^*A)^{-1}A^*A\xi = A\xi = y,
\]
hence $\mathrm{Im}(\P_h)=\mathrm{Im}(A)$ and $\P_h^2=\P_h$. This completes the proof of part (a).

For part (b), the Riemannian gradient of $g(\cdot; v)$ at $h \in \mathcal{M}( u_\NN(\tilde w))$
is the unique vector $\operatorname{grad}_{\mathcal{M}( u_\NN(\tilde w))} g(h; v)\in T_h\mathcal{M}( u_\NN(\tilde w))$
satisfying (see \citep[Section~3.6.1]{ambrosio2005gradient})
\[
\langle\operatorname{grad}_{\mathcal{M}( u_\NN(\tilde w))} g(h; v),\xi\rangle_{\mathcal{M}( u_\NN(\tilde w))}
 = \langle\nabla g(h; v),\xi\rangle_{\mathcal H},
 \qquad \forall \xi\in T_h\mathcal{M}( u_\NN(\tilde w)),
\]
and therefore
\[
\operatorname{grad}_{\mathcal{M}( u_\NN(\tilde w))} g(h; v)
   =\P_h \left( \nabla g(h; v) \right).
\]
\end{proof}

\subsection{Proof of Lemma~\ref{lem:DP_bound}}\label{app:proof-of-lemmas}

\begin{proof}
Set $A:=\D u_\NN(w):\R^p\to\mathcal H$ and denote by
\[
\P := \P_h = A(A^*A)^{-1}A^*
\]
the orthogonal projector onto $\mathrm{Im}(A)=T_h\mathcal M(u_\NN(\tilde w))$ (Proposition~\ref{prop:projection}).
Let $\xi\in T_h\mathcal M(u_\NN(\tilde w))$. Since $T_h\mathcal M=\mathrm{Im}(A)$ and $A$ has full
column rank, there exists a \emph{unique} $v\in\R^p$ such that
\begin{equation}\label{eq:xi_rep}
\xi=A v.
\end{equation}
Consider the perturbation $w(t)=w+t v$ and set $A(t):=\D u_\NN(w(t))$ and
$\P(t):=A(t)\big(A(t)^*A(t)\big)^{-1}A(t)^*$.
Then $\P(0)=\P$ and $\D \P_{(h)}[\xi]$ coincides with $\P'(0)$ under the identification
$\xi=\gamma'(0)$ with $\gamma(t)=\Phi(w(t))$.

Since $\P(t)$ is the orthogonal projector onto $\mathrm{Im}(A(t))$, we have
\begin{equation}\label{eq:proj_id1}
(I-\P(t))A(t)=0
\qquad\text{and}\qquad
A(t)^*(I-\P(t))=0
\quad\text{for all }t.
\end{equation}
Differentiating \eqref{eq:proj_id1} at $t=0$ gives
\[
-\P'(0)A + (I-\P)A'(0)=0,
\qquad
A'(0)^*(I-\P) - A^*\P'(0)=0.
\]
Let $A^+=(A^*A)^{-1}A^*$ be the (left) pseudoinverse of $A$ (well-defined since $A$ has full
column rank). Right-multiplying the first identity by $A^+$ and left-multiplying the second
identity by $(A^+)^*$ yield
\[
\P'(0)\P = (I-\P)A'(0)A^+,
\qquad
\P \P'(0) = (A^+)^*A'(0)^*(I-\P).
\]
Adding these two identities and using $\P'(0)=\P'(0)\P + \P \P'(0)$ (because $\P$ is a projector)
gives the standard formula
\begin{equation*}
\P'(0) = (I-\P) A'(0) A^+ + (A^+)^* A'(0)^* (I-\P).
\end{equation*}
Taking operator norms and using $\|I-\P\|=1$ we obtain
\begin{equation}\label{eq:dP_basic_bound}
\|\P'(0)\|
\le 2 \|A^+\| \|A'(0)\|.
\end{equation}

By the $\lambda_{\tilde w}$-non-degeneracy Jacobian of $u_\NN$  (see Definition~\ref{def:non-degen-jacobian}) on $B(\tilde w,r_{\tilde w})$,
\[
A^*A \succeq \lambda_{\tilde w}^2 I
\quad\Longrightarrow\quad
\|A^+\|=\|(A^*A)^{-1}A^*\|=\frac{1}{\sigma_{\min}(A)}\le \frac{1}{\lambda_{\tilde w}}.
\]
Moreover, since $\D u_\NN$ is $L$-Lipschitz, we have for all $t$,
\[
\|A(t)-A(0)\| = \|\D u_\NN(w+tv)-\D u_\NN(w)\| \le L|t| \|v\|_2,
\]
hence $\|A'(0)\|\le L\|v\|_2$.

Combining with \eqref{eq:dP_basic_bound} yields
\begin{equation}\label{eq:dP_v_bound}
\|\P'(0)\|\le \frac{2L}{\lambda_{\tilde w}} \|v\|_2.
\end{equation}

Now \eqref{eq:xi_rep} and $\sigma_{\min}(A)\ge \lambda_{\tilde w}$ gives
\[
\|\xi\|_{\mathcal H}=\|A v\|_{\mathcal H}\ge \lambda_{\tilde w} \|v\|_2,
\qquad\text{so that}\qquad
\|v\|_2\le \frac{1}{\lambda_{\tilde w}}\|\xi\|_{\mathcal H}.
\]
Substituting this into \eqref{eq:dP_v_bound} gives
\[
\|\P'(0)\|
\le \frac{2L}{\lambda_{\tilde w}}\cdot \frac{1}{\lambda_{\tilde w}}\|\xi\|_{\mathcal H}
= \frac{2L}{\lambda_{\tilde w}^2} \|\xi\|_{\mathcal H}.
\]
This proves the claimed differential bound.
\end{proof}

\subsection{Proof of Lemma~\ref{lem:convex-radius}}\label{app:proof-convex-radius}
\begin{proof}
Fix $\tilde w \in \R^p$, recall the open ball $B(\tilde w,r_{\tilde w})$ defined in~\eqref{eq:B}--\eqref{eq:rwdef} and the inner radius $r^{\rm in}_{\tilde w} = r_{\tilde w}/2$, and let $\Phi(w) := u_\NN(w)-u_\NN(\tilde w)$. By Proposition~\ref{prop:manifold}, $\calM(u_\NN(\tilde w)) = \Phi(B(\tilde w,r_{\tilde w}))$ is a $p$-dimensional boundaryless smooth embedded Riemannian submanifold of $\calH$, and $\calK(u_\NN(\tilde w))$ is compact.

\emph{Step 1: Sectional curvature upper bound via the second fundamental form.}
For $h=\Phi(w)\in \mathcal M(u_\NN(\tilde w))$, let $\P_h$ be the orthogonal projector onto $T_h\mathcal M(u_\NN(\tilde w))$. For a nonzero tangent direction $\xi\in T_h\mathcal M(u_\NN(\tilde w))$, the quantity
\[
\mathcal L(h,\xi) := \left\|\left.\tfrac{\ud}{\ud t}\P_{\beta(t)}\right|_{t=0}\right\|,
\]
where $\beta:(-\epsilon,\epsilon)\to \mathcal M(u_\NN(\tilde w))$ is any smooth curve with $\beta(0)=h$ and $\beta'(0)=\xi$, equals the operator norm $\|\mathrm{II}_h(\xi,\cdot)\|_{\mathrm{op}}$ of the second fundamental form in direction $\xi$ \citep{docarmo1992riemannian}. Lemma~\ref{lem:DP_bound} gives the uniform bound
\begin{equation}\label{eq:LambdaBound}
\sup_{h\in\mathcal M(u_\NN(\tilde w))}\,
\sup_{\substack{\xi \in T_h\mathcal M(u_\NN(\tilde w))\\\|\xi\|_\calH=1}}\!
\|\mathrm{II}_h(\xi,\cdot)\|_{\mathrm{op}} \;\le\; \Lambda := \frac{2L}{\lambda_{\tilde w}^2}.
\end{equation}
Because the ambient Hilbert space $\mathcal H$ is flat (sectional curvature $0$), the Gauss equation \citep[Theorem~1]{alexander2006gauss} implies that the intrinsic sectional curvature of $\calM(u_\NN(\tilde w))$ is bounded above by $0 + \Lambda^2 = \Lambda^2$ in the Alexandrov sense.

\emph{Step 2: Strongly geodesically convex normal neighborhood.}
By the convexity-radius theorem \citep[Theorem~5.14]{cheeger2008comparison}: for a smooth Riemannian manifold with sectional curvature $\le \Lambda^2$, the geodesic ball $B_\calM(p,r)$ is strongly geodesically convex provided
\[
r \;\le\; \min\Bigl\{\tfrac{\pi}{2\Lambda},\ \tfrac{1}{2}\mathrm{inj}_\calM(p)\Bigr\}.
\]
By the hypothesis~\eqref{eq:inj-hypothesis}, $\mathrm{inj}_{\calM(u_\NN(\tilde w))}(0)\ge \pi/\Lambda$, so the minimum equals $\pi/(2\Lambda)$. Therefore the geodesic ball
\[
    \calG(u_\NN(\tilde w)) := B_{\calM(u_\NN(\tilde w))}\bigl(0,\,\pi/(2\Lambda)\bigr)
\]
defined in~\eqref{eq:Gdef} is a strongly geodesically convex normal neighborhood of $0$, proving part~(a).

\emph{Step 3: $\calK(u_\NN(\tilde w)) \subset \calG(u_\NN(\tilde w))$.}
Bound the geodesic distance from $0$ to any $h = \Phi(w) \in \calK(u_\NN(\tilde w))$, i.e., for $w \in B(\tilde w, r^{\rm in}_{\tilde w})$. Consider the parameter-space segment $w(t) := (1-t)\tilde w + tw$, $t \in [0,1]$, and its $\calM$-image $\gamma(t) := \Phi(w(t))$. Its length in $\calM(u_\NN(\tilde w))$ is
\begin{align*}
    \mathrm{Length}_{\calM}(\gamma)
    &= \int_0^1 \|\D u_\NN(w(t))(w-\tilde w)\|_\calH \, dt \\
    &\le \int_0^1 \bigl(\|\D u_\NN(\tilde w)\|_{\mathrm{op}} + L\,t\|w-\tilde w\|_2\bigr) \|w-\tilde w\|_2 \, dt \\
    &\le \bigl(\|\D u_\NN(\tilde w)\|_{\mathrm{op}} + L\,r_{\tilde w}\bigr) \|w-\tilde w\|_2,
\end{align*}
using $\|w(t)-\tilde w\| \le \|w-\tilde w\| \le r^{\rm in}_{\tilde w} \le r_{\tilde w}$ and $L$-Lipschitz $\D u_\NN$ on $B(\tilde w, r_{\tilde w})$. With $\|w-\tilde w\|_2 \le r^{\rm in}_{\tilde w} = r_{\tilde w}/2$,
\[
    d_{\calM}(h, 0) \le \mathrm{Length}_\calM(\gamma) \le \bigl(\|\D u_\NN(\tilde w)\|_{\mathrm{op}} + L\,r_{\tilde w}\bigr)\,\tfrac{r_{\tilde w}}{2}.
\]
We now use the choice of $r_{\tilde w}$ in~\eqref{eq:rwdef}: in particular, its second entry gives
$r_{\tilde w} \le \lambda_{\tilde w}^2/\bigl[2L\bigl(\sqrt{\|\D u_\NN(\tilde w)\|_{\mathrm{op}}^2 + \lambda_{\tilde w}^2} + \|\D u_\NN(\tilde w)\|_{\mathrm{op}}\bigr)\bigr]$, so
\[
    \bigl(\|\D u_\NN(\tilde w)\|_{\mathrm{op}} + L\,r_{\tilde w}\bigr)\,r_{\tilde w}
    \;\le\; \bigl(\|\D u_\NN(\tilde w)\|_{\mathrm{op}} + \lambda_{\tilde w}\bigr) \cdot \frac{\lambda_{\tilde w}^2}{2L(\sqrt{\|\D u_\NN(\tilde w)\|_{\mathrm{op}}^2+\lambda_{\tilde w}^2}+\|\D u_\NN(\tilde w)\|_{\mathrm{op}})}
    \;\le\; \frac{\lambda_{\tilde w}^2}{2L},
\]
where the first inequality uses $L r_{\tilde w} \le \lambda_{\tilde w}$ (from the first entry of~\eqref{eq:rwdef}) and the second uses $\|\D u_\NN(\tilde w)\|_{\mathrm{op}} + \lambda_{\tilde w} \le \sqrt{\|\D u_\NN(\tilde w)\|_{\mathrm{op}}^2 + \lambda_{\tilde w}^2} + \|\D u_\NN(\tilde w)\|_{\mathrm{op}}$ (since $\lambda_{\tilde w} \le \sqrt{\|\D u_\NN(\tilde w)\|_{\mathrm{op}}^2 + \lambda_{\tilde w}^2}$). Therefore
\[
    d_{\calM}(h, 0) \;\le\; \frac{\lambda_{\tilde w}^2}{4L} \;=\; \frac{1}{\pi}\cdot\frac{\pi}{2\Lambda} \;<\; \frac{\pi}{2\Lambda},
\]
which equals the radius of $\calG(u_\NN(\tilde w))$ from~\eqref{eq:Gdef} up to a factor of $1/\pi$. Hence $h \in \calG(u_\NN(\tilde w))$, proving part~(b).
\end{proof}

\subsection{Proof of Lemma~\ref{lem:smooth-strongconv}}\label{app:proof-smooth-strongconv}
\begin{proof}
By Fr\'echet differentiation and chain rule,
\[
\nabla g(h;v)=\frac{1}{\tau} h + \nabla \mathcal F\!\bigl[h+v\bigr].
\]
For any $h,k\in\mathcal H$, set $x=h+v$, $y=k+v$. Then
\[
\|\nabla g(h;v)-\nabla g(k;v)\|
\le \frac{1}{\tau}\|h-k\|+\|\nabla\mathcal F[x]-\nabla\mathcal F[y]\|
\le\Big(\frac{1}{\tau}+L_{\mathcal F}\Big)\|h-k\|.
\]
Thus, $\nabla g(\cdot;v)$ is $\mu$-Lipschitz. Moreover,
\[
\begin{aligned}
\langle \nabla g(h;v)-\nabla g(k;v), h-k\rangle
&=\frac{1}{\tau}\|h-k\|^2
+\langle \nabla\mathcal F[x]-\nabla\mathcal F[y], h-k\rangle\\
&=\frac{1}{\tau}\|h-k\|^2
+\langle \nabla\mathcal F[x]-\nabla\mathcal F[y], x-y\rangle\\
&\ge \Big(\frac{1}{\tau}+m_{\calF}\Big)\|h-k\|^2,
\end{aligned}
\]
which yields $\nu$-strong convexity.
\end{proof}

\subsection{Riemannian connection and Hessian on Hilbert-embedded manifolds}\label{app:riem-hess}

The proofs of Theorems~\ref{thm:g} and~\ref{thm:RGF-conv} use the following standard decomposition of the Riemannian Hessian for an embedded submanifold of a flat ambient Hilbert space. The result is classical (see, e.g., \citealt[Chapter~3]{absil2008optimization}, \citealt[Chapter~5]{boumal2023intromanifolds}); we record it here for completeness.

\begin{lemma}[Riemannian connection and Hessian on Hilbert-embedded manifolds]\label{lem:connection}
Let $(\mathcal H,\langle\cdot,\cdot\rangle_{\mathcal H})$ be a real Hilbert space endowed with its flat metric. Let $u_\NN$ be a neural network with $L$-Lipschitz Jacobian and $\lambda_{\tilde w}$-nondegenerate Jacobian at $\tilde w$, so that $\mathcal{M}( u_\NN(\tilde w))$ as defined in \eqref{eq:Mn} is an embedded Riemannian submanifold of $\mathcal{H}$ (Proposition~\ref{prop:manifold}). Denote by $\nabla^{\mathcal H}$ the flat (ambient) connection on $\mathcal H$ and by $\nabla^{\mathcal M}$ the induced Levi--Civita connection on $\mathcal{M}( u_\NN(\tilde w))$. Then:
\begin{enumerate}
\item \textbf{Flat connection on $\mathcal H$.}
For any smooth vector field $Y:\mathcal H\to\mathcal H$ and any $\xi\in\mathcal H$,
\[
\nabla^{\mathcal H}_{\xi_{u}}Y = \mathrm D Y(u)[\xi],
\]
that is, the Levi--Civita connection on $\mathcal H$ is simply the Fr\'echet differential of~$Y$.

\item \textbf{Induced connection on $\mathcal{M}( u_\NN(\tilde w))$.}
For smooth vector fields $X,Y$ tangent to $\mathcal{M}( u_\NN(\tilde w))$, the induced connection is
\begin{equation}\label{eq:induced-connection}
\nabla^{\mathcal M}_{X_h} Y = \P_h (  \mathrm D Y(h)[X]),
\end{equation}
where $h\in \mathcal{M}( u_\NN(\tilde w))$.

\item \textbf{Riemannian Hessian.}
For the smooth function $g(\cdot;v):\mathcal{M}( u_\NN(\tilde w))\to\R$ defined in \eqref{eq:intro:g_def}, its Riemannian Hessian at $h \in \mathcal{M}( u_\NN(\tilde w))$ applied to $\xi_h \in T_h\mathcal{M}( u_\NN(\tilde w))$ is
\begin{equation}\label{eq:RiemHess}
\Hess g(h; v)[\xi_h]
 = \P_h \left(  \mathrm D \grad g(h; v)[\xi_h] \right)
 = \P_h\left(\nabla^2 g(h; v)[\xi_h] +\left.\frac{\ud}{\ud t}\right\vert_{t=0} \P_{h(t)} \left(\nabla g(h; v)\right)\right),
\end{equation}
where $h(0) = h$ and $h'(0) = \xi_h$.
Consequently, the quadratic form of the Riemannian Hessian satisfies
\begin{equation}\label{eq:RiemHessForm}
\langle \Hess g(h; v)[\xi_h],\xi_h\rangle_{\mathcal H}
 = \langle\nabla^2 g(h; v)[\xi_h],\xi_h\rangle_{\mathcal H}
   + \langle\mathrm D \P_{(h)}[\xi_h] \nabla g(h; v) , \xi_h \rangle_{\mathcal H}.
\end{equation}
\end{enumerate}
\end{lemma}

\begin{proof}
These are standard results; see, e.g., \citet[Chapter~3]{absil2008optimization}. We outline the key steps. (i) Since the metric on $\mathcal H$ is constant, the Christoffel symbols vanish, so $\nabla^{\mathcal H}$ coincides with the Fr\'echet derivative. (ii) For an embedded Riemannian submanifold, projecting the ambient connection onto the tangent space yields~\eqref{eq:induced-connection}. (iii) Differentiating $\grad g(h; v)= \P_h\nabla g(h; v)$ and projecting again gives~\eqref{eq:RiemHess}--\eqref{eq:RiemHessForm}.
\end{proof}

\subsection{Proof of Theorem~\ref{thm:g}}\label{app:proof-thm-g}
\begin{proof}
Since $g(\cdot;v)$ has $\mu$-Lipschitz continuous
gradient and is $\nu$-strongly convex, for any $q\in\mathcal H$,
\begin{equation}\label{eq:ambient-Hess}
\nu \|q\|_{\mathcal H}^2 \le \langle\nabla^2 g(h;v)[q],q\rangle_{\mathcal H}
 \le \mu \|q\|_{\mathcal H}^2 .
\end{equation}

Let $h\in\mathcal{M}:=\mathcal{M}(u_\NN(\tilde w))$ and $\Hess g(h;v)$ denote the Riemannian Hessian
of $g(\cdot;v)|_{\mathcal M}$ at $h$.
For $\xi_h \in T_h\mathcal M$, \eqref{eq:RiemHessForm} yields
\begin{equation}\label{eq:gHess-decomp}
\langle \Hess g(h;v)[\xi_h],\xi_h\rangle_{\mathcal H}
 = \langle\nabla^2 g(h;v)[\xi_h],\xi_h\rangle_{\mathcal H}
   + \langle\D \P_{(h)}[\xi_h] \nabla g(h;v) , \xi_h \rangle_{\mathcal H}.
\end{equation}
By \eqref{eq:ambient-Hess}, the first term in \eqref{eq:gHess-decomp} is bounded as
\begin{equation}\label{eq:term1_bounds}
\nu\|\xi_h\|_{\mathcal H}^2
\le \langle\nabla^2 g(h;v)[\xi_h],\xi_h\rangle_{\mathcal H}
\le \mu\|\xi_h\|_{\mathcal H}^2.
\end{equation}

By Lemma~\ref{lem:DP_bound}, we have for all $\xi_h\in T_h\mathcal M$, the differential bound in the second term is
\begin{equation}\label{eq:DP_bound_recall}
\|\D \P_{(h)}[\xi_h]\|
 \le  \frac{2L}{\lambda_{\tilde w}^2} \|\xi_h\|_{\mathcal H}.
\end{equation}

Let $h^*:=\arg\min_{q\in\mathcal H} g(q;v)$ (unique by $\nu$-strong convexity). Then
$\nabla g(h^*;v)=0$, and the $\mu$-Lipschitz continuity of $\nabla g(\cdot;v)$ implies
\begin{equation}\label{eq:grad_Lip_mu}
\|\nabla g(h;v)\|
=\|\nabla g(h;v)-\nabla g(h^*;v)\|
\le \mu \|h-h^*\|_{\mathcal H}.
\end{equation}

For any $h\in\mathcal M$, we have $\|h-h^*\|\le \|h\|+\|h^*\|$, hence
\begin{equation}\label{eq:grad_bound_on_M}
\|\nabla g(h;v)\|\le \mu(\|h\|_{\mathcal H}+\|h^*\|_{\mathcal H}).
\end{equation}
Moreover, for $h=\Phi(w)=u_\NN(w)-u_\NN(\tilde w)$ with $w\in B(\tilde w,r_{\tilde w})$, the Lipschitz continuity of $u_\NN$ yields
\[
\|u_\NN(w)-u_\NN(\tilde w)\|_{\mathcal H}
\le \Lip_{u_\NN}\nm{w - \tilde w} \le \Lip_{u_\NN} r_{\tilde w}.
\]
Hence
\begin{equation}\label{eq:h_norm_bound}
\|h\|_{\mathcal H}\le \Lip_{u_\NN} r_{\tilde w},
\end{equation}
Combining \eqref{eq:grad_bound_on_M}--\eqref{eq:h_norm_bound} gives, for all $h\in\mathcal M$,
\begin{equation}\label{eq:grad_uniform_M}
\|\nabla g(h;v)\|\le \mu\big( \Lip_{u_\NN} r_{\tilde w}  +\|h^*\|_{\mathcal H}\big).
\end{equation}

By Cauchy--Schwarz, \eqref{eq:DP_bound_recall}, and \eqref{eq:grad_uniform_M}, we obtain for any $\xi_h\in T_h\mathcal M$,
\begin{equation}\label{eq:term2_bound_general}
\big|\langle\D \P_{(h)}[\xi_h] \nabla g(h;v) , \xi_h \rangle_{\mathcal H}\big|
\le \frac{2L}{\lambda_{\tilde w}^2} \mu\big( \Lip_{u_\NN}  r_{\tilde w} +\|h^*\|_{\mathcal H}\big)\|\xi_h\|_{\mathcal H}^2.
\end{equation}

\medskip
\noindent\textbf{(a) Geodesic convexity of $g$ on $\calG(u_\NN(\tilde w))$.}
By Lemma~\ref{lem:convex-radius}(a), $\calG(u_\NN(\tilde w))$ is a strongly geodesically convex normal neighborhood of $0$ in $\calM(u_\NN(\tilde w))$, so any two of its points are joined by a unique minimizing geodesic lying inside $\calG(u_\NN(\tilde w))$. It therefore suffices to show $\langle \Hess g(h;v)[\xi_h], \xi_h\rangle_\calH \ge 0$ for all $h \in \calG(u_\NN(\tilde w))$ and $\xi_h \in T_h\calM(u_\NN(\tilde w))$. Combining \eqref{eq:gHess-decomp}, \eqref{eq:term1_bounds},
\eqref{eq:term2_bound_general}, and \eqref{eq:grad_uniform_M} yields, for any $h \in \calM(u_\NN(\tilde w)) \supset \calG(u_\NN(\tilde w))$,
\[
\langle \Hess g(h;v)[\xi_h],\xi_h\rangle_{\mathcal H}
\ge \Big(\nu-\frac{2L}{\lambda_{\tilde w}^2}\mu\big( \Lip_{u_\NN}  r_{\tilde w} +\|h^*\|_{\mathcal H}\big)\Big)\|\xi_h\|^2.
\]
The radius $r_{\tilde w}$ is chosen in \eqref{eq:rwdef} such that
\begin{equation}\label{eq:R_choice}
\frac{2L\mu}{\lambda_{\tilde w}^2} \Lip _{u_\NN} r _{\tilde w}  \le  \frac{\nu}{2}.
\end{equation}
Moreover, by Lemma~\ref{lem:g-properties}(a), $h^*=h^*(v)$, hence
\begin{equation}\label{eq:hstar-cond}
\frac{2L\mu}{\lambda_{\tilde w}^{2}}\|h^*\|_\calH
\;=\; \frac{\nu}{2}\cdot\frac{4L\kappa}{\lambda_{\tilde w}^{2}}\|h^*(v)\|_\calH .
\end{equation}
Under the data condition~\eqref{eq:data-condition}, $\frac{4L\kappa}{\lambda_{\tilde w}^{2}}\|h^*(v)\|_\calH\le 1$, and therefore $\frac{2L\mu}{\lambda_{\tilde w}^{2}}\|h^*\|_\calH\le \nu/2$.
Therefore, under \eqref{eq:R_choice} and \eqref{eq:hstar-cond},
\[
\langle \Hess g(h;v)[\xi_h],\xi_h\rangle_{\mathcal H}\ge 0,
\qquad \forall h\in\calG(u_\NN(\tilde w)),  \forall \xi_h\in T_h\mathcal M(u_\NN(\tilde w)).
\]
This non-negativity of the Riemannian Hessian on $\calG(u_\NN(\tilde w))$, together with the unique geodesics provided by Lemma~\ref{lem:convex-radius}(a), implies that $g(\cdot;v)|_{\calG(u_\NN(\tilde w))}$ is geodesically convex.

\medskip
\noindent\textbf{(b) Geodesic convexity of the sublevel set $S$.}
By hypothesis~(H2) and Lemma~\ref{lem:convex-radius}(b), $S \subset \calK(u_\NN(\tilde w)) \subset \calG(u_\NN(\tilde w))$. Since $g(\cdot;v)|_{\calG(u_\NN(\tilde w))}$ is geodesically convex from~(a), sublevel sets of $g$ inside $\calG(u_\NN(\tilde w))$ are themselves geodesically convex (see \citep[Proposition~11.8]{boumal2023intromanifolds}). In particular, $S=\{h\in\calG(u_\NN(\tilde w)) : g(h;v)\le g(0;v)\}$ is geodesically convex.

\medskip
\noindent\textbf{(c) Geodesic strong convexity of $g$ on $S$.}
Since $g(\cdot;v)$ is $\nu$-strongly convex on $\mathcal H$, we have
\[
g(h;v)-g(h^*;v)\ge \frac{\nu}{2}\|h-h^*\|_{\mathcal H}^2,
\]
and since $0 \leq g(h^*;v)\le g(h;v)\le g(0;v)$ for $h\in S$, it follows that
\begin{equation}\label{eq:dist_to_star_on_S}
\|h-h^*\|_{\mathcal H}\le \sqrt{\frac{2g(0;v)}{\nu}},
\qquad \forall h\in S.
\end{equation}
Combining \eqref{eq:grad_Lip_mu} and \eqref{eq:dist_to_star_on_S} yields the uniform bound
\begin{equation}\label{eq:grad_uniform_S}
\|\nabla g(h;v)\|\le \mu\sqrt{\frac{2g(0;v)}{\nu}},
\qquad \forall h\in S.
\end{equation}

For $h\in S$, using \eqref{eq:gHess-decomp}, \eqref{eq:term1_bounds}, and the projector-derivative bound~\eqref{eq:term2_bound_general} with the $S$-restricted gradient bound~\eqref{eq:grad_uniform_S} (i.e., $\|\nabla g(h;v)\|\le\mu\sqrt{2g(0;v)/\nu}$ for $h\in S$) substituted for the $\calM$-wide bound, gives
\[
\langle \Hess g(h;v)[\xi_h],\xi_h\rangle_{\mathcal H}
\ge \Big(\nu-\frac{2L}{\lambda_{\tilde w}^2}\mu\sqrt{\frac{2g(0;v)}{\nu}}\Big)\|\xi_h\|^2.
\]
By Lemma~\ref{lem:g-properties}(c), $\sqrt{2g(0;v)/\nu}=K(v)$, hence
\begin{equation}\label{eq:K-cond}
\frac{2L\mu}{\lambda_{\tilde w}^{2}}\sqrt{\frac{2g(0;v)}{\nu}}
\;=\; \frac{2L\mu}{\lambda_{\tilde w}^{2}}K(v)
\;=\; \frac{\nu}{2}\cdot\frac{4L\kappa}{\lambda_{\tilde w}^{2}}K(v) .
\end{equation}
Under the data condition~\eqref{eq:data-condition}, $\frac{4L\kappa}{\lambda_{\tilde w}^{2}}K(v)\le 1$, and therefore $\frac{2L\mu}{\lambda_{\tilde w}^{2}}\sqrt{2g(0;v)/\nu}\le \nu/2$, hence
\[
\langle \Hess g(h;v)[\xi_h],\xi_h\rangle_{\mathcal H}
\ge \frac{\nu}{2} \|\xi_h\|_{\mathcal H}^2,
\qquad \forall h\in S,  \forall \xi_h\in T_h\mathcal M.
\]
Since $S$ is geodesically convex by (b), this shows that $g(\cdot;v)|_S$ is geodesically
$\nu/2$-strongly convex.

$S=\{h\in\mathcal M:  g(h;v)\le g(0;v)\}$ is nonempty since $0 \in S$, and by hypothesis~(H2) and Lemma~\ref{lem:convex-radius}(b), $S\subset\calK(u_\NN(\tilde w))\subset\calG(u_\NN(\tilde w))$, where $\calK(u_\NN(\tilde w))$ is compact by Proposition~\ref{prop:manifold}. Since $g(\cdot;v)$ is continuous on $\calH$, the restricted minimization $\min_{h\in S} g(h;v)$ admits a global minimizer $\hat h\in S$, and $g(\hat h;v)\le g(0;v)\le g(h;v)$ for any $h\in\calM\setminus S$ (by the definition of $S$), so $\hat h$ is also a global minimizer of $g(\cdot;v)$ on all of $\calM$.

Finally, by part (b), $S$ is geodesically convex. By part (c), $g(\cdot;v)|_S$ is geodesically
$\nu/2$-strongly convex; we claim that it can have at most one global minimizer on $S$. Indeed,
if $\hat h_1,\hat h_2\in\mathcal M$ are two global minimizers of $g(\cdot;v)$ over $\mathcal M$,
then both satisfy $g(\hat h_i;v)\le g(0;v)$ and thus belong to $S$.
Let $\gamma:[0,1]\to S$ be the (unique) minimizing geodesic joining $\hat h_1$ and $\hat h_2$.
Geodesic strong convexity implies that for any $t\in(0,1)$,
\[
g(\gamma(t);v)
\le (1-t)g(\hat h_1;v)+t g(\hat h_2;v)-\frac{\nu}{4}t(1-t)d_{\mathcal M}(\hat h_1,\hat h_2)^2.
\]
Since $g(\hat h_1;v)=g(\hat h_2;v)=\min_{\mathcal M} g(\cdot;v)$, the right-hand side is strictly
smaller unless $d_{\mathcal M}(\hat h_1,\hat h_2)=0$. Hence $\hat h_1=\hat h_2$.
Therefore, $g(\cdot;v)|_{\mathcal M}$ has \emph{exactly one} global minimizer, and it belongs to $S$.

\medskip
\noindent\textbf{(d) Geodesic smoothness of $g$ on $S$.}
Similarly, for $h\in S$, by \eqref{eq:gHess-decomp}, \eqref{eq:term1_bounds},
\eqref{eq:term2_bound_general}, and \eqref{eq:grad_uniform_S},
\[
\langle \Hess g(h;v)[\xi_h],\xi_h\rangle_{\mathcal H}
\le \Big(\mu+\frac{2L}{\lambda_{\tilde w}^2}\mu\sqrt{\frac{2g(0;v)}{\nu}}\Big)\|\xi_h\|^2.
\]
By the chain of equalities in~\eqref{eq:K-cond} and the data condition~\eqref{eq:data-condition},
$\frac{2L\mu}{\lambda_{\tilde w}^2}\sqrt{\frac{2g(0;v)}{\nu}}\le \nu/2\le \mu/2$ (since $\mu\ge\nu$),
and therefore
\[
\langle \Hess g(h;v)[\xi_h],\xi_h\rangle_{\mathcal H}
\le \frac{3\mu}{2} \|\xi_h\|_{\mathcal H}^2,
\qquad \forall h\in S,  \forall \xi_h\in T_h\mathcal M.
\]
This proves that $g(\cdot;v)|_S$ is geodesically $3\mu/2$-smooth.
\end{proof}

\subsection{Proof of Theorem~\ref{thm:RGF-conv}}\label{app:proof-RGF-conv}
\begin{proof}
\textbf{(a) Convergence of Riemannian Gradient Flow.}
By Theorem~\ref{thm:g}(b), $S=\{h\in\mathcal M(u_\NN(\tilde w)):  g(h;v)\le g(0;v)\}$ is geodesically convex, and by Theorem~\ref{thm:g}(c),
$g(\cdot;v)|_{S}$ is geodesically $\nu/2$-strongly convex.
Moreover, by Theorem~\ref{thm:g}(c), $g(\cdot;v)|_{\mathcal M(u_\NN(\tilde w))}$ admits a unique
global minimizer $h_\infty\in S$.

Along the Riemannian gradient flow $\dot h_t=-\grad g(h_t;v)$, the energy dissipation identity
holds:
\begin{equation}\label{eq:energy_dissipation}
\frac{\ud}{\ud t}g(h_t;v)
=\langle \grad g(h_t;v),\dot h_t\rangle
=-\|\grad g(h_t;v)\|^2 \le 0.
\end{equation}
Hence $t\mapsto g(h_t;v)$ is nonincreasing. Since $h_0\in S$ and $S$ is a sublevel set,
\eqref{eq:energy_dissipation} implies $h_t\in S$ for all $t\ge 0$.

\textbf{(b) Exponential Decay of Energy Gap.}
Define $R_t:=g(h_t;v)-g(h_\infty;v)\ge 0$.
Since $h_t\in S$ for all $t\ge 0$, we may use geodesic strong convexity on $S$.
A standard consequence of geodesic $\nu/2$-strong convexity is the Polyak--\L{}ojasiewicz (PL) inequality (see \citep[Lemma~11.28]{boumal2023intromanifolds}):
\begin{equation}\label{eq:PL}
\|\grad g(h;v)\|^2 \ge 2 \cdot \nu/2 \cdot \big(g(h;v)-g(h_\infty;v)\big) \qquad \forall h\in S.
\end{equation}
Combining \eqref{eq:energy_dissipation} and \eqref{eq:PL} yields the differential inequality
\[
\frac{d}{dt}R_t
= -\|\grad g(h_t;v)\|^2
\le -\nu R_t.
\]
By Gr\"onwall's inequality,
\begin{equation*}
R_t\le R_0 e^{-\nu t},\qquad t\ge 0.
\end{equation*}

\textbf{(c) Exponential Decay of Flow Gap.}
Let $h_\infty=\arg\min_{h\in S} g(h;v)$ and set $D_t:=d_{\mathcal M}(h_t,h_\infty)$.
Let $\eta_t:=\Log_{h_t}(h_\infty)\in T_{h_t}\mathcal M$ (see \citep[Definition~10.2]{boumal2023intromanifolds}), so that $\|\eta_t\|=D_t$.
Define $\phi(t):=\frac12 D_t^2$. By the first variation formula (see \citep[Theorem~6.3]{lee2006riemannian}),
\begin{equation}\label{eqn:first-variation}
\phi'(t)=\Big\langle -\Log_{h_t}(h_\infty),\dot h_t\Big\rangle
=\langle \eta_t,\grad g(h_t;v)\rangle,
\end{equation}
since $\dot h_t=-\grad g(h_t;v)$.
Because $g(\cdot;v)|_S$ is geodesically $\nu/2$-strongly convex, we have the following two inequalities:
\begin{equation}\label{eqn:geo-con-1}
g(h_\infty;v) \geq g(h_t;v) + \ip{\grad g(h_t;v),\Log_{h_t}(h_\infty)} + \frac{\nu/2}{2}d^2_{\mathcal{M}}(h_t,h_\infty)
\end{equation}
and
\begin{equation}\label{eqn:geo-con-2}
g(h_t;v) \geq g(h_\infty;v) + \ip{\grad g(h_\infty;v), \Log_{h_\infty}(h_t)} + \frac{\nu/2}{2}d^2_{\mathcal{M}}(h_t,h_\infty).
\end{equation}

Summing up \eqref{eqn:geo-con-1} and \eqref{eqn:geo-con-2}, and noticing the first-order optimality condition $\grad g(h_\infty;v) = 0$ as well as the previously defined $\eta_t = \Log_{h_t}(h_\infty)$, we have that, for all
$h_t \in S$,
\begin{equation}\label{eqn:quadratic}
\langle \grad g(h_t;v),\eta_t\rangle \le -\nu/2 \cdot d_{\mathcal M}^2(h_t,h_\infty).
\end{equation}
\eqref{eqn:first-variation} and \eqref{eqn:quadratic} yield $\phi'(t)\le -\nu \phi(t)$; hence
\begin{equation}\label{eq:Dt_contract}
D_t \le D_0 e^{-\nu t/2},\qquad t\ge 0.
\end{equation}

In particular, $D_t\to 0$ as $t\to\infty$, hence $h_t\to h_\infty$.

Finally, since $\mathcal M$ is an embedded submanifold of $\mathcal H$ with the induced metric,
the intrinsic distance dominates the ambient chord length, i.e.,
$\|h_t-h_\infty\|_{\mathcal H}\le d_{\mathcal M}(h_t,h_\infty)=D_t$.
Combining with \eqref{eq:Dt_contract} gives
\[
\|h_t-h_\infty\|_{\mathcal H}\le D_t \le D_0 e^{-\nu t/2},
\]
which completes the proof.
\end{proof}

\section{Convergence of the Minimizing Movement Scheme in Hilbert Space}\label{sec:conv-mms}
For completeness, we include the standard result on the convergence of the minimizing movement scheme in Hilbert space \citep{ambrosio2005gradient, mielke2023introduction}.

\begin{theorem}\label{thm:mms-convergence}
Let $\calH$ be a Hilbert space, and let $\calF:\calH\to\mathbb{R}$ be continuously Fr\'echet differentiable, coercive, and $m_{\calF}$-strongly convex with $m_\calF>0$. Assume moreover that, for every $E\in\mathbb{R}$, the sublevel set
\[
S_E^{\calF}:=\{u\in\calH:\calF[u]\le E\}
\]
is compact in $\calH$. Given $u^0\in \mathrm{dom}(\calF)$ and $\tau>0$, define the minimizing movement scheme by
\begin{equation}\label{eq:appendix-mms}
u^{n+1}=J_\tau(u^n),
\qquad
J_\tau(x):=\prox_{\tau\calF}(x)
:=\argmin_{u\in\calH}
\left\{
\frac{1}{2\tau}\|u-x\|_{\calH}^2+\calF[u]
\right\}.
\end{equation}
If $1+\tau m_{\calF}>0$, then the minimization problem in \eqref{eq:appendix-mms} admits a unique solution, making $J_\tau$ well defined and single-valued.

Let $\hat u_\tau:[0,T]\to\calH$ denote the piecewise linear interpolation associated with the sequence $\{u^n\}_{n\ge0}$. As $\tau\downarrow0$, the interpolants $\hat u_\tau$ converge in $\calH$, uniformly on $[0,T]$, to the unique solution
\(
u\in W^{1,2}([0,T];\calH) 
\)
of the gradient flow equation $u_t = - \nabla \calF[u]$ associated with $\calF$ and the initial datum $u^0$. Moreover, for any two gradient flow solutions $u_0,u_1$, one has the $m_{\calF}$-contractivity estimate
\begin{equation*}
\|u_1(t)-u_0(t)\|_{\calH}
\le
e^{-m_{\calF}(t-s)}
\|u_1(s)-u_0(s)\|_{\calH},
\qquad 0\le s<t.
\end{equation*}
\end{theorem}

\begin{proof} 

\noindent \textbf{Step 1: Well-posedness and Energy Estimate.} The objective function in \eqref{eq:appendix-mms} is given by $\Phi(u) = \frac{1}{2\tau}\|u-u^n\|_{\calH}^2+\calF[u]$. Since $\calF$ is continuously Fr\'echet differentiable, coercive, and has compact sublevel sets, $\Phi(u)$ is coercive and bounded from below, guaranteeing the existence of a minimizer. Furthermore, because $\calF$ is $m_{\calF}$-convex, $\Phi(u)$ is strictly convex provided $1 + \tau m_{\calF} > 0$. Thus, $u^{n+1}$ exists and is unique. Moreover, $\Phi(u^{n+1}) \le \Phi(u^n)$ indicates
\begin{equation}
 \frac{1}{2 \tau} \| u^{n+1} - u^n \|^2 \leq \mathcal{F}[u^{n}] - \mathcal{F}[u^{n+1}]
\end{equation}
Summing this inequality over $n=0, \dots, N-1$, we have
\begin{equation}\label{sum_incre}
 \sum_{n=0}^{N-1}\| u^{n+1} - u^n \|^2 \leq 2 \tau (\mathcal{F}(u^0) - \mathcal{F}(u^*) ) \leq \tau C  ,
\end{equation}
where $u^*$ is the global minimizer of $\mathcal{F}$. Hence,
\begin{equation}
  \| u^{n+1} - u^n \|^2 \rightarrow 0, \quad n \rightarrow \infty
\end{equation}

\noindent \textbf{Step 2: Interpolations and Compactness.} Define the piecewise affine interpolation $\hat u_\tau:[0,T]\to\mathcal H$ and the piecewise constant interpolation $\bar u_\tau:[0,T]\to\mathcal H$ by
\[
\hat u_\tau((n+\theta)\tau)
=(1-\theta)u^n+\theta u^{n+1},
\qquad \theta\in[0,1],\quad n=0,\dots,N-1,
\]
and
\[
\bar u_\tau(t)=u^{n+1},
\qquad t\in(n\tau,(n+1)\tau],\quad n=0,\dots,N-1.
\]
Then, (\ref{sum_incre}) indicates that
\begin{equation}
\int_{0}^T  \| \dot{\hat{u}}_{\tau} \|^2 \dd t = \sum_{n=1}^N \tau \left \| \frac{u^{n+1} - u^n}{\tau} \right \|^2 \leq 2 (\calF[u^0] - \calF[u^*])
\end{equation}
Hence, $\hat{u}_{\tau}$ is bounded in $W^{1, 2} ([0, T], \calH)$ and we can extract a subsequence (not relabeled) such that
\begin{equation}
\hat{u}_{\tau} \rightharpoonup u ~\text{in}~ L^2([0, T]; \calH), \quad \dot{\hat{u}}_{\tau} \rightharpoonup \dot{u} ~\text{in}~ L^2([0, T]; \calH)
\end{equation}
for a limit $u \in W^{1, 2} ([0, T]; \calH)$.

Moreover, for any $0\le s<t\le T$, since $\hat u_\tau$ is absolutely continuous on $[0,T]$, we have
\(
\hat u_\tau(t)-\hat u_\tau(s)=\int_s^t \dot{\hat u}_\tau(r)\,dr.
\)
Therefore, by the Cauchy--Schwarz inequality, we have
\[
\|\hat u_\tau(t)-\hat u_\tau(s)\|_{\mathcal H}
\le \int_s^t \|\dot{\hat u}_\tau(r)\|_{\mathcal H}\,dr
\le |t-s|^{1/2}\|\dot{\hat u}_\tau\|_{L^2([0,T];\mathcal H)} \le |t-s|^{1/2} \bigl(2(\mathcal F[u^0]-\calF[u^*])\bigr)^{1/2}
\]
Hence $\{\hat u_\tau\}_\tau$ is uniformly H\"older continuous on every $[0,T]$. Since $\hat{u}_{\tau}$ lie in the compact sublevel set $S_{\mathcal{F}(u^0)}^\calF$, by the Arzel\`a--Ascoli theorem, by extracting a further subsequence (not relabeled), we have the uniform convergence
\begin{equation}
 \hat{u}_{\tau} \rightarrow u ~\text{in}~ C^0([0, T]; \mathcal{H}) 
\end{equation}
Notice $\hat{u} (n \tau) = \bar{u} (n \tau)$, and $\bar{u}_{\tau}$ is piecewise constant, we have
\begin{equation}
 \| \hat{u}_{\tau} - \bar{u}_{\tau} \|_{L^{\infty} ([0, T]; \calH)} \leq \sqrt{\tau} \bigl(2(\mathcal F[u^0]-\calF[u^*])\bigr)^{1/2} \rightarrow 0, \quad \tau \rightarrow 0^+.
\end{equation}
Hence, for the piecewise constant interpolation, we have
\begin{equation}
  \bar{u}_{\tau} \rightarrow u ~\text{in}~ L^{\infty} ([0, T]; \mathcal{H})
\end{equation}

\noindent \textbf{Step 3: Convergence to the Gradient Flow.} Next, we show that the limit $u$ satisfies the gradient flow equation
\begin{equation}
u_t = - \nabla \mathcal{F}(u)
\end{equation}
Since $\mathcal F$ is Fr\'echet continuously differentiable, the Euler--Lagrange equation for the minimizing movement scheme reads
\[
\frac{u^{n+1}-u^n}{\tau}=-\nabla\mathcal F[u^{n+1}],
\qquad n=0,1,\dots,N-1.
\]
Recalling the piecewise affine interpolation $\hat u_\tau$ and the piecewise constant interpolation $\bar u_\tau$, we have
\begin{equation}\label{eq:gf-discrete-interpolant}
\dot{\hat u}_\tau(t)=-\nabla\mathcal F[\bar u_\tau(t)]
\qquad \text{for a.e. } t\in(0,T).
\end{equation}
We have proved that up to a subsequence,
\[
\hat u_\tau \to u \qquad \text{in } C^0([0,T];\mathcal H), \dot{\hat u}_\tau \rightharpoonup \dot u
\qquad \text{weakly in } L^2([0,T];\mathcal H), \bar u_\tau \to u
\qquad \text{in } L^\infty([0,T];\mathcal H).
\]
Since the trajectories $\bar u_\tau(t)$ remain in the compact sublevel set
\(
S_{\mathcal F[u^0]}^{\mathcal F} \)
and since $\nabla\mathcal F$ is continuous on bounded subsets of $\mathcal H$, it follows that
\[
\nabla\mathcal F[\bar u_\tau]\to \nabla\mathcal F[u]
\qquad \text{in } L^\infty([0,T];\mathcal H),
\]
and therefore also in $L^2([0,T];\mathcal H)$. We now pass to the limit in \eqref{eq:gf-discrete-interpolant}. For any $\phi\in L^2([0,T];\mathcal H)$,
\[
\int_0^T \langle \dot{\hat u}_\tau(t),\phi(t)\rangle_{\mathcal H}\,dt
=
-\int_0^T \langle \nabla\mathcal F[\bar u_\tau(t)],\phi(t)\rangle_{\mathcal H}\,dt.
\]
Letting $\tau\to0$, the weak convergence of $\dot{\hat u}_\tau$ and the strong convergence of $\nabla\mathcal F[\bar u_\tau]$ yield
\[
\int_0^T \langle \dot u(t),\phi(t)\rangle_{\mathcal H}\,dt
=
-\int_0^T \langle \nabla\mathcal F[u(t)],\phi(t)\rangle_{\mathcal H}\,dt.
\]
Since $\phi$ is arbitrary, we conclude that
\[
\dot u=-\nabla\mathcal F[u]
\qquad \text{in } L^2([0,T];\mathcal H).
\]
That is, $u$ satisfies the gradient flow equation on $[0,T]$. Finally, since $\hat u_\tau(0)=u^0$ for all $\tau$ and $\hat u_\tau\to u$ in $C([0,T];\mathcal H)$, we also have
\(
u(0)=u^0.
\)
Therefore $u$ is a solution of the gradient flow equation with initial datum $u^0$.

\noindent \textbf{Step 4: Contractivity and Uniqueness.} Since $\mathcal F:\mathcal H\to\mathbb R$ is $m_{\mathcal F}$-convex, i.e.,
\[
\mathcal F[v]\ge \mathcal F[u]
+\langle \nabla\mathcal F[u],\,v-u\rangle_{\mathcal H}
+\frac{m_{\mathcal F}}{2}\|v-u\|_{\mathcal H}^2,
\qquad \forall u,v\in\mathcal H.
\]
Then the gradient $\nabla\mathcal F$ is $m_{\mathcal F}$-monotone:
\begin{equation}\label{eq:monotone-gradient}
\langle \nabla\mathcal F[u]-\nabla\mathcal F[v],\,u-v\rangle_{\mathcal H}
\ge m_{\mathcal F}\|u-v\|_{\mathcal H}^2,
\qquad \forall u,v\in\mathcal H.
\end{equation}
Let $u_0,u_1\in W^{1,2}([0,T];\mathcal H)$ be two solutions of the gradient flow equation
\[
\dot u_j(t)=-\nabla\mathcal F[u_j(t)],
\qquad j=0,1.
\]
Since $u_0,u_1\in W^{1,2}([0,T];\mathcal H)$, using the gradient flow equation and \eqref{eq:monotone-gradient}, we obtain
\[
\begin{aligned}
\frac12\frac{d}{dt}\|u_1(t)-u_0(t)\|_{\mathcal H}^2
&=
-\langle \nabla\mathcal F[u_1(t)]-\nabla\mathcal F[u_0(t)],\,u_1(t)-u_0(t)\rangle_{\mathcal H} \\
&\le
-\,m_{\mathcal F}\|u_1(t)-u_0(t)\|_{\mathcal H}^2 .
\end{aligned}
\]
By Gr\"onwall's inequality, for all $0\le s<t\le T$,
\begin{equation*}
\|u_1(t)-u_0(t)\|_{\mathcal H}
\le e^{-m_{\mathcal F}(t-s)}\|u_1(s)-u_0(s)\|_{\mathcal H}.
\end{equation*}
In particular, if $u_1(0)=u_0(0)=u^0$, then $u_1\equiv u_0$ on $[0,T]$. Hence the gradient flow solution with initial datum $u^0$ is unique.
\end{proof}

The previous theorem establishes that the minimizing movement scheme converges, as $\tau\downarrow0$, to the solution of the gradient flow of $\calF$ on every finite time interval. To relate this evolution to the minimization problem for $\calF$, one further needs the long-time behavior of the gradient flow. Under the assumption $m_{\calF}>0$, the latter converges exponentially fast to the unique global minimizer $u^\ast$ of $\calF$. 

\begin{remark}
The contractivity estimate above is the continuous-time counterpart of the
contraction property of the proximal map associated with the minimizing movement
scheme. Indeed, let
\[
    J_\tau(x)
    :=
    \arg\min_{u\in\mathcal H}
    \left\{
        \frac{1}{2\tau}\|u-x\|_{\mathcal H}^2
        +
        \mathcal F[u]
    \right\}.
\]
If \(u=J_\tau(x)\) and \(v=J_\tau(y)\), then the optimality conditions give
\[
    \frac{u-x}{\tau}+\nabla\mathcal F[u]=0,
    \qquad
    \frac{v-y}{\tau}+\nabla\mathcal F[v]=0 .
\]
Subtracting the two identities and testing with \(u-v\), we obtain
\[
    \frac{1}{\tau}\|u-v\|_{\mathcal H}^2
    +
    \left\langle
    \nabla\mathcal F[u]-\nabla\mathcal F[v],
    u-v
    \right\rangle_{\mathcal H}
    =
    \frac{1}{\tau}\langle x-y,u-v\rangle_{\mathcal H}.
\]
Using the \(m_{\mathcal F}\)-monotonicity of \(\nabla\mathcal F\), it follows that
\[
    \left(\frac{1}{\tau}+m_{\mathcal F}\right)
    \|u-v\|_{\mathcal H}^2
    \le
    \frac{1}{\tau}
    \|x-y\|_{\mathcal H}\|u-v\|_{\mathcal H}.
\]
Therefore,
\[
    \|J_\tau(x)-J_\tau(y)\|_{\mathcal H}
    \le
    \frac{1}{1+\tau m_{\mathcal F}}
    \|x-y\|_{\mathcal H}.
\]
Thus the continuous gradient flow is \(e^{-m_{\mathcal F}t}\)-contractive,
whereas one step of the minimizing movement scheme is
\((1+\tau m_{\mathcal F})^{-1}\)-contractive. This discrete contraction is the
stability mechanism used in the error-propagation analysis for the neural MMS
iterations.
\end{remark}

\begin{theorem}[Exponential convergence to the minimizer]\label{thm:gf-exp-conv}
Let $\calH$ be a Hilbert space, and let $\calF:\calH\to\mathbb R$ be Fr\'echet differentiable and $m_{\calF}$-convex with $m_{\calF}>0$, i.e.,
\[
\calF[v]\ge \calF[u]
+\langle \nabla\calF[u],\,v-u\rangle_{\calH}
+\frac{m_{\calF}}{2}\|v-u\|_{\calH}^2,
\qquad \forall u,v\in\calH.
\]
Assume that $\calF$ admits a global minimizer $u^\ast\in\calH$. Then $u^\ast$ is unique. Moreover, for every initial datum $u^0\in\calH$, the corresponding gradient flow solution
\[
\dot u(t)=-\nabla\calF[u(t)],
\qquad u(0)=u^0,
\]
satisfies
\[
\|u(t)-u^\ast\|_{\calH}
\le e^{-m_{\calF}t}\|u^0-u^\ast\|_{\calH},
\qquad t\ge0.
\]
In particular,
\[
u(t)\to u^\ast
\qquad \text{in }\calH \quad \text{as } t\to\infty.
\]
\end{theorem}

\begin{proof}
Since $u^\ast$ is a global minimizer and $\calF$ is Fr\'echet differentiable, we have
\[
\nabla\calF[u^\ast]=0,
\]
so $u^\ast$ is a stationary solution of the gradient flow equation. By the $m_{\calF}$-monotonicity of $\nabla\calF$,
\[
\langle \nabla\calF[u]-\nabla\calF[v],\,u-v\rangle_{\calH}
\ge m_{\calF}\|u-v\|_{\calH}^2,
\qquad \forall u,v\in\calH.
\]
As proved in Theorem (\ref{thm:mms-convergence}), we have
\[
\|u(t)-u^\ast\|_{\calH}
\le e^{-m_{\calF}t}\|u^0-u^\ast\|_{\calH}.
\]
This proves the claim.
\end{proof}

In conclusion, combining the previous two theorems, the discrete iterate $u^N$ first approximates the gradient flow solution at time $T=N\tau$, and for large times this state in turn approaches $u^\ast$. In this sense, $u^N$ provides an approximation of the global minimizer. Formally, the total approximation may be viewed as consisting of two parts: a discretization error between $u^N$ and the gradient flow state $u(N\tau)$, and a long-time relaxation error between $u(N\tau)$ and the minimizer $u^\ast$.

\section{Discussion about Assumptions}

\subsection{Locality (Assumption~\ref{ass:intro:locality})}\label{app:suff-locality}

Assumption~\ref{ass:intro:locality} requires that, at each step $n$, the increment problem~\eqref{eq:intro:evnn_increment} admits at least one global minimizer in the local neighborhood $\calM(u_\NN(w^n))$. In this appendix we give a verifiable sufficient condition. The condition has two ingredients:
\begin{enumerate}
\item[(I)] \emph{Output isolation:} parameters $w$ outside $B(w^n,r_{w^n})$ cannot produce network outputs arbitrarily close to $v=u_\NN(w^n)$; equivalently, the \emph{output isolation radius}
\begin{equation}\label{eq:rho-isolation}
\rho_{w^n}\;:=\;\dist_{\mathcal H}\Big(v,\ u_\NN\big(\R^p\setminus B(w^n,r_{w^n})\big)\Big)\;>\;0.
\end{equation}
\item[(II)] A step size $\tau$ small enough relative to $\rho_{w^n}$ and the gradient $\|\nabla\calF[v]\|_\calH$ at the current iterate.
\end{enumerate}
The intuition is that the proximal term penalizes large increments, so any global minimizer of $g(\cdot;v)$ must produce a network output close to $v$ in $\calH$. If outputs near $v$ can only come from parameters near $w^n$, then a global minimizer must lie in $\calM(v)$.

\begin{theorem}\label{thm:small-tau-locality}
Fix step $n$ and let $v:=u_\NN(w^n)\in\calH$. Assume $u_\NN$ has an $L$-Lipschitz Jacobian and a $\lambda_{w^n}$-non-degenerate Jacobian at $w^n$, and let $r_{w^n}>0$ be the radius from Lemma~\ref{lem:gram-lower-bound}, so that
\[
\calM(v):=\{u_\NN(w)-v:\,\|w-w^n\|_2 < r_{w^n}\}\subset\calH
\]
is the local increment set. Assume the output isolation radius $\rho_{w^n}>0$.
If $\tau>0$ satisfies
\begin{equation}\label{eq:tau-threshold-locality}
\rho_{w^n}\ \ge\ \frac{2\tau\,\|\nabla\calF[v]\|_\calH}{1+\tau m_\calF},
\end{equation}
then $\inf_{h\in\mathcal{S}_n}g(h;v)=\inf_{h\in\calM(v)}g(h;v)$, where $\mathcal{S}_n:=\{u_\NN(w)-v:w\in\R^p\}$. In particular, every global minimizer of $g(\cdot;v)$ on $\calM(v)$ (which exists and is unique under the conditions of Theorem~\ref{thm:g}) is also a global minimizer over $\mathcal{S}_n$, and Assumption~\ref{ass:intro:locality} holds at step $n$.

\smallskip
\noindent
Moreover, if \eqref{eq:tau-threshold-locality} holds with strict inequality, every global minimizer $h^*\in\argmin_{\mathcal{S}_n}g(\cdot;v)$ is realized by some parameter $w^*\in B(w^n,r_{w^n})$ with the quantitative bounds
\begin{equation}\label{eq:k-bound-thm}
\|u_\NN(w^*)-v\|_\calH\;\le\;\frac{2\tau\,\|\nabla\calF[v]\|_\calH}{1+\tau m_\calF},
\end{equation}
\begin{equation}\label{eq:w-bound-thm}
\|w^*-w^n\|_2\;\le\;\frac{1}{\lambda_{w^n}}\cdot\frac{2\tau\,\|\nabla\calF[v]\|_\calH}{1+\tau m_\calF}.
\end{equation}
\end{theorem}

\begin{proof}
Since $\calM(v)\subset\mathcal{S}_n$, $\inf_{\mathcal{S}_n}g\le \inf_{\calM(v)}g$. We show the reverse inequality. Note $0\in\calM(v)$ (taking $w=w^n$), so $\inf_{\calM(v)}g\le g(0;v)=\calF[v]$.

It suffices to show $g(h;v)\ge \calF[v]$ for every $h\in\mathcal{S}_n\setminus\calM(v)$. Let $h\in\mathcal{S}_n\setminus\calM(v)$, so $h=u_\NN(w)-v$ for some $w\notin B(w^n,r_{w^n})$. By the definition of $\rho_{w^n}$ in \eqref{eq:rho-isolation},
\[
\|h\|_\calH=\|u_\NN(w)-v\|_\calH\ge \rho_{w^n}.
\]
By $m_\calF$-strong convexity of $\calF$,
\[
\calF[v+h]\ \ge\ \calF[v]+\langle\nabla\calF[v],h\rangle_\calH+\frac{m_\calF}{2}\|h\|_\calH^2.
\]
Substituting into the definition of $g(h;v)=\|h\|_\calH^2/(2\tau)+\calF[v+h]$ and applying Cauchy--Schwarz to $\langle\nabla\calF[v],h\rangle_\calH\ge -\|\nabla\calF[v]\|_\calH\|h\|_\calH$:
\begin{equation}\label{eq:g-quadratic-lower-bound}
g(h;v)\ \ge\ \frac{1+\tau m_\calF}{2\tau}\|h\|_\calH^2-\|\nabla\calF[v]\|_\calH\,\|h\|_\calH+\calF[v].
\end{equation}
The right-hand side, viewed as a quadratic in $t=\|h\|_\calH\ge 0$, is non-negative on top of $\calF[v]$ when $t\ge \frac{2\tau\|\nabla\calF[v]\|_\calH}{1+\tau m_\calF}$. Under \eqref{eq:tau-threshold-locality}, this threshold is at most $\rho_{w^n}\le \|h\|_\calH$, so $g(h;v)\ge \calF[v]$.

Hence $\inf_{\mathcal{S}_n\setminus\calM(v)}g\ge\calF[v]\ge\inf_{\calM(v)}g$, giving $\inf_{\mathcal{S}_n}g=\inf_{\calM(v)}g$. By Theorem~\ref{thm:g}, $g(\cdot;v)|_{\calM(v)}$ admits a unique global minimizer $h_\infty$, which is therefore also a global minimizer over $\mathcal{S}_n$. Assumption~\ref{ass:intro:locality} holds at step $n$.

\medskip
\noindent\emph{Quantitative bounds.}
Now suppose \eqref{eq:tau-threshold-locality} holds with strict inequality, and let $h^*\in\argmin_{\mathcal{S}_n}g(\cdot;v)$ with $h^*=u_\NN(w^*)-v$ for some $w^*\in\R^p$.

\emph{Step 1 (energy bound).}
Since $0\in\mathcal{S}_n$ (take $w=w^n$), $g(h^*;v)\le g(0;v)=\calF[v]$. With $h^*=u_\NN(w^*)-v$ and $\calF$ being $m_\calF$-strongly convex,
\[
g(h^*;v)
=\frac{1}{2\tau}\|h^*\|_\calH^2+\calF[v+h^*]
\;\ge\;\frac{1+\tau m_\calF}{2\tau}\|h^*\|_\calH^2-\|\nabla\calF[v]\|_\calH\|h^*\|_\calH+\calF[v].
\]
Combining with $g(h^*;v)\le\calF[v]$ and dividing through by $\|h^*\|_\calH$ (or noting the inequality holds trivially when $h^*=0$) gives \eqref{eq:k-bound-thm}.

\emph{Step 2 (parameter-space localization).}
By \eqref{eq:k-bound-thm} and the strict version of \eqref{eq:tau-threshold-locality},
\[
\|u_\NN(w^*)-v\|_\calH\;\le\;\frac{2\tau\,\|\nabla\calF[v]\|_\calH}{1+\tau m_\calF}\;<\;\rho_{w^n}.
\]
By definition of $\rho_{w^n}$ in \eqref{eq:rho-isolation}, this rules out $w^*\notin B(w^n,r_{w^n})$, so $w^*\in B(w^n,r_{w^n})$.

\emph{Step 3 (norm bound in $\R^p$).}
Set $e:=w^*-w^n$ and $w_t:=w^n+te$ for $t\in[0,1]$; convexity of $B(w^n,r_{w^n})$ yields $w_t\in B(w^n,r_{w^n})$. The fundamental theorem of calculus gives
\[
u_\NN(w^*)-u_\NN(w^n)
\;=\;\D u_\NN(w^n)[e]+\int_0^1\big(\D u_\NN(w_t)-\D u_\NN(w^n)\big)[e]\,dt.
\]
By Definition~\ref{def:non-degen-jacobian}, $\|\D u_\NN(w^n)[e]\|_\calH\ge 2\lambda_{w^n}\|e\|_2$. The $L$-Lipschitz Jacobian assumption (Definition~\ref{def:NN-Lip-L}) bounds the remainder:
\[
\Big\|\!\int_0^1\!\big(\D u_\NN(w_t)-\D u_\NN(w^n)\big)[e]\,dt\Big\|_\calH
\le\int_0^1\!Lt\|e\|_2^2\,dt=\frac{L}{2}\|e\|_2^2.
\]
Combining,
\[
\|u_\NN(w^*)-v\|_\calH
\;\ge\;2\lambda_{w^n}\|e\|_2-\frac{L}{2}\|e\|_2^2.
\]
Since $\|e\|_2 < r_{w^n}\le \lambda_{w^n}/L$ by \eqref{eq:rwdef}, the quadratic remainder satisfies $\tfrac{L}{2}\|e\|_2^2 < \tfrac{\lambda_{w^n}}{2}\|e\|_2$, so
\[
\|u_\NN(w^*)-v\|_\calH\;\ge\;\frac{3\lambda_{w^n}}{2}\|e\|_2\;\ge\;\lambda_{w^n}\|e\|_2.
\]
Combining with \eqref{eq:k-bound-thm} yields \eqref{eq:w-bound-thm}.
\end{proof}

\begin{remark}[Interpretation and $n$-uniform locality]\label{rem:rho-positive}
Condition~\eqref{eq:tau-threshold-locality} couples the step size $\tau$ to the local gradient $\|\nabla\calF[v]\|_\calH$. Two natural regimes:
\begin{itemize}
\item \emph{Near a minimizer:} as $v\to u^*$, $\|\nabla\calF[v]\|_\calH\to 0$, and the condition becomes vacuous. This matches the intuition that locality is automatic once iterates are close to the optimum.
\item \emph{Far from the minimizer:} for fixed $\rho_{w^n}$ and $\|\nabla\calF[v]\|_\calH$, \eqref{eq:tau-threshold-locality} is equivalent to $\tau\,(2\|\nabla\calF[v]\|_\calH-\rho_{w^n}m_\calF)\le \rho_{w^n}$. If $\rho_{w^n}m_\calF\ge 2\|\nabla\calF[v]\|_\calH$ the condition holds for any $\tau>0$; otherwise it imposes $\tau\le \rho_{w^n}/(2\|\nabla\calF[v]\|_\calH-\rho_{w^n}m_\calF)$.
\end{itemize}

\paragraph{Output isolation $\rho_{w^n}>0$ and global inverse-stability.}
The condition $\rho_{w^n}>0$ rules out ``near-collisions'' of the network map around $v$ from parameters far away from $w^n$. Two sufficient conditions are useful in practice:
\begin{enumerate}[label=\textnormal{(i)},leftmargin=*,nosep]
\item \emph{Global inverse-stability:} if $u_\NN$ satisfies~\eqref{eq:global-inverse-stability} with constant $\sigma>0$, then $\rho_{w^n}\ge\sigma\,r_{w^n}$, and~\eqref{eq:global-inverse-stability} alone supplies all locality bounds.
\item \emph{Local injectivity at $w^n$:} if $u_\NN$ is injective on $\R^p\setminus B(w^n,r_{w^n})$ relative to a neighborhood of $v$, $\rho_{w^n}$ is bounded below by a positive constant depending on the local injectivity radius.
\end{enumerate}
In over-parameterized networks, the parametrization is typically \emph{not} injective on the whole $\R^p$ (e.g., permutations of hidden units and sign flips leave $u_\NN$ unchanged), but injectivity holds modulo these symmetries, and on an open neighborhood of generic initializations the local injectivity (ii) suffices.

\paragraph{$n$-uniform locality (used by the global theorems).}
For Theorem~\ref{thm:small-tau-locality} to underwrite Assumption~\ref{ass:intro:locality} at \emph{every} iterate $w^n$ along the trajectory, the threshold~\eqref{eq:tau-threshold-locality} must hold uniformly in $n$. This requires $\inf_n\rho_{w^n}>0$, $\sup_n\|\nabla\calF[u_\NN(w^n)]\|_\calH<\infty$, and $\inf_n\lambda_{w^n}>0$. The latter two are derived uniformities of Assumption~\ref{assm:traj-unif} (Remark~\ref{rem:traj-unif}): $\inf_n\lambda_{w^n}\ge\lambda_\infty$ is~(U1), and $\sup_n\|\nabla\calF[u_\NN(w^n)]\|_\calH\le L_\calF\sup_n\|u^n_{\NN,T}-u^*\|_\calH<\infty$ follows from $L_\calF$-Lipschitz gradient plus Theorem~\ref{thm:global_error}(b). The first, $\inf_n\rho_{w^n}\ge\rho_\infty>0$, is exactly~(U3) of Assumption~\ref{assm:traj-unif}, and is automatic under global inverse-stability~\eqref{eq:global-inverse-stability} (since then $\rho_\infty\ge\sigma r_\infty$ with $r_\infty>0$ from Remark~\ref{rem:traj-unif}).

Consequently, under Assumption~\ref{assm:traj-unif} (or under the standing assumptions plus global inverse-stability of $u_\NN$), there exists a single $\tau$ for which Theorem~\ref{thm:small-tau-locality} applies at every iterate, and Assumption~\ref{ass:intro:locality} holds along the entire trajectory. The threshold on $\tau$ is
\[
\tau \;\le\; \frac{\rho_\infty^{2}}{2\bigl(\sup_n\calF[u^n_{\NN,T}]-\calF_{\inf}\bigr)},
\]
which is a finite positive number under the trajectory uniformities.
\end{remark}

\begin{remark}[Strength of the inner output-isolation~\eqref{eq:rho-in-def}]\label{rem:rho-in-strength}
The bi-Lipschitz bound proved in Proposition~\ref{prop:manifold} gives $\|u_\NN(w)-v\|_\calH \ge \lambda_{\tilde w}\|w-\tilde w\|_2$ for $w$ in the outer ball, so the shell $r^{\rm in}_{\tilde w}\le \|w-\tilde w\|_2 < r_{\tilde w}$ contributes at least $\lambda_{\tilde w} r^{\rm in}_{\tilde w}$ to $\rho^{\rm in}_{\tilde w}$. Positivity of $\rho^{\rm in}_{\tilde w}$ then reduces to a far-field separation condition $\dist_\calH(v,\,u_\NN(\R^p\setminus B(\tilde w,r_{\tilde w})))>0$, the same form of local injectivity addressed by Remark~\ref{rem:rho-positive} above.
\end{remark}

\subsection{Jacobian non-degeneracy (Definition~\ref{def:non-degen-jacobian})}\label{app:disc-nondegen}

The non-degeneracy condition $\D u_\NN(w)^*\D u_\NN(w)\succeq 4\lambda^2 I_p$ requires the
Gram matrix of the partial derivatives $\partial_{w_j} u_\NN(\cdot;w)$ (with respect to the
$\calH$-inner product) to be uniformly positive definite. We discuss when this holds and
when it can fail.

\paragraph{When it holds.}
The non-degeneracy is essentially a generic property in the underparameterized regime
($p\le N$), provided the architecture is rich enough and the data $\{x_i\}_{i=1}^N$ are in
``general position.'' Concretely:
\begin{itemize}
\item For two-layer tanh or sigmoid networks with $p\le N$ and generic random initializations,
the empirical Gram matrix $J(w)^\top J(w)\in\R^{p\times p}$ is almost surely positive definite.
\item For deep fully-connected networks with smooth activations, results from
\citet{du2018gradient} and subsequent works show that with high probability over random
initialization, $\sigma_{\min}(J(w^0))$ is bounded away from zero by an architecture-dependent
constant when the width is sufficient relative to the input dimension and sample size.
\item Empirically, our experiments (Figure~\ref{fig:app_Jacobian}) confirm that
$s_{\min}(J(w))$ stays at the $O(10^{-2})$ level along the GN trajectory in
underparameterized regimes.
\end{itemize}

\paragraph{When it fails or weakens.}
The non-degeneracy can fail or weaken in several practically important scenarios:
\begin{itemize}
\item \emph{Overparameterization} ($p\gg N$): the Gram matrix
$J(w)^\top J(w)\in\R^{p\times p}$ has rank at most $N$, so it is rank-deficient and
cannot satisfy the strict bound $\succeq 4\lambda^2 I_p$. In this regime, our theory
applies after a damping fix: replacing $J^\top J$ with $J^\top J + \rho I$ for some $\rho>0$
(Levenberg--Marquardt-type regularization), or working with the dual Gram matrix
$JJ^\top\in\R^{N\times N}$ which is full-rank (cf.\ \citet{cai2019gram}).
The numerical experiments in Section~\ref{sec:experiments} demonstrate that the damped
GN flow continues to perform well in this regime.
\item \emph{Symmetries and parameter redundancy}: networks with permutation symmetries of
hidden units have automatically a degenerate Hessian along symmetry directions. The
non-degeneracy condition is then satisfied not on the full $\R^p$, but on a quotient by
the symmetry group; our local analysis tolerates this since it only requires non-degeneracy
in a neighborhood of the current iterate (after fixing a representative).
\item \emph{Activation saturation}: at large weights, common activations such as $\tanh$
and sigmoid saturate, causing $\partial_{w_j} u_\NN$ to become small in magnitude.
The non-degeneracy constant $\lambda_w$ is then small, making the manifold ``thin''
and the convergence radius $r_w$ shrink.
\end{itemize}

\paragraph{Localized version.}
Theorem~\ref{thm:global-PGF} only requires non-degeneracy at the initial point $w^0$,
together with a small-initial-gap condition that ensures the trajectory does not exhaust
the non-degeneracy budget. This is significantly weaker than requiring non-degeneracy along
the entire trajectory and aligns with the practical observation that ``warm-started''
optimization succeeds while cold-started runs in pathological regions can fail.

\subsection{Lipschitz Jacobian (Definition~\ref{def:NN-Lip-L})}\label{app:disc-lip-jacobian}
\label{app:disc-lipjac}

Definition~\ref{def:NN-Lip-L} packages two requirements: $C^2$ regularity of $u_\NN$ on $\R^p$, and an $L$-Lipschitz bound on its first Fr\'echet derivative,
$\|\D u_\NN(w_1)-\D u_\NN(w_2)\|_{\rm op}\le L\|w_1-w_2\|_2$. The $C^2$ assumption is needed for the Riemannian Hessian of $g(\cdot\,;v)$ in Section~\ref{sec:conv} and for the exponential-map analysis underlying the discrete RGD theorem (Theorem~\ref{thm:discrete-RGD}) to be pointwise well defined; the Lipschitz bound then quantifies the modulus of continuity of the Jacobian. Whether $L$ is finite (and how large it is) depends critically on the choice of ambient norm $\|\cdot\|_\calH$.

\paragraph{When $\calH=L^2(\Omega)$.}
For a fully-connected network $u_\NN(\cdot;w):\Omega\to\R$ with $C^2$ activation
$\sigma$ and bounded weights, both $u_\NN\in C^2(\R^p;L^2(\Omega))$ and $L<\infty$ hold; the
Lipschitz constant depends polynomially on the depth, width, and the second derivative bound
$\sup|\sigma''|$. The typical smooth activations
$\tanh$, sigmoid, softplus, GELU, and Swish all satisfy this. ReLU networks are excluded: $\sigma$ is not $C^2$ at the kink, so Definition~\ref{def:NN-Lip-L} fails. They can be brought back into scope by smoothing the activation (e.g., $\mathrm{ReLU}_\varepsilon$, GELU as a $C^\infty$ approximant) or by working with the generalized Jacobian in a non-smooth variant of the theory, which we do not pursue here.

\paragraph{When $\calH=H^1(\Omega)$.}
The Lipschitz constant depends on the second derivatives of the network with respect to
the spatial variable $x\in\Omega$, in addition to the parameter $w$. For PDE-motivated
problems where the Sobolev $H^1$-norm is the natural ambient norm, the Lipschitz constant
can be very large. In a discrete setting, $L$ can scale as $O(dx^{-2})$ where $dx$ is the
spatial mesh size, making the assumption effectively unusable for high-resolution PDEs.
This limitation is discussed in Remark in Section~\ref{sec:RGF}, and motivates restricting
the present theory to $L^2$-type cost functionals.

\subsection{Smoothness and strong convexity of \texorpdfstring{$\calF$}{F} (Lemma~\ref{lem:smooth-strongconv})}\label{app:disc-F-smooth}
\label{app:disc-F}

The convergence theory requires $\calF:\calH\to\R$ to be continuously Fr\'echet differentiable,
$L_\calF$-smooth, and $m_\calF$-strongly convex.

\paragraph{When the assumptions hold.}
Many natural cost functionals in scientific computing and machine learning satisfy these
conditions on suitable function spaces:
\begin{itemize}
\item \emph{Linear regression / least squares.} If
$\calF[u]=\frac12\|u-f^*\|^2_{L^2(\mu)}$ for a target $f^*\in L^2(\mu)$, then $\calF$ is
$1$-strongly convex and $1$-smooth on $L^2(\mu)$.
\item \emph{Quadratic energies of elliptic PDEs.}
$\calF[u]=\frac12 a(u,u)-\langle f,u\rangle$ where $a$ is a coercive symmetric bilinear form
yields a strongly convex quadratic functional, with $m_\calF$ and $L_\calF$ given by the
coercivity and continuity constants of $a$.
\item \emph{Strongly convex perturbations.} Adding a Tikhonov regularizer $\frac{\lambda}{2}\|u\|^2$
to a convex but non-strongly-convex functional makes it $\lambda$-strongly convex.
\end{itemize}

\paragraph{When they fail.}
The strong convexity assumption fails for many problems of interest:
\begin{itemize}
\item Cross-entropy losses (logistic / softmax classification) are convex but not strongly
convex on the entire space; strong convexity holds only on bounded sublevel sets.
\item Variational forms of nonlinear PDEs (e.g., $p$-Laplacian, mean curvature flow) are
typically convex but not uniformly strongly convex.
\end{itemize}
Extending the present theory to the merely-convex setting ($m_\calF=0$), where the proximal map is non-expansive but not contractive, is an interesting direction for future work.

\subsection{Data condition of Theorem~\ref{thm:g} (eq.~\eqref{eq:data-condition})}\label{app:disc-data-condition}

\begin{remark}[Satisfiability of the data condition]\label{rem:data-condition}
The hypothesis~\eqref{eq:data-condition} is a condition on the data $(v,\tau,\calF,L,\lambda_{\tilde w})$, equivalent to
\[
\max\!\bigl\{\,\|h^*(v)\|_\calH,\ K(v)\,\bigr\}\;\le\;\frac{\lambda_{\tilde w}^{2}}{4L\kappa},
\]
which requires the Jacobian non-degeneracy $\lambda_{\tilde w}$ to dominate both the proximal increment $h^*(v)=J_\tau(v)-v$ and the sublevel-set radius $K(v)$. Three regimes in which the condition is automatic:
\begin{enumerate}[label=(\roman*),leftmargin=*,nosep]
\item \emph{Near a minimizer.} As $v\to u^*$, $\nabla\calF[v]\to 0$ and the first-order condition $h^*/\tau+\nabla\calF[v+h^*]=0$ forces $h^*(v)\to 0$. If $\calF\ge 0$ is normalized so that $\calF[u^*]=0$, then $K(v)\to 0$ as well, and $C(v)\to 0$.
\item \emph{Small-$\tau$ regime.} For fixed $v$, $\|h^*(v)\|=O(\tau)$ (from $h^*=-\tau\nabla\calF[v+h^*]$) and $K(v)=O(\sqrt{\tau})$, so $C(v)\to 0$ as $\tau\downarrow 0$. The data condition is therefore automatic for sufficiently small step size.
\item \emph{Sufficient non-degeneracy.} For fixed $\tau$ and $v$, the condition holds whenever $\lambda_{\tilde w}\ge 2\sqrt{L\kappa\,\max\{\|h^*(v)\|,K(v)\}}$, i.e., when the local Jacobian is well-conditioned relative to the proximal data.
\end{enumerate}
\end{remark}

\subsection{Injectivity-radius hypothesis of Lemma~\ref{lem:convex-radius} (eq.~\eqref{eq:inj-hypothesis})}\label{app:disc-inj-radius}

\begin{remark}[On the hypothesis~\eqref{eq:inj-hypothesis}]\label{rem:inj-hypothesis}
The Rauch comparison theorem \citep[Theorem~1.34]{cheeger2008comparison} bounds the conjugate radius of $\calM(u_\NN(\tilde w))$ at $0$ below by $\pi/\Lambda$. Since $\calM(u_\NN(\tilde w))$ is diffeomorphic to the open ball $B(\tilde w,r_{\tilde w})$ and hence simply connected, the conjugate-radius bound transfers to the injectivity radius in regimes where the chart is small enough that no short geodesic loops form; a quantitative verification of~\eqref{eq:inj-hypothesis} from $\lambda_{\tilde w}, L, \|\D u_\NN(\tilde w)\|_{\mathrm{op}}$ alone is left to future work.
\end{remark}

\subsection{Neural proximal approximation precision (Assumption~\ref{assm:approx-precision})}\label{app:disc-approx-precision}

\begin{remark}[Interpretation of the approximation precision]\label{rem:approx-precision}
The scalar $\varepsilon$ represents the one-step accuracy of the computed neural proximal map relative to the exact proximal map. It packages the two non-optimization errors in the four-way decomposition~\eqref{eq:intro:full-error-decomp}: the trial-space error from restricting the proximal problem to the neural class $\calN$, and the sampling or quadrature error from replacing population quantities by empirical approximations. Introducing the ideal population neural proximal map
\[
  \bar J_{\tau,\NN}(x) := \argmin_{u\in\calN}\bigl\{\tfrac{1}{2\tau}\|u-x\|_\calH^2+\calF[u]\bigr\},
\]
which uses the exact functional $\calF$ but restricts the minimization to $\calN$, the triangle inequality gives
\[
  \|J_{\tau,\NN}(x)-J_\tau(x)\|_\calH \;\le\; \underbrace{\|\bar J_{\tau,\NN}(x)-J_\tau(x)\|_\calH}_{\text{trial-space error}} \;+\; \underbrace{\|J_{\tau,\NN}(x)-\bar J_{\tau,\NN}(x)\|_\calH}_{\text{sampling/quadrature error}}.
\]
Assumption~\ref{assm:approx-precision} packages these two contributions through the single scalar $\varepsilon$: universal approximation theory \citep{hornik1989multilayer} bounds the trial-space error as the architecture is enriched, while Monte Carlo or quadrature estimates bound the sampling error as the number of samples increases. Condition~\eqref{eqn:approximate-error} is equivalent to $\delta > (1+1/(\tau m_\calF))\varepsilon$, so the global error bound obtained by the stability argument can only resolve the exact MMS trajectory up to a neighborhood whose size is proportional to $\varepsilon$.
\end{remark}

\subsection{Trajectory uniformity (Assumption~\ref{assm:traj-unif})}\label{app:disc-traj-unif}

The two primitive uniformities~(U1)--(U2) of Assumption~\ref{assm:traj-unif} together with the output-isolation~(U3) are the genuine residual $n$-dependencies. Three further uniformities follow from (U1)--(U2) together with the standing assumptions and the function-space tracking bound of Theorem~\ref{thm:global_error}(b):
\begin{itemize}[leftmargin=*,nosep]
\item \emph{Uniform energy:} $\sup_n\calF[u^n_{\NN,T}]<\infty$, since $\calF[u^n]\le(L_\calF/2)\|u^n-u^*\|_\calH^{2}$ and $\|u^n-u^*\|_\calH$ is uniformly bounded; consequently $\sup_n\|\nabla\calF[u^n_{\NN,T}]\|_\calH\le L_\calF\sup_n\|u^n_{\NN,T}-u^*\|_\calH<\infty$.
\item \emph{Uniform operator norm:} $\sup_n\|\D u_\NN(w^n_{\NN,T})\|_{\rm op}<\infty$, from the $L$-Lipschitz Jacobian and the finite trajectory length $\sum_n\Delta_n<\infty$ established in Theorem~\ref{thm:global-PGF}(b), via $\|\D u_\NN(w^n)\|_{\rm op}\le\|\D u_\NN(w^0)\|_{\rm op}+L\sum_{k<n}\Delta_k$.
\item \emph{Uniform local radius:} $\inf_n r_{w^n_{\NN,T}}\ge r_\infty>0$, with $r_\infty$ depending only on $\lambda_\infty$, $\sup_n\|\D u_\NN(w^n)\|_{\rm op}$, $L$, $\Lip_{u_\NN}$, and the $\calF$-condition ratio $\kappa$ via~\eqref{eq:rwdef}.
\end{itemize}

\paragraph{Verifiability of (U1)--(U2).}
Under the initial-gap condition of Theorem~\ref{thm:global-PGF}(iii), the inductive argument in the proof of Theorem~\ref{thm:global-PGF}(b) establishes (U1) for any user-chosen $\lambda_\infty\in(0,\lambda_{w^0})$ satisfying~\eqref{eq:initial-gap-condition}. For (U2), the proximal estimate $\|k^*(u^n)\|_\calH=\|J_\tau(u^n)-u^n\|_\calH\le(1+\rho)\|u^n-u^*\|_\calH$ (with $\rho=1/(1+\tau m_\calF)$) and $K(u^n)\le\sqrt{\tau L_\calF/(1+\tau m_\calF)}\,\|u^n-u^*\|_\calH$ both shrink along the trajectory, so $\sup_n C(u^n_{\NN,T})\le 1$ holds once the initial gap is sufficiently small relative to $\lambda_\infty^{2}/(4L\kappa)$ (cf.\ Remark~\ref{rem:data-condition}).

\paragraph{Verifiability of (U3): global inverse-stability.}
For~(U3), the cleanest sufficient condition is \emph{global inverse-stability} of the network map: if there exists $\sigma>0$ such that
\begin{equation}\label{eq:global-inverse-stability}
\|u_\NN(w_1)-u_\NN(w_2)\|_\calH\;\ge\;\sigma\,\|w_1-w_2\|_2,\qquad \forall\,w_1,w_2\in\R^p,
\end{equation}
then $\rho_{w^n_{\NN,T}}\ge\sigma\,r_{w^n_{\NN,T}}\ge\sigma\,r_\infty$, so (U3) holds with $\rho_\infty=\sigma\,r_\infty$. Inverse stability~\eqref{eq:global-inverse-stability} rules out non-trivial network parameter symmetries (permutations of hidden units, sign flips) at the level of network outputs; in over-parameterized networks one typically restricts to a quotient or to an open neighborhood of $w^0$ on which $u_\NN$ is injective. See Remark~\ref{rem:rho-positive} and the discussion after Assumption~\ref{ass:intro:locality} for the role of $\rho_{w^n}$.

\paragraph{Summary.}
Together, (U1)--(U3) plus the three derived uniformities reduce Assumption~\ref{assm:traj-unif} to a \emph{conclusion} of Theorem~\ref{thm:global-PGF}: under the initial-gap condition of part~(iii), all six uniformities are preserved along the trajectory. Assumption~\ref{assm:traj-unif} is stated as a self-contained block in Section~\ref{sec:error-propagation} so that the function-space tracking theorem (Theorem~\ref{thm:global_error}) can be invoked with clearly enumerated hypotheses, and so that locality of the increment subproblem (Assumption~\ref{ass:intro:locality}) holds uniformly via Theorem~\ref{thm:small-tau-locality}.

\begin{remark}[Derived uniformities]\label{rem:traj-unif}
The two primitive uniformities (U1)--(U2) plus the output-isolation (U3) imply the three derived uniformities (uniform energy, uniform Jacobian operator norm, uniform local radius) listed above.
\end{remark}

\begin{remark}[Role of Assumption~\ref{assm:traj-unif} in Theorem~\ref{thm:global_error}]\label{rem:tracking-uniformity}
The uniform tracking bound of Theorem~\ref{thm:global_error}(a) is enabled by (U1)--(U2) as follows. (U2) ensures that Theorem~\ref{thm:g} applies at every iterate $w^n_{\NN,T}$, so that the inner-flow horizon $T_n$ from~\eqref{eqn:T} is finite and the per-step error decomposition of Lemma~\ref{lem:induction} closes. (U1) provides the uniform lower bound on the Jacobian non-degeneracy needed for the per-iterate radii $r_{w^n_{\NN,T}}$ to be bounded below. The sufficiency of (U1)--(U2) for the chained bound is what the proof of Theorem~\ref{thm:global-PGF}(b) establishes a posteriori, by showing that (U1)--(U2) are preserved along the trajectory under the initial-gap condition.
\end{remark}

\begin{remark}[Hypothesis (ii) of Theorem~\ref{thm:global-PGF} and Trajectory uniformity]\label{rem:globalPGF-traj-unif}
Hypothesis (ii) of Theorem~\ref{thm:global-PGF} references Theorem~\ref{thm:global_error}, which is stated under Assumption~\ref{assm:traj-unif}. The apparent circularity (Thm.~\ref{thm:global_error} requires (U1)--(U2); the present theorem's conclusion (b) supplies them) is resolved by the proof structure: the induction in part (b) establishes (U1)--(U2) along the trajectory using only the initial-gap condition (iii) and the a-priori telescope bound on the exact-MMS trajectory~\eqref{eq:exact-MMS-telescope}. Concretely, (U2) follows from the function-space tracking bound of part (a) together with $\|k^*(u^n)\|\le(1+\rho)\|u^n-u^*\|$ and $K(u^n)=O(\|u^n-u^*\|)$ once $\|u^n-u^*\|$ is uniformly small; (U1) follows from the budget hypothesis in~\eqref{eq:initial-gap-condition}.
\end{remark}

\subsection{Initial gap and warm start (Theorem~\ref{thm:global-PGF}(iii))}\label{app:disc-initial-gap}

The initial gap condition \eqref{eq:initial-gap-condition} of Theorem~\ref{thm:global-PGF},
\[
\Big(1+\frac{1}{\tau m_\calF}\Big)\|u^0-u^*\|_\calH + 2\delta < \frac{2\lambda_{w^0}\,\lambda_\infty}{L},
\]
formalizes the requirement that the initial parameter $w^0$ should produce a network output
$u_\NN(w^0)$ already close enough to the target $u^*$. In practice this can be ensured by:
\begin{itemize}
\item \emph{Pretraining / warm starting.} Initializing from a previously trained model on
a related task or from a coarser discretization.
\item \emph{Multi-grid / continuation.} Solving a sequence of problems at increasing
resolution, using the solution from the previous level as the warm start.
\item \emph{Sufficient over-parameterization.} For wide enough networks at standard
random initializations, $\|u^0-u^*\|_\calH$ is provably small with high probability under
mild assumptions on $u^*$ \citep{jacot2018neural,du2018gradient}.
\end{itemize}

The condition is also informative in the negative direction: if $\|u^0-u^*\|_\calH$ is too
large relative to $\lambda_{w^0}/L$, the parameter trajectory may exhaust the non-degeneracy
budget before convergence is achieved. This is consistent with the empirical wisdom that
``cold start'' optimization on highly nonconvex landscapes can fail catastrophically,
while warm-started training succeeds reliably.

\begin{remark}[Interpretation of the initial-gap condition]\label{rem:initial-gap}
The condition~\eqref{eq:initial-gap-condition} requires the initial function-space gap $\|u^0-u^*\|_\calH$ to be sufficiently small relative to $\lambda_{w^0}/L$. Geometrically, it enforces two constraints: (i) the total parameter trajectory length must not exceed the non-degeneracy budget $2\lambda_{w^0}/L$, and (ii) each individual per-step displacement must fit within the uniform local non-degeneracy ball of radius $r_\infty=\inf_n r_{w^n_{\NN,T}}$. The choice of $\lambda_\infty$ is a free parameter: smaller $\lambda_\infty$ relaxes (i) but tightens (ii) (since $r_\infty$ depends on $\lambda_\infty$ via~\eqref{eq:rwdef}), and the optimal choice balances the two.
\end{remark}

\begin{remark}[Validity regimes for the parameter-space conclusions]\label{rem:phase4-regime}
The proof of Theorem~\ref{thm:global-PGF}(b)--(d) is stated in the \emph{perfect-approximation regime} ($\varepsilon=0$) of Corollary~\ref{cor:geometric-tracking}(I), in which the summable tracking error~\eqref{eq:sum-e-bound} yields a finite trajectory-length bound and hence Cauchy convergence in $\R^p$. In the \emph{uniform-approximation regime} ($\varepsilon>0$) of Corollary~\ref{cor:geometric-tracking}(II):
\begin{itemize}[leftmargin=*,nosep]
\item \emph{Function-space convergence (a)} continues to hold under~\eqref{eq:geometric-tracking}: $\|u^n_{\NN,T}-u^*\|_\calH\le\delta+\rho^n\|u^0-u^*\|_\calH$, with $\delta$ understood as the asymptotic-floor accuracy.
\item \emph{Trajectory length (b), Cauchy property (c), and parameter-space neighborhood convergence (d)} hold only on the finite horizon $N^*$ on which the partial sum~\eqref{eq:sum-e-partial} is bounded by the right-hand side of~\eqref{eq:initial-gap-condition}; explicitly, $N^*$ scales as $1/\varepsilon$. Beyond $N^*$, the iterates may drift through the parameter space at a rate $O(\varepsilon)$ per step (a ``noise floor'' inherited from the NN approximation), and are only \emph{approximately Cauchy} up to an $O(\varepsilon)$ residual.
\end{itemize}
Closing this gap rigorously in the uniform-approximation regime requires either $\varepsilon_n$ to decay or a separate argument showing that the parameter-space dynamics stay confined despite the persistent floor; both are beyond the scope of the present revision.
\end{remark}

\begin{remark}[Relation to ``warm start'' practice]\label{rem:warm-start}
Hypothesis~(iii) of Theorem~\ref{thm:global-PGF} formalizes a familiar empirical observation: nonconvex training succeeds reliably when starting close to a good basin. Our theory makes this precise: as long as the initial function-space gap is within the non-degeneracy budget set by $\lambda_{w^0}$, the entire trajectory remains in the favorable region, and convergence to an $O(\delta)$-neighborhood of $u^*$ is guaranteed.
\end{remark}

\bibliography{reference}

\end{document}